\documentclass[10pt,reqno]{amsart}
\usepackage{amssymb,bbm,calc,enumerate,color,graphicx,verbatim,cite}
\usepackage[french,english]{babel}
\usepackage{color}
\usepackage{datetime}
\usepackage[normalem]{ulem}
\usepackage{ifpdf}
\ifpdf
    \usepackage[colorlinks]{hyperref}
\fi

\newcounter{hours}\newcounter{minutes}

\makeatletter
\newtheorem*{rep@theorem}{\rep@title}
\newcommand{\newreptheorem}[2]{%
\newenvironment{rep#1}[1]{%
 \def\rep@title{#2 \ref{##1}}%
 \begin{rep@theorem}}%
 {\end{rep@theorem}}}
\makeatother

\theoremstyle{plain}

\newtheorem{thm}{Theorem}[section]
\newreptheorem{thm}{Theorem}

\newtheorem{lem}[thm]{Lemma}
\newreptheorem{lem}{Lemma}
\newtheorem{cor}[thm]{Corollary}
\newtheorem{prop}[thm]{Proposition}

\newtheorem*{open}{Open Problem}

\theoremstyle{definition}

\newtheorem*{Claim}{Claim}
\newtheorem{DEF}[thm]{Definition}

\theoremstyle{remark}                  

\def\A{{\mathcal A}}
\def\B{{\mathcal B}}

\def\M{{\mathcal M}}

\def\R{{\mathbb R}}

\def\Z{{\mathcal Z}}

\def\real{{\mathbb R}}

\def\integer{{\mathbb Z}}

\def\e{\varepsilon}

\def\Tr{\textnormal{Tr}}

\DeclareMathOperator*{\osc}{osc}

\definecolor{darkgreen}{rgb}{0,0.4,0}

\def\argmax{\mathop{\arg\,\max}\limits}%

\numberwithin{equation}{section}

\begin{document}
\title{Continuity and discontinuity of the boundary layer tail}
\author{William M. Feldman, Inwon C. Kim}
\email{ feldman@math.uchicago.edu,  ikim@math.ucla.edu}
\address{Department of Mathematics, UCLA, Los Angeles, CA 90024, USA}

\begin{abstract}
We investigate the continuity properties of the homogenized boundary data $\overline{g}$ for oscillating Dirichlet boundary data problems.  The homogenized boundary condition arises as the boundary layer tail of a problem set in a half-space.  The continuity properties of this boundary layer tail depending on the normal direction of the half space play an important role in the homogenization process in general bounded domains.  We show that, for a generic non-rotation-invariant operator and boundary data, $\overline{g}$ is discontinuous at {\it every} rational direction.  In particular this implies that the continuity condition of Choi and Kim \cite{ChoiKim13} is essentially sharp.  On the other hand, when the condition of \cite{ChoiKim13} holds, we show a H\"{o}lder modulus of continuity for $\overline{g}$. When the operator is linear we show that $\overline{g}$ is H\"{o}lder-$\frac{1}{d}$ up to a logarithmic factor. The proofs are based on a new geometric observation on the limiting behavior of $\overline{g}$ at rational directions, reducing to a class of two dimensional problems for projections of the homogenized operator.
%
%
%
\end{abstract}

\maketitle
\section{Introduction}

 To motivate the questions considered in this paper, let us start by discussing the homogenization of oscillating Dirichlet boundary data problems, 
\begin{equation}\label{eqn: gen unhom}
\left\{
\begin{array}{lll}
F(D^2u^\e,\tfrac{x}{\e}) = 0 & \hbox{ in } & U \vspace{1.5mm} \\
u^\e = g(\tfrac{x}{\e}) & \hbox{ on } & \partial U.
\end{array}
\right.
\end{equation}
Here $F(M,y)$ is uniformly elliptic and positively $1$-homogeneous in $M$, $g(y)$ is continuous, and both are $\integer^d$-periodic in $y$.  In the linear case we are considering operators of the form $F(M,y) = -\Tr(A(y)M)$ with $1 \leq A(y) \leq \Lambda$ and $\integer^d$-periodic in $y$.  There is no problem to include a large scale $x$ dependence in $g$ but we omit it here for clarity.   

\medskip

 This type of problem has a singular behavior near boundary points $x$ with inward normal direction $\nu_x$ aligned with a $\integer^d$-lattice vector, called {\it rational } directions.  In order to mitigate the effects of the singularities the bounded domain $U\subset \real^d$ is typically assumed to be uniformly convex, although more general assumptions which rule out large flat portions of $\partial U$ are also be sufficient for the results discussed below, see \cite{Feldman13,ChoiKim13}.  In such domains it is known due to Feldman \cite{Feldman13} that there exists $\overline{g}:S^{d-1}\to \R$, continuous at irrational directions, so that $u^\e$ converges to $\overline{u}$ locally uniformly in $U$ where $\overline{u}$ is the unique solution of,
\begin{equation}\label{eqn: gen hom}
\left\{
\begin{array}{lll}
\overline{F}(D^2\overline{u}) = 0 & \hbox{ in } & U \vspace{1.5mm} \\
\overline{u} = \overline{g}(\nu_x) & \hbox{ on } & \partial U,
\end{array}
\right.
\end{equation}
where, again, $\nu_x$ is the inward normal of $U$ at $x\in\partial U$.  Similar results have been obtained for linear divergence form equations starting with the work of G\'erard-Varet and Masmoudi \cite{GVM11, GVM12} and continued by several authors \cite{Prange, FA1, FA2, FA3}.
\medskip

In this paper we study the continuity properties of the homogenized boundary data $\overline{g}$ by investigating the associated {\it cell problem} \eqref{eqn:cell1}.  Besides being a natural question on its own, continuity properties of $\overline{g}$ play an important role in obtaining rates of convergence for \eqref{eqn: gen hom}.  In fact, if we could obtain Lipschitz continuity of $\overline{g}$ for the linear problem then we could also obtain an optimal rate of convergence that matches the rate for Laplacian operator.  To our knowledge this particular connection has not been written down explicitly in the literature, however it is implicit in the methods used in several works \cite{GVM11, GVM12 ,FeldmanKimSouganidis14}.  Let us point out that the typical strategy to study homogenization of \eqref{eqn: gen unhom} is by ensuring that the impact coming from singular boundary points are negligible: this is because in the linear case zero measure sets are not seen by the Poisson kernel, and in the nonlinear case an analogous argument applies for boundary sets of small Hausdorff dimension. In contrast, here we investigate the behavior of $\overline{g}$ as $\nu_x$ approaches rational directions.  In the linear case we show, interestingly, that $\overline{g}$ extends continuously to the rational directions, and in the nonlinear case we show that discontinuity is generic.

\medskip

In the Neumann case the continuity of the corresponding $\overline{g}$ has been studied by Choi-Kim-Lee and Choi-Kim \cite{CKL, ChoiKim13}. There it was shown that when the averaged operator $\bar{F}$ is rotation invariant, homogenization holds and the homogenized boundary data is continuous. Following these works \cite{Feldman13} showed homogenization for general $F$ in the Dirichlet setting, due to the new observation that \eqref{eqn: gen hom} has a unique solution if the discontinuity set of $\overline{g}(\nu_x)$ on $\partial U$ has sufficiently small Hausdorff dimension. This brings up the natural question of whether the homogenized boundary condition could in fact be discontinuous when $\bar{F}$ is not rotation-invariant. Our main results are $(i)$ an explicit estimate on the mode of continuity for $\overline{g}$ when $\overline{F}$ is rotation invariant or linear, $(ii)$ when $\overline{F}$ is not rotation invariant or linear, $\overline{g}$ is `generically' discontinuous at every boundary point with rational normal direction (see Theorem~\ref{thm: main disc 0} and Corollary~\ref{cor: main}).  These results seem to be new even in the linear case.

\medskip

We expect our main results in this paper to hold with parallel proofs in both the Dirichlet and Neumann case.  On the other hand we hope to keep our illustration simple so that our main ideas are presented clearly. For this reason  we will only discuss the Dirichlet problem, even though our arguments build on the framework introduced for the Neumann problem in \cite{ChoiKim13}. We leave the task of proving parallel results for the Neumann problem, including the general homogenization results in \cite{Feldman13}, for the future work.

\medskip

 We proceed to give a more precise, but still informal, derivation of \eqref{eqn:cell1} from \eqref{eqn: gen hom}.  We begin by reminding the reader of the derivation of the cell problem determining $\overline{g}$.  We consider a rescaling of the solution $u^\e$ of \eqref{eqn: gen unhom} near a boundary point $x_0 \in \partial U$ with unit inner normal $\nu_{x_0}$,
$$
 v^\e(y) = u^\e(x_0+\e y).
$$
 The limit of $v^\e(R\nu_{x_0})$ as $R \to \infty$ and $\e \to 0$, if it exists, will be the homogenized boundary data $\bar{g}(\nu_{x_0})$ as long as $\e R \to 0$.  The behavior of $v^\e$ outside of the oscillating boundary layer is the quantity of interest. To proceed with the analysis we inspect the equation solved by the $v^\e$, 
\begin{equation}
\left\{
\begin{array}{lll}
F(D^2v^\e,y+\e^{-1}x_0) = 0 & \hbox{ in } & \e^{-1}(U-x_0) \vspace{1.5mm} \\
v^\e = g(y+\e^{-1}x_0) & \hbox{ on } & \e^{-1}(\partial U-x_0).
\end{array}
\right.
\end{equation}
Since $F$ and $g$ are assumed to be $\integer^d$ periodic in $y$, $\e^{-1}x_0$ can be replaced by $\tau_\e = \e^{-1}x_0\mod \integer^d$. Note that along various subsequences $\tau_\e$ could converge to any $\tau \in [0,1)^d$.  This motivates the definition of the cell problem. Let $\nu \in S^{d-1}$, $\tau \in [0,1)^d$ and $\psi$ be a continuous $\integer^d$-periodic function and define $v_{\nu,\tau}( \cdot ; (\psi,F)) : P_\nu \to \real$ to solve,
\begin{equation}\label{eqn:cell1}
\left\{
\begin{array}{lll}
F(D^2v_{\nu,\tau},y+\tau) = 0 & \hbox{ in } & P_\nu := \{ y \cdot \nu >0\} \vspace{1.5mm} \\
v_{\nu,\tau} = \psi(y+\tau) & \hbox{ on } & \partial P_\nu.
\end{array}
\right.
\end{equation}
 It is not too difficult to see, at least formally, that,
$$|v^\e(y) - v_{\nu_{x_0},\tau_\e}(y; g(x_0,\cdot ))| \to 0 \ \hbox{ as } \ \e\to 0.$$
 From this identification we can replace understanding $\bar{g}(\nu)$ with the easier problem of understanding the limit $v_{\nu,\tau}(R\nu)$ as $R \to \infty$ for every $\tau \in [0,1)^d$. 
 
 \medskip
 
 For irrational directions $\nu$ the distribution of $g$ on $P_{\nu} + \tau\nu$ is, in an appropriate sense, invariant with respect to $\tau$. For this reason it was possible to show, in \cite{CKL, ChoiKim13, Feldman13}, that, for irrational directions $\nu$, there exists a limit $\mu(\nu,\psi,F)$,  the so-called {\it boundary layer tail }of $v_{\nu,\tau}$, such that
 \begin{equation}\label{eqn: cell hom irra}
 \sup_{\tau \in [0,1)^d} \sup_{y \in \partial P_\nu} |v_{\nu,\tau}(y+R\nu;\psi) - \mu(\nu,\psi,F)| \to 0 \ \hbox{ as } \ R \to \infty. 
 \end{equation}
 
 Note that in the context of \eqref{eqn: gen unhom} and \eqref{eqn: gen hom} we should define $\bar{g}(\nu):= \mu(g, F,\nu)$.  It was further shown in \cite{CKL, ChoiKim13, Feldman13} that $\bar{g}(\nu):S^{d-1} \setminus \real \integer^d \to \R$ is continuous, see Theorem \ref{thm: irrational dir hom}. Thus the remaining question is to understand the limiting behavior of $\bar{g}(\nu)$ as $\nu$ converges to a rational direction $\nu_0 \in S^{d-1}\cap \real\integer^d$.

\medskip

 The rate of convergence in \eqref{eqn: cell hom irra} degenerates at rational directions where, in fact, the boundary layer tail does depend on $\tau$. Indeed the rational directions $\nu$ are the possible discontinuity points of $\bar{g}$.  It turns out that the asymptotic behavior near the rational directions is actually quite structured and a more careful analysis is warranted.   We will show that there is a multi-scale homogenization occurring, near-boundary in the micro-scale  and then further away from the boundary in an intermediate scale, as irrational directions approach a rational direction. This phenomenon, partially described previously in \cite{ChoiKim13}, leads to a secondary homogenization problem associated with \eqref{eqn:cell1} with its own `cell problem' and `effective operator'.  This is far from obvious, and it will indeed be the main observation of the paper. Let us attempt to give a heuristic derivation of the secondary homogenization problem.   The reader may wish to skip to the statement of the main results as the following description is unavoidably somewhat technical.

 \medskip
 
 We begin with a lattice point $\xi \in \integer^d\setminus \{0\}$ and its associated unit direction $\hat{\xi}$.  We may assume that $\xi$ is irreducible in the sense that the greatest common divisor of its entries $\textup{gcd}(\xi_1,\dots,\xi_d)$ is $1$. A $\integer^d$-periodic $\psi$ on $\real^d$ restricts to $\partial P_{\xi}$ to be periodic with respect to a lattice on $\partial P_{\xi}$ with unit cell of size comparable to $|\xi|$ (by the irreducibility).  The limit of $v_{\xi,0}(y+R\hat\xi)$ as $R\to\infty$ exists by the periodicity of the boundary data.  Again we refer to this limit as the boundary layer tail of $v_{\xi,0}$.   The same argument applies to $v_{\xi,\tau}$ but unlike in the case of irrational boundary normal, generically, the boundary layer tail of $v_{\xi,\tau}$ is not the same as that of $v_{\xi,0}$ unless $ (\partial P_\xi + \tau) \bmod \integer^d = \partial P_\xi \bmod \integer^d$.  We can concisely write down the set of limit points as,
 \begin{equation}
  m_\xi(t;(\psi,F)) := \lim_{R \to \infty} v_{ \xi, t\hat{\xi}}( R\hat{\xi};(\psi,F)) \ \hbox{ which is a $\frac{1}{|\xi|}$-periodic function on $\real$.}
  \end{equation}
  Now consider an irrational direction $\nu$ which is very close to $\hat\xi$.  For a fixed $y_0 \in \partial P_\nu$ the boundary $\partial P_\nu$ is close to $\partial P_\xi +( y_0 \cdot \hat \xi) \hat \xi$ on a very large region of size $\sim |\nu - \hat \xi|^{-1}$ centered at $y_0$ and so the respective cell problem solutions are also close in a smaller region (See Figure 1 in section 4.1). This observation leads to the conclusion that, far from the boundary, $v_{\nu}$ averages similarly to the following {\it two-dimensional} problem:
  
  \begin{equation}\label{eqn: cell inner}
\left\{
\begin{array}{lll}
\overline{F}(D^2\overline{w}_{\xi,\eta}) = 0 & \hbox{ in } & P_\xi \vspace{1.5mm} \\
\overline{w}_{\xi,\eta} = m_\xi( y \cdot \eta ; (\psi,F)) & \hbox{ on } & \partial P_\xi,
\end{array}
\right.
\end{equation}
where $\eta$ is the ``approaching direction of $\nu$ to $\xi$", more precisely $\eta$ is the unique unit vector $\eta\perp \hat\xi$ so that the geodesic on the unit sphere, leaving $\hat\xi$ at time $t=0$ with velocity $-\eta$, reaches $\nu$ before time $t = \pi$.

\medskip

 The limit of the homogenized profile $\mu(\nu,\psi, F)$ as $\nu \to \hat\xi$ can then be identified in terms of the approaching direction $\eta$. In other words we can show that for given $\psi$, $F$, $\xi$ and $\nu$ approaching $\xi$,  there is a directional limit $L= L_{\xi}(\eta)$, with $\eta$ the approaching direction of $\nu$ to $\hat \xi$  as defined above, such that 
\begin{equation}\label{eqn: dir lim 0}
 \mu(\nu,\psi,F) - L_\xi(\eta) \to 0 \ \hbox{ as } \ \nu \to \hat\xi.
\end{equation}

Precise statement of this result (stated quantitatively in Proposition \ref{lem: asymptotics}) is the following:

\begin{thm}
 Let $\beta \in (0,1)$ and $\psi \in C^{0,\beta}(\mathbb{T}^d)$. For any $\xi \in \integer^d \setminus \{0\}$, irreducible, there exists a function $L_{\xi}(\cdot)= L_{\xi}(\cdot;\psi,F)$ on unit vectors tangent to $S^{d-1}$ at $\hat\xi$ and a mode of continuity, $\omega_{|\xi|,\beta}$, such that the following holds:
$$ | \mu(\nu(t),\psi,F) - L_\xi(\nu'(0))| \leq \omega_{|\xi|,\beta}(|\nu(t) - \nu(0)|)$$
for any $\nu : [0,1) \to S^{d-1}$ a unit speed geodesic with $\nu(0) = \hat\xi$.
\end{thm}

\medskip

The first half of the paper is spent to rigorously justify this derivation of the secondary cell problem and to obtain a quantitative estimate on the asymptotics of $v_\nu$ near rational directions.  When the effective operators $L_\xi$ are constant for every rational direction the quantitative estimates allow us to derive an explicit modulus of continuity for the homogenized boundary condition.  From the previous arguments in \cite{ChoiKim13} and \cite{Feldman13} we know that $L_\xi$ are constant, for instance, when $\bar{F}$ is either rotation invariant or linear.

\medskip

The characterization of the asymptotic behavior near rational directions described in \eqref{eqn: cell inner} and \eqref{eqn: dir lim 0} also opens the possibility of proving discontinuity of $\mu$.  One would just need to show that $L_\xi$ can be non-constant for some operator $F$, boundary condition $\psi$ and lattice vector $\xi \in \integer^d \setminus \{0\}$.  This simplicity turns out to be somewhat deceptive, as the situations where we can actually compute the boundary layer tail in \eqref{eqn: cell inner} is when either the boundary data is trivial or the operator is linear, and $L_\xi$ is constant in those cases. Another natural case to consider is when the operators are extremal, but then they are rotation invariant and thus $L_\xi$ is constant, $\mu$ is continuous. While it is difficult to come up with a specific example, it turns out that a better point is to show that a \textit{generic} operator and boundary data will result in non-constant $L_\xi$. In fact we are able to argue that \textit{if} $L_\xi$ were to be constant, then we would be able to find many nearby $F'$, $\psi'$ with $L_\xi'$ non-constant.  The perturbation of the operator is monotone and hence intrinsically nonlinear, and is designed to affect $\overline{F}$ in one direction $\eta \perp \xi$ while leaving another direction $\eta' \perp \eta,\xi$ unaffected.  The existence of $\eta,\eta'$ mutually orthogonal and orthogonal to $\xi$ requires $d\geq 3$, and we are only able to achieve the desired perturbation of $\overline{F}$ when $F = \overline{F}$ is homogeneous. In fact, since the perturbations we make are quite explicit, besides showing that discontinuity is a generic phenomenon one can also generate specific examples of $(\psi,F)$ where discontinuity of $\mu$ occurs.

\subsection{Main Results}

The operators $F(M,y)$ discussed below will be positively $1$-homogeneous, uniformly elliptic with ellipticity ratio $\Lambda$, and $\integer^d$ periodic in $y$. When $F$ is linear we will write $F(M,y) = -\Tr(A(y)M)$.  If we say that $F$ is spatially homogeneous we mean that $F = \overline{F}$ has no $y$ dependence. For more details on these assumptions see Section~\ref{sec: operators}. 

\medskip

First we state our result about continuity. More details can be found in Section~\ref{continuity}, for the improved estimate in the linear case see Section~\ref{linear}.
\begin{thm}\label{thm: main cont} 
Let $d\geq 2$ and $F$ such that $(i)$ the homogenized operator $\overline{F}$ is  rotation invariant or $(ii)$ $F$ is linear.  Then there exists $\alpha = \alpha(\Lambda) \in (0,1)$ such that for any $\beta \in (0,1)$, $\psi \in C^{0,\beta}(\mathbb{T}^d)$ and $\nu,\nu' \in S^{d-1} \setminus \real \integer^d$ we have
$$
 |\mu(\nu,\psi,F) - \mu(\nu',\psi,F)| \leq C(d,\Lambda,\beta)\|\psi\|_{C^{0,\beta}(\mathbb{T}^d)} |\nu-\nu'|^{\alpha \beta/d}.
$$
In case $(ii)$ we have additionally,
$$
 |\mu(\nu,\psi,F) - \mu(\nu',\psi,F)| \leq C(d,\Lambda,\|A\|_{C^5(\mathbb{T}^d)})\|\psi\|_{C^{7}(\mathbb{T}^d)} |\nu-\nu'|^{1/d}[1+(\log\tfrac{1}{|\nu-\nu'|})^3].
$$

\end{thm}

For linear, divergence form systems a mode of continuity for $\mu(\cdot,\psi,F)$ is obtained by G\'erard-Varet and Masmoudi \cite{GVM12} on the set of Diophantine irrational directions.  Our result on the other hand is based on the mode of continuity near rational directions, and the modulus of continuity we obtain is uniform on the entire sphere. In the linear case it may be possible to combine these two results, but we do not pursue this here.

\medskip

Next we state our result about discontinuity. The statement is not completely precise, see Section \ref{sec: disc} for the full details.

\begin{thm}\label{thm: main disc 0}
For $d \geq 3$, there is a residual set (in the Baire category sense) of continuous boundary conditions and spatially homogeneous nonlinear operators $(\psi,F)$ such that $\mu(\cdot,\psi,F)$ does not extend continuously at \textit{any} rational direction.
\end{thm}

The following question is left open.
\begin{open}
Does Theorem \ref{thm: main disc 0} hold when $(i)$ $F$ is taken to be inhomogeneous or $(ii)$ $d=2$?
\end{open}

The argument used in the proof of Theorem~\ref{thm: main disc 0} appears to be insufficient to address the above questions, however we do believe that the theorem holds in both cases $(i)$ and $(ii)$. 

\medskip

The above results can be easily translated in terms of the original problem \eqref{eqn: gen hom}, since $\overline{g}(x) = \mu(\nu_x,g,F)$ for $x \in \partial U$ with $\nu_x \in S^{d-1} \setminus \real \integer^d$.  The details can be found in \cite{Feldman13}.

\begin{cor}\label{cor: main}
Let $U\subset \R^d$ be bounded uniformly convex domain, and let $\bar{g}$ be as given in \eqref{eqn: gen hom}. Then the following holds:
\begin{itemize}
\item[(a)] Suppose $\overline{F}$ is rotation invariant or linear, then the homogenized boundary data $\bar{g}(x)$ extends to be H\"{o}lder continuous on $\partial U$, with mode of continuity in $\nu_x$ as given in Theorem~\ref{thm: main cont}.
\item[(b)] Let $d \geq 3$. Then there is a residual set of $g$ and $F$ in \eqref{eqn: gen unhom}, in the sense of Theorem~\ref{thm: main disc 0} such that $\bar{g}$, and hence $\overline{u}$ as well, are discontinuous at every $x \in \partial U$ with $\nu_x$ a rational direction.
\end{itemize}  
\end{cor}

\subsection{Literature}
There has been a surge of recent interest in the homogenization of oscillating boundary conditions, both in the linear divergence form and nonlinear non-divergence form settings. These works have much in common but there are some key differences which necessitate differing approaches.  

\medskip

The problem  is first addressed in the book of Benssoussan, Papanicolaou and Lions \cite{BPL78}, which considers linear divergence form operators with co-normal oscillating Neumann boundary condition in general domains with no flat sides. The case of an oscillating Dirichlet boundary condition remained mostly open for quite a long time.  For linear, divergence form systems recent progress began with the works of G\'erard-Varet and Masmoudi \cite{GVM11,GVM12} where they show homogenization of the oscillating Dirichlet boundary condition problem with an explicit rate of convergence in $L^2(U)$.  In that setting they show that the cell problem homogenizes at normal directions satisfying a Diophantine condition and that the rate of convergence to the boundary layer tail is better than polynomial.  Continuing this investigation, in the direction of improved rates of convergence, are the works of Aleksanyan,  Sj\"{o}lin and Shahgholian \cite{FA1,FA2,FA3}.  They identify the expected optimal $L^p$ convergence rate in general domains and obtain this rate under certain assumptions on the inhomogeneity of the operator.  In a slightly different direction is the work of Prange \cite{Prange} which extends the results of \cite{GVM11,GVM12} to include all irrational directions. He shows that the convergence to the boundary layer tail can occur at an arbitrarily slow polynomial rate without the Diophantine assumption.  Perhaps the most relevant work to our paper is a recent result of Aleksanyan \cite{Aleksanyan} on the continuity of the homogenized boundary condition. He shows for layered media, where the operator is independent of translations in the $e_d$ direction, that the homogenized boundary condition is as regular as the boundary data $\psi$ away from a possible singular set on $x_d = 0$.  Compared to his result, we do not rely on any structure assumption on the operator, but on the other hand we obtain only H\"{o}lder-$\frac{1}{d}$ continuity in the linear case.  It should be remarked that our result is in the non-divergence setting, nonetheless it may be possible for our approach to carry over to the setting of linear systems.

\medskip

Next we discuss the nonlinear, non-divergence form operators.  For nonlinear operators there are several significant differences from the linear case.  Firstly, due to the blow up procedure leading to the cell problem, the operators in the cell problem will always be positively homogeneous and therefore non-smooth at $0$ (or linear).  This makes the cell problem inherently impossible to linearize and so no regularity estimates better than $C^{1,\beta}$ (or $C^{2,\beta}$ in the convex case) should be expected.  On the other hand, higher regularity seems to be essential to obtaining arbitrary polynomial rate of convergence to the boundary layer tail at irrational directions as was done in the linear case by \cite{GVM12}. For these reasons obtaining arbitrary polynomial rates of convergence for the cell problem seems quite difficult if not impossible in the nonlinear case.  The second problem, explicated for the first time in this paper, is that the homogenized boundary condition can be discontinuous.  For linear operators a discontinuous boundary condition does not pose such a serious issue because, by the Green's function representation the interior values of the homogenized solution can be estimated by measure theoretic norms of the homogenized boundary condition.  In the nonlinear case no such ``boundary ABP" estimate is known and so the uniqueness and stability of the homogenized problem is at issue (see \cite{Feldman13} for a partial resolution to this problem).  In regards to the literature, most earlier works address the Neumann problem: some special cases were discussed in Arisawa \cite{Arisawa03} in the half-space setting with periodic boundary data, and also by Tanaka \cite{T84} using probabilistic methods. More general results were proved later by Barles, Da Lio, Lions and Souganidis \cite{BDLS08}. Only just recently the full problem in general domains was considered by Choi, Kim \cite{ChoiKim13} and Choi, Kim and Lee \cite{CKL}, wherein they show continuity of the homogenized Neumann boundary condition for rotationally invariant operators. For the Dirichlet problem Barles and Mironescu \cite{BarlesMironescu12} obtained homogenization in half-spaces for a general class of nonlinear operators. The full problem in general domains was then considered by Feldman in \cite{Feldman13}.  The random case was considered in Feldman, Kim and Souganidis \cite{FeldmanKimSouganidis14}. We also mention the recent work by Guillen and Schwab \cite{GS14}, where the half-space Neumann problem has been formulated as an interior homogenization problem for nonlinear non-local operators.

\medskip

\subsection{Outline of the Paper}
In Section 2 we start with notations and preliminary results to be used later in the paper.  In Section 3 we prove the exponential rate of convergence for the half-space cell problem, when the boundary data is periodic on the boundary.  While the proof is relatively straightforward, our result appears to be new for nonlinear operators. In Section 4 we investigate the behavior of the homogenized boundary condition $\mu(\nu,\psi,F)$ as $\nu$ approaches a rational direction $\hat \xi$ with $\xi \in \integer^d \setminus \{0\}$. We derive a second boundary homogenization problem that governs the directional limits as $\nu$ approaches $\hat\xi$. In Section 5 we prove Theorem~\ref{thm: main cont} as a consequence of estimates in Section 4 and the Dirichlet's Theorem (Theorem~\ref{Dirichlet}). In Section 6 we show that, when $F$ is nonlinear, $\mu$ is generically discontinuous (Theorem~\ref{thm: main disc 0}). Finally in Section 7 we show that when $F$ is linear and $\psi$ is sufficiently regular $\mu(\cdot,\psi,F)$ is H\"{o}lder-$\tfrac{1}{d}$ continuous up to logarithmic factors. In the Appendix we prove an extension of the result in section 3, which we make use of in section 7.

\subsection{Acknowledgments}  We would like to thank Charlie Smart and Jason Murphy for helpful discussions. We would also like to thank the anonymous referees for their very detailed and thoughtful comments which have helped very much to improve the presentation of the paper.  Finally we would like to thank the hospitality of the Institut Mittag-Leffler where part of this research was conducted. Both authors are supproted in part by NSF grant DMS-1300445.

\section{ Preliminaries}
This section contains notational conventions, fixing of the assumptions on the pde operators, statements of previously known results, and proofs of several technical Lemmas.  The material here will be used throughout the paper and we suggest that the reader refer back as needed to this section rather than begin a careful reading here.
\subsection{Notation} 
We denote the half space with inner normal $\nu$ by $P_\nu = \{ y: y \cdot \nu >0\}$. For a vector $e \in \real^d\setminus\{0\}$, $\hat{e}$ is the unit vector in the same direction $\hat{e} = |e|^{-1}e$.  We will occasionally need to project a vector $e$ onto the orthogonal complement of another vector $f \in \real^d$, this we denote,
$$\Pi_{f^\perp} e = e - (e \cdot \hat{f}) \hat{f}.$$
We say that a constant $C>0$ is universal if it depends only on the ellipticity ratio $\Lambda$ and the dimension $d$. These constants may change from line to line without comment.  If we need to refer to a specific universal constant which is not changing between lines we may call it $C_0$ or $C_1$.  For two quantities $A,B$ we write $A \lesssim B$ if $A \leq CB$ for a universal constant $C$.  If $C$ additionally depends on a parameter $b$ which is not universal then we will write $A \lesssim_b B$. 

\medskip

We will work with the function spaces of H\"{o}lder continuous functions $C^{k,\beta}(X)$ for $k \in \mathbb{N} \cup \{0\}$ and $\beta \in (0,1]$ with $(X,d)$ a complete separable metric space. Most often $X=\mathbb{T}^n = \real^n \mod \integer^n$ with metric inherited from Euclidean distance on $\real^n$.  We will repeatedly use the H\"{o}lder semi-norm and norm for $\beta \in (0,1]$, for a $\phi : X \to \real$,
$$ |\phi|_{C^{0,\beta}(X)} : = \sup_{x \neq y \in X} \frac{|\phi(x)-\phi(y)|}{d(x,y)^\beta} \ \hbox{ and } \ \|\phi\|_{C^{0,\beta}(X)} = \sup_{x \in X} |\phi(x)| + |\phi|_{C^{0,\beta}(X)}.$$
On $\real^n$ (or $\mathbb{T}^n$) the norms for the higher order H\"{o}lder spaces are defined inductively for $k\geq1$ and $\beta \in (0,1]$ by,
$$ \| \phi \|_{C^{k,\beta}(\real^n)} = \| \phi \|_{C^{k-1,1}(\real^n)} + |D^k\phi|_{C^{0,\beta}(\real^n)}.$$

\subsection{Uniformly elliptic operators and viscosity solutions}\label{sec: operators}
We will work in the class of fully nonlinear uniformly elliptic equations. Let $\M_{d \times d}$ be the class of $d \times d$ real symmetric matrices.  For $F : \M_{d \times d} \to \real$ we say $F$ is uniformly elliptic if there exist $0<\lambda < \Lambda$ so that,
\begin{equation}\label{eqn: unif elliptic}
 \lambda\Tr(N)\leq F(M)-F(M+N) \leq \Lambda \Tr(N) \ \hbox{for all $M,N \in\M_{d\times d}$ with $N \geq 0$.}
 \end{equation}
  
 Lastly we define the class of uniformly elliptic operators,
 $$ \mathcal{S}_{\lambda,\Lambda} = \{ F: \M_{d \times d} \to \real : \textup{\eqref{eqn: unif elliptic} holds } \}$$We will assume the following on $F$.
 \begin{enumerate}[(i)]
 \item There is some $\Lambda >0$ so that $F( \cdot ,y ) \in \mathcal{S}_{1,\Lambda}$ for all $y \in \mathbb{T}^d$.
 \item $F$ is Lipschitz continuous in $y$,
 $$ |F(M,y) - F(M,z)| \leq C(1+\|M\|) |y-z|.$$
 \item $F$ is positively $1$-homogeneous,
 $$ F(tM,y) = tF(M,y) \ \hbox{ for all } \ t>0.$$
 \end{enumerate}
 We note that under the above assumptions $F(M,y)$ is in fact an {\it Isaacs operator} arising from differential games (see for instance \cite{CC95}),
 $$ F(M,y) = \inf_{a \in \A} \sup_{b \in \B} - \Tr(A^{ab}(y)M) \ \hbox{ with } \ 1 \leq A^{ab}(y) \leq \Lambda.$$
 
  Next we recall the {\it Pucci } extremal operators associated with the ellipticity class $\mathcal{S}_{\lambda,\Lambda}$, whose basic properties can be found in the book \cite{CC95}:
 $$
 \mathcal{P}^+_{\lambda,\Lambda}(M):= \Lambda \Sigma_{e_i>0} e_i + \lambda \Sigma_{e_i<0} e_i \hbox{ and } \mathcal{P}^-_{\lambda,\Lambda}(M):= \lambda \Sigma_{e_i>0} e_i + \Lambda \Sigma_{e_i<0} e_i.
  $$
  Here $e_i$'s denote the eigenvalues of $M$. The Pucci operators govern the worst possible behavior for viscosity solutions of $F(D^2u,y)=0$ with $F\in \mathcal{S}_{\lambda, \Lambda}$. More precisely note that for any $M,N\in \M^{d\times d}$ and any $x\in \R^n$ we have
  \begin{equation}\label{Pucci}
  -\mathcal{P}^+(M-N) \leq F(M,y)-F(N,y) \leq -\mathcal{P}^-(M-N).
  \end{equation}

  It is not too difficult to check that the weak maximum principle holds for uniformly elliptic equations, the more difficult thing is the comparison principle.  The following lemma, proved based the method of sup and inf-convolutions originally used by Jensen \cite{Jensen88}, shows that for uniformly elliptic nonlinear equations comparison principle for $F$ follows from maximum principle for the Pucci operators.

  \begin{lem}\label{lem: comp pucci}
  Let $\Omega$ be a domain in $\R^n$, and let $F\in\mathcal{S}(\lambda, \Lambda)$ satisfy (i). Let $u$ and $v$ satisfy, in the viscosity sense, 
  $$
  F(D^2u,y) \leq F(D^2v, y) \hbox{ in } \Omega.
  $$  
  Then $w = u-v$ satisfies, in the viscosity sense, $-\mathcal{P}^+_{\lambda,\Lambda}(D^2w) \leq 0$  and $-\mathcal{P}^-_{\lambda,\Lambda}(D^2w) \geq 0$ in $\Omega$.
  \end{lem}
  
  In addition to the role they play in the above lemma, the Pucci operators are useful because regularity results hold uniformly in the ellipticity class $\mathcal{S}(\lambda, \Lambda)$. The following result is from the book of Caffarelli and Cabr\'e \cite{CC95}:

\begin{lem}\label{lem: C_alpha est}
Let $F\in \mathcal{S}(\lambda, \Lambda)$, and let $u$ be a continuous viscosity solution of 
$$
-\mathcal{P}^+(D^2u) \leq 0 \hbox{ and } -\mathcal{P}^-(D^2u) \geq 0 \hbox{ in } B_r(0).
$$
Then for every $\alpha\in (0,1)$ there exists $C = C(\lambda, \Lambda, n,\alpha) >0$ such that
$$
\sup_{x,y\in B_{r/2}(x_0)} \dfrac{|u(x)-u(y)|}{|x-y|^\alpha} \leq C\dfrac{1}{r^\alpha}\sup_{x\in B_r(x_0)} u(x).
$$
\end{lem}

 Now given Lemma~\ref{lem: comp pucci} we can discuss uniqueness/comparison principle in bounded domains and half spaces. To start let us consider a given bounded domain $U\subset \R^n$ with a smooth boundary and a continuous boundary data $g:\partial U \to \R$. For an operator $F$ which satisfies above assumptions (i)-(iii),  
 
  \begin{equation}\label{eqn: comp example}
 \left\{
\begin{array}{lll}
F(D^2u,y) = 0 & \hbox{ in } & U \vspace{1.5mm} \\
u = g(y) & \hbox{ on } & \partial U,
\end{array} 
\right.
 \end{equation}

We refer to \cite{CC95, CIL92} for existence and uniqueness of viscosity solutions of \eqref{eqn: comp example}, which is based on the following comparison principle.
\begin{lem}[Comparison principle]
Let $u_1$, $u_2$ be viscosity sub- and supersolutions of \eqref{eqn: comp example} with boundary data $g_1\leq g_2$. Then 
$$u_1\leq u_2 \ \hbox{ in } \ U.$$
\end{lem}

As for our cell problem \eqref{eqn:cell1} posed in the half-space, we will be using the following comparison principle for \emph{bounded} viscosity solutions.

\begin{lem}[Lemma 2.9, \cite{Feldman13}]\label{lem: comparison half space}
Suppose that $U = P_\nu$ a half space with inward normal $\nu\in S^{d-1}$.  Let $g_1, g_2 \in C^{\alpha}(\R^n)$ bounded and $u_1$, $u_2$ be bounded sub and supersolutions of \eqref{eqn: comp example} with Dirichlet date $g_1$ and $g_2$ respectively, then   
$$
u_1 \leq u_2 \ \hbox{ in } \ U.
$$
\end{lem}

A similar result will hold for sub/super solutions with sublinear growth, as one can easily check there is non-uniqueness once one allows for linear growth.

\subsection{Regularity in Two Dimensions}
In $d\geq 3$ it is not known in general whether the solutions of fully nonlinear uniformly elliptic equations are smooth, examples of non-classical viscosity solutions in high dimensions ($d\geq 12$) have been given by Nadirashvili and Vl{\u{a}}du{\c{t}}\cite{NadirashviliVladut1,NadirashviliVladut2} .  On the other hand in $d=2$ it is a classical result of Nirenberg \cite{Nirenberg} that solutions are $C^{2,\alpha}$ for a small $\alpha$.  We will be able to use this result because the asymptotics near rational directions of the homogenized boundary condition in any dimension naturally turn out to be determined by a two-dimensional problem.  We state the result using more modern terminology, but our statement follows easily from Nirenberg's theorem in \cite{Nirenberg}. 

\begin{thm}[Nirenberg]\label{thm: nirenberg}
There exists $\alpha(\Lambda) \in (0,1)$ and $C(\Lambda)>0$ so that if $u: B_1 \to \real$ is a viscosity solution of $F(D^2u) = f$ in $B_1$ for some $F \in S_{1,\Lambda}$ and $f \in C^{0,\beta}(B_1)$ then for $\alpha = \min \{ \alpha_0,\beta\}$,
$$ \|D^2u\|_{C^{0,\alpha}(B_{1/2})} \leq C(\Lambda)[ \osc_{B_1} u + \|f\|_{C^{0,\beta}(B_1)}].$$
\end{thm}

\subsection{Results from Homogenization Theory}
First we describe the results obtained in \cite{Feldman13} regarding the cell problem, the Neumann counterpart is in \cite{ChoiKim13,CKL}.  Let $v_{\nu,\tau}(\cdot;(\psi,F))$ solve the cell problem \eqref{eqn:cell1}.
The following result says that, when $\nu$ is irrational, $v_{\nu,\tau}$ has a limit as $y\cdot \nu \to \infty$ and the limit is independent of $\tau$.  
\begin{thm}[Theorem $1.2$ of \cite{Feldman13}]\label{thm: irrational dir hom}
For $\nu \in S^{d-1} \setminus \real \integer^d$ there exists $\mu(\nu,\psi,F)$, called the boundary layer tail or homogenized boundary condition, such that,
$$
 \sup_{\tau \in [0,1)^d} \sup_{ y \in \partial P_\nu} |v_{\nu,\tau}(y+R\nu) - \mu|  \to 0 \ \hbox{ as } \ R \to \infty.
$$
Moreover $\mu(\cdot,\psi,F)$ is continuous on $S^{d-1}\setminus \real \integer^d$.
\end{thm}

\medskip

We will also need a rate of interior homogenization.  In general this can be derived by the same methods used by Caffarelli-Souganidis \cite{CaffarelliSouganidis} (also see Armstrong-Smart \cite{ArmstrongSmart}).  However in this paper we will only require an interior homogenization rate in the special situation where the solution of the homogenized problem in consideration is $C^{2,\alpha_0}$ due to our two dimensional reduction and Theorem~\ref{thm: nirenberg}.  In this case it is  straightforward to obtain a rate of convergence, so we provide the proof.  

\medskip

For $\nu \in S^{d-1}$ and $R>0$,  we consider the homogenization problem,
\begin{equation}\label{eqn: hom unbdd}
\left\{
\begin{array}{lll}
F(D^2u^\e,\tfrac{x}{\e}) = 0 & \hbox{ in } &   0 < x \cdot \nu < R \vspace{1.5mm}\\
u^\e = g(x) & \hbox{ on } & x \cdot \nu \in \{0,R\}
\end{array}\right.
\ \hbox{ which homogenizes to } \ \left\{
\begin{array}{lll}
\overline{F}(D^2\overline{u}) = 0 & \hbox{ in } & 0< x \cdot \nu < R \vspace{1.5mm}\\
\overline{u} = g(x) & \hbox{ on } & x \cdot \nu \in \{0,R\},
\end{array}\right.
\end{equation}
where we are considering $g : \real^d \to \real$ to be bounded and continuous. Suppose $g$ satisfies 
\begin{equation}\label{eqn: g 2d}
g(x) = g_0(x \cdot \eta, x \cdot \nu)  \ \hbox{ for some unit vector } \eta \perp \nu \ \hbox{ and some } \ g_0 : \real^2 \to \real.
\end{equation}
 Then by uniqueness  $\overline{u}(x+ t \zeta) = \overline{u}( x )$ for any $\zeta \perp \textup{span} \{ \nu,\eta\}$.  In particular $\overline{u}(t\nu+s \eta)$ actually solves a fully nonlinear uniformly elliptic problem in $d=2$ and hence has interior $C^{2,\alpha_0}$ estimates by Theorem~\ref{thm: nirenberg}.
 
 \medskip
 
We recall, for example from Evans \cite{Evans}, that for each $M \in \M_{d\times d}$ there is a unique constant $\overline{F}(M)$ and a unique  (modulo constants) $\integer^d$-periodic bounded solution $v(y;M)$ of
 \begin{equation}\label{corrector}
F(M+D^2v,y) = \overline{F}(M) \ \hbox{ in } \ \real^d,
\end{equation}
satisfying $\|v(\cdot;M)\|_{L^\infty} \leq C(\Lambda,d)\|M\|$.  Again from \cite{Evans}, $\overline{F}$ turns out to be uniformly elliptic with the same ellipticity ratio $\Lambda$ as $F(M,y)$. 

\begin{thm}\label{thm: 2d hom}
Let $u^\e,u,g$ be as given in \eqref{eqn: hom unbdd} and \eqref{eqn: g 2d}.  There exists $0<\alpha(\Lambda)<1$ such that for any $\beta \in (0,1)$ and any $R>0$,
$$
\sup_{ 0 < x \cdot \nu < R}|u^\e(x) - \overline{u}(x)| \leq C(\Lambda,d)(\osc_{x\cdot \nu \in \{0,R\}} g+R^\beta|g|_{C^{0,\beta}})(R^{-1}\e)^{\alpha\beta} .
$$
\end{thm}

\begin{proof}
After rescaling  we may assume that $R=1$ and $U = \{0 < x \cdot \nu < 1\}$. Let $\delta \in (0,\|g\|_{C^{0,\beta}(\partial U)}] $ to be chosen later and $\overline{u}^\delta$ solve
\begin{equation}\label{eqn: u bar delta}
\left\{
\begin{array}{lll}
\overline{F}(D^2\overline{u}^\delta) = \delta & \hbox{ in } & U \vspace{1.5mm} \\
\overline{u}^\delta = g(x) & \hbox{ on } & \partial U.
\end{array}\right.
\end{equation}
We claim that
$$
 \sup_U|\overline{u}^\delta - \overline{u}| \leq \tfrac{1}{8}\delta.
 $$
To prove this we look at $w = \overline{u}^\delta - \overline{u}$ which, by Lemma~\ref{lem: comp pucci},  is a solution of 
$$-\mathcal{P}^+_{1,\Lambda}(D^2w) \leq \delta \leq -\mathcal{P}^-_{1,\Lambda}(D^2w) \ \hbox{ in $U$ with $w = 0$ on $\partial U$}. $$
  Comparing with $0$ implies $w \geq 0$ and, for the other direction, let $\varphi(x) = \frac{\delta}{2}(\frac{1}{4}-(x \cdot \nu-\frac{1}{2})^2)$, then $-\mathcal{P}^+_{1,\Lambda}(D^2\varphi) = \delta$ with $\varphi \geq 0$ on $\partial U$, so comparison principle implies $w \leq \varphi \leq \frac{\delta}{8}$ in $U$.

\medskip

We will construct a supersolution barrier function based on $\bar{u}^\delta$ to compare with $u^\e$ away from the boundary. We begin by collecting uniform estimates on $\overline{u}^\delta$. 
  Let $1>h>0$ to be chosen small and call $U_h = \{ x : d(x,\partial U) >h\}$.   By the $C^{0,\beta}$ estimates up to the boundary -- see Lemma 2.11 of Feldman or combine Lemma~\ref{lem: C_alpha est} above with Lemma~\ref{bdry cont} below -- for both $\overline{u}$ and $u^\e$ at unit scale,
$$ \|\overline{u}^\delta\|_{C^{0,\beta}(U)}+\|u^\e\|_{C^{0,\beta}(U)}\leq C(\|g\|_{C^{0,\beta}(\partial U)}+\delta) \leq C\|g\|_{C^{0,\beta}(\partial U)}, $$
where we have used that $\delta \leq \|g\|_{C^{0,\beta}(\partial U)}$.  Moreover, due to Theorem~\ref{thm: nirenberg}  we have 
$$ |D^2\overline{u}^\delta(x)| \leq C(h^{-2}\osc_{B_h(x)} \overline{u}^\delta + \delta ) \leq Ch^{\beta-2}\|g\|_{C^{0,\beta}} \ \hbox{ for } \ x \in U_h$$
 where we have also used again $\delta \leq \|g\|_{C^{0,\beta}(\partial U)}$ and $h^{\beta-2} >1$.  Similarly from Theorem~\ref{thm: nirenberg},
$$ |D^2\overline{u}^\delta(x)|_{C^{0,\alpha_0}(U_h)} \leq Ch^{\beta-(2+\alpha_0)}\|g\|_{C^{0,\beta}}.$$

\medskip

Note that for $x \in U \setminus U_h$ there is $y \in \partial U$ with $|y-x| \leq h$ and thus
$$ |\overline{u}^\delta(x) - u^\e(x)| \leq |\overline{u}^\delta(x)-g(y)| + |u^\e(x)-g(y)|\leq C\|g\|_{C^{0,\beta}(\partial U)}h^\beta.$$
We wish to show that, in fact, the maximum of $u^\e(x)-\overline{u}^\delta(x)$ in $U$ is obtained in $U \setminus U_h$.  Suppose otherwise, then there exists $x_0\in U_h$ such that 
\begin{equation}\label{maximum}
  u^\e(x_0) -\overline{u}^\delta(x_0)= \max_{U}(   u^\e-\overline{u}^\delta).
 \end{equation}
In particular $\overline{u}^\delta(x)+ u^\e(x_0) -\overline{u}^\delta(x_0)$ touches $u^\e$ from above at $x_0$. Let us define the barrier function
$$
\phi^\e(x):= u^\e(x_0)+D\overline{u}^\delta(x_0) \cdot (x-x_0) + \tfrac{1}{2}(x-x_0)\cdot D^2\overline{u}^\delta(x_0) (x-x_0) + \e^2 v(\tfrac{x}{\e};D^2\overline{u}^\delta(x_0))+\tfrac{\delta}{2\Lambda}|x-x_0|^2,
$$
where $v$ is the corrector given in \eqref{corrector}.  Note that
 \begin{equation}\label{touching}
|\phi^\e(x_0) - u^\e(x_0)| \leq C_0 \|D^2\overline{u}^\delta\|_{L^\infty(U_h)}\e^2.
\end{equation}

One can verify that, using the uniform ellipticity and the definition of the corrector,
$$
F(D^2\phi^\e(x),\tfrac{x}{\e}) \geq \overline{F}(D^2\overline{u}^\delta(x_0))-\delta \geq 0.
$$

Let us choose 
$$
r = \min[\e^{\frac{2}{2+\alpha_0}}(\|D^2\overline{u}^\delta\|_{L^\infty(U_h)}/|D^2\overline{u}^\delta|_{C^{0,\alpha_0}(U_h)})^{\frac{1}{2+\alpha_0}},h].
$$

We claim that, for $\delta$ sufficiently small and $C_0$ from \eqref{touching},
\begin{equation}\label{claim}
 \phi^\e(x) \geq u^\e(x) + 2C_0 \|D^2\overline{u}^\delta\|_{L^\infty(U_h)}\e^2 \ \hbox{ on } \ \partial B_{r}(x_0).
 \end{equation}
The comparison principle then would yield that the same inequality holds in $B_r(x_0)$, yielding a contradiction to \eqref{touching}.  We now verify that $\delta$ can be chosen so that \eqref{claim} holds. Using \eqref{maximum} we have
$$ \phi^\e(x) \geq u^\e(x) + \tfrac{\delta}{2\Lambda}|x-x_0|^2  - C\|D^2\overline{u}^\delta\|_{L^\infty(U_h)}\e^2 - C|D^2\overline{u}^\delta|_{C^{0,\alpha_0}(U_h)}|x-x_0|^{2+\alpha_0},$$
we have chosen $r$ above so that the last two terms are of the same size on $\partial B_r(x_0)$.  Thus, evaluating this on $\partial B_r(x_0)$ we have
$$ 
\phi^\e(x) \geq u^\e(x) + \tfrac{\delta}{2\Lambda}\e^{\frac{4}{2+\alpha_0}}(\|D^2\overline{u}^\delta\|_{L^\infty(U_h)}/|D^2\overline{u}^\delta|_{C^{0,\alpha_0}(U_h)})^{\frac{2}{2+\alpha_0}}  - C_1\|D^2\overline{u}^\delta\|_{L^\infty(U_h)}\e^2.$$
Now suppose $r<h$ and let us choose 
\begin{equation}\label{delta}
\delta \leq C\|g\|_{C^{0,\beta}}\max\{\e^{\frac{2\alpha_0}{2+\alpha_0}}h^{\beta-(2+\alpha_0)},h^{\beta-4}\e^2\},
\end{equation}
then due to the regularity estimates on $\bar{u}^{\delta}$ given above we arrive at \eqref{claim}. If $r = h$ then $\delta = Mh^{\beta-4}\e^2$ to get the same contradiction.  By a parallel argument we can show that the same choice of $\delta$ will result in the minimum of $u^\e(x)-\overline{u}^\delta(x)$ occurring in $U \setminus U_h$.

\medskip

Now we put together the bounds obtained above. Since the maximum and minimum of $\overline{u}^\delta - u^\e$ are achieved in $U \setminus U_h$ for $\delta$ as above,
$$ \sup_U |\overline{u} - u^\e| \leq \sup_U |\overline{u} -\overline{u}^\delta| + \sup_{U \setminus U_h}|\overline{u}^\delta - u^\e| \leq C\delta + C\|g\|_{C^{0,\beta}}h^\beta.$$
Using $\delta$ as chosen in \eqref{delta},
$$
 \sup_U |\overline{u} - u^\e| \leq C\|g\|_{C^{0,\beta}}(\e^{\frac{2\alpha_0}{2+\alpha_0}}h^{\beta-(2+\alpha_0)}+h^{\beta-4}\e^2+h^\beta).
$$
By choosing $h = \e^{\frac{2\alpha_0}{(2+\alpha_0)^2}}$ we arrive at 
$$ \sup_U |\overline{u} - u^\e| \leq C\|g\|_{C^{0,\beta}}\e^{\beta\frac{2\alpha_0}{(2+\alpha_0)^2}}.$$
\end{proof}

\subsection{Continuity up to the Boundary}
We will use the following result repeatedly in what follows. It is a fundamental technical tool used in estimating the difference between cell problem solutions in nearby half-spaces.  The result addresses the continuity-up-to-the boundary for solutions of the Dirichlet problem, but it can be also viewed as a localization result.
\begin{lem}\label{bdry cont}
Suppose that $\omega : [0,\infty) \to [0,\infty)$ is a modulus of continuity and $u \leq \omega(1)$ satisfies,
\begin{equation*}
\left\{
\begin{array}{lll}
-\mathcal{P}^+_{1,\Lambda}(D^2u) \leq 0 & \hbox{ in } & B_1 \cap K \vspace{1.5mm}\\
u(x) \leq \omega(|x|) & \hbox{ on } & \partial (B_1 \cap  K)
\end{array} 
\right.
\end{equation*}
where $K$ is any set satisfying $0 \in \partial K$ and $B_1 \cap K \subset B_1^+$. Then there is a modulus $\bar{\omega}$ depending on $\Lambda,d$ and $\omega$  such that
$$ u(x) \leq C(\Lambda,d) \bar{\omega}(|x|) \ \hbox{ in } \ B_1 \cap K.$$
If $\omega(r) =  r^\beta$ for some $\beta \in (0,1)$ then $\bar\omega(r) = C(\Lambda,d,\beta)r^\beta$. 
\end{lem}
For us $K$ will either be the upper half space $P_{e_d}$ or an intersection of two half-spaces (see Lemma~\ref{lemma_1}). Because it is not obvious how to calculate $\bar\omega$ for general $\omega$ we will work with H\"{o}lder continuous boundary conditions throughout the paper so that we get explicit estimates.  The generalizations to arbitrary $\omega$ present only notational difficulties.  
\begin{proof}

By rescaling, without loss $\omega(1) = 1$.  The proof is quite analogous to the standard barrier method for boundary continuity for harmonic functions, we just need to work with the Pucci operator instead of the Laplacian.  Let $\phi$ be a positive, smooth function in $\overline{B_1^+}$ satisfying
\begin{equation*}
\left\{
\begin{array}{lll}
-\mathcal{P}^+_{1,\Lambda}(D^2\phi) \geq 0 & \hbox{ in } & B_1 \cap K \vspace{1.5mm}\\
 \phi \geq 1 & \hbox{ on } & \partial B_1 \cap K
\end{array}
\right. \ \hbox{ and } \  \phi(x) \leq C_0(\Lambda,d)|x| \ \hbox{ in } \ B_1 \cap K.
\end{equation*}
For example one can choose $\phi$ to be a rescaled translation of the downward pointing fundamental solution for the Pucci operator, 
$$\phi(x) = L(1-|x+e_d|^{1-\Lambda(d-1)}) \ \hbox{ with } \ L = (\min_{|x| = 1, x_d >0} |1- |x+e_d|^{1-\Lambda(d-1)}| )^{-1} ,$$
which one can check is actually a smooth solution of $-\mathcal{P}^+_{1,\Lambda}(D^2\phi) = 0$ except at $x = -e_d$.  For each $r>0$ and an $M>1$  to be chosen large consider the barrier,
$$ 
\phi_r(x) = \omega(Mr)+( \sup_{B_{Mr} \cap K} u )_+\phi((Mr)^{-1}x).
$$
Then
$\phi_r(x) \geq ( \sup_{B_{Mr} \cap K} u) \geq u$ on $\partial B_{Mr} \cap K$ and
$$
 \phi_r(x) \geq \omega(Mr) \geq \omega(|x|) \geq u(x) \hbox{ on }\partial K \cap B_{Mr},
$$
since $\omega$ is monotone.  By comparison principle $u \leq \phi_r$ in $B_{Mr} \cap K$ and therefore it follows that
\begin{equation}\label{eqn: Mr}
 (\sup_{B_r \cap K} u(x) )_+\leq \omega(Mr)+( \sup_{B_{Mr} \cap K} u)_+ \sup_{x \in B_r \cap K}\phi(\tfrac{x}{Mr}) \leq \omega(Mr)+( \sup_{B_{Mr} \cap K} u)_+\tfrac{C_0}{M},
\end{equation}
where we have used that $\phi(x) \leq C_0|x|$ for the second inequality. To get a modulus of continuity, let $\e>0$, choose $M \geq 2\e^{-1}C_0$ and then choose $r$ sufficiently small to make the right hand side in \eqref{eqn: Mr} less than or equal to $\e$.  

\medskip

On the other hand, since the argument is valid for every $r>0$, applying the estimate repeatedly up until $M^{n+1}r \geq 1$,
$$ \sup_{B_r \cap K} u(x) \leq \sum_{j=1}^{n-1}\omega(M^jr)(\tfrac{C_0}{M})^{j-1}+( \sup_{B_{1} \cap K} u)_+(\tfrac{C_0}{M})^{n-1} $$
In case $\omega(r) = r^\beta$ for some $\beta \in (0,1)$, choose $M^{1-\beta} = 2C_0(\Lambda,d)$ so that
\begin{align*}
 \sup_{B_r \cap K} u(x) &\leq Mr^\beta\sum_{j=1}^{n-1}C_0^jM^{-(1-\beta)j}+( \sup_{B_{1} \cap K} u)_+(M^{-(1-\beta)}C_0)^{n-1}M^{-(n-1)\beta} \\
 & \leq 2 M r^\beta+( \sup_{B_{1} \cap K} u)_+2^{-(n-1)}M^{2\beta} r^\beta \\
 &\leq C'(\Lambda,d,\beta)(1+( \sup_{B_{1} \cap K} u)_+) r^\beta.
 \end{align*}

\end{proof}

\subsection{Some Number Theory}\label{sec: Basic Number Theory}
Lastly we present some elementary number theoretic results which we will make use of.  When $\nu \in \real\integer^d$ is a rational direction then there is some minimal $T = T(\nu)>0$ such that,
\begin{equation}\label{period}
 (\partial P_\nu + T\nu )\bmod \integer^d = \partial P_\nu .
 \end{equation}
 \begin{lem}\label{T period}
If $\xi \in \integer^d \setminus \{0\}$ is irreducible in the sense that $\textup{gcd}(\xi_1,\dots,\xi_d) = 1$ then,
 $$ T( \hat{\xi}) = |\xi|^{-1}.$$
 \end{lem}
  \begin{proof}
 From B\'ezout's identity there exists $x \in \integer^d$ so that $\xi \cdot x = \textup{gcd}(\xi) = 1$.  Then $x \cdot \hat \xi = |\xi|^{-1}$ and so $x \in \partial P_{\hat\xi} + \frac{1}{|\xi|}\hat\xi$. Thus,
  $$0 \in [\partial P_{\hat{\xi}} \bmod \integer^d] \cap [(\partial P_{\hat{\xi}} + \frac{1}{|\xi|} \hat{\xi} )\bmod \integer^d],$$
  and so the two sets are the same.  This shows that $T(\hat\xi) \leq \frac{1}{|\xi|}$, for the other direction one just needs to note that $\xi \cdot x$ is an integer for every $x \in \integer^d$ so that if $\xi \cdot x \neq 0$ then $|\xi \cdot x| \geq 1$.
  
 \end{proof}

\begin{lem}\label{lem: lattice size}
If $\xi \in \integer^d \setminus \{0\}$ then $ \partial P_{\xi}$ is spanned by $d-1$ vectors $f^j \in \integer^d$ with $|f^j| \leq |\xi|$.
\end{lem}
\begin{proof}
Without loss assume that $|\xi_d| = \argmax_{1 \leq i \leq d} |\xi_i| >0$ since $\xi \neq 0$.  Then call, for $1 \leq j \leq d-1$,
$$ f^j = \xi_d e_j - \xi_j e_d \ \hbox{ and from the definition } \ f^j \cdot \xi =  0.$$
\end{proof}

\noindent Next we state a classical number theoretic result, the simultaneous version of Dirichlet's approximation Theorem.  The proof is by pigeon-hole principle.

\begin{thm}\label{Dirichlet}
For given real numbers $\alpha_1,...,\alpha_n$ and $N\in\mathbb{N}$, there are integers $p_1,...,p_n,q\in\integer$ with $1\leq q\leq N$ such that
$$
|q\alpha_i - p_i | \leq \frac{1}{N^{1/n}}.
$$
\end{thm}

\section{Asymptotics of Half-space Solutions with Periodic Boundary Conditions}\label{sec: periodic bc}
Here we consider the convergence rate for homogenization of half-space problems.  First consider the solution $v$ of the following problem in a half-space,

\begin{equation}\label{eqn: periodic BC}
\left\{
\begin{array}{lll}
F(D^2v,y) = f(y) & \hbox{ in } & P_{e_d} \vspace{1.5mm} \\
v = \phi(y) & \hbox{ on } & \partial P_{e_d}.
\end{array}\right.
\end{equation}
We assume that $F$, $f$ and $\phi$ are periodic with respect to linearly independent translations $\ell_1,...\ell_{d-1} \in\partial P_{e_d}$.  We define the lattice of periodicity and its unit cell, 
$$
 \Z := \bigg\{ z : z = \sum_{j=1}^{d-1}k_j\ell_j , \ k_j \in \integer\bigg\}, \quad Q := \bigg\{ \sum \lambda_j \ell_j : \lambda_j \in [0,1) \bigg\},  \quad\hbox{ and } L := \textup{diam}(Q).
$$
In this section we will only consider the case $f\equiv0$ for the simplicity of presentation.  The proof of Lemma~\ref{lem: exp rate per} for the general case $f\neq 0$, which is needed in Section~\ref{linear}, is presented in Appendix~\ref{appendix}.  The calculations are rather delicate, since for later usage it will be important for us to keep track of the dependence on the unit cell size $L$.  

\medskip

 The following lemma states that the rate of convergence to the homogenized boundary condition will be exponentially fast depending on $L$ and universal constants.  This result is originally due to Tartar \cite{Tartar} for linear divergence form operators. To the best of our knowledge the result is new for nonlinear operators.  The proof is an iterative argument using the $\Z$-periodicity of the solution and the interior oscillation decay from Harnack inequality.

\begin{lem}\label{lem: exp rate per}
There exist $\mu(\phi,F)$ and $c_0(\Lambda,d)>0$ such that,
$$ \sup_{y \cdot e_d \geq R}|v(y) - \mu| \leq C(\Lambda,d)(\osc \phi)\exp(-c_0L^{-1}R),$$
and this estimate gives the optimal rate up to the determination of $c_0$.
\end{lem}
\begin{proof}
By rescaling we may assume without loss that $\osc \phi = 1$.  Let $\alpha(d,\Lambda) \in (0,1)$ be the H\"{o}lder continuity exponent and $C_0(\Lambda,d)$ the constant in the interior H\"{o}lder estimate for the maximal class (see \cite{CC95}).  For $r>C_0^{1/\alpha}L$ we claim that 
\begin{equation}\label{eqn: periodic iteration}
 \osc_{ P_{e_d} +k r e_d} v \leq C^k_0(L/r)^{\alpha k }\quad \hbox{ for any } k\in\mathbb{N}.
 \end{equation}
Supposing that this result holds, let $k = [R/r]$ to obtain
$$ \osc_{ P_{e_d} +Re_d} v  \leq \osc_{ P_{e_d} +k re_d} v \leq \exp\left(k\log\frac{C_0L^\alpha}{r^\alpha}\right).$$
When we choose $r = eC_0^{1/\alpha}L$ the estimate becomes,
$$ \osc_{ P_{e_d} +Re_d} v  \leq \exp\left( - \left[ \frac{R}{eC_0^{1/\alpha}L}\right]\right) \leq C\exp\left( - cR/L\right)$$
with $C = e$ and $c = e^{-1}C_0^{-1/\alpha}$.

\medskip

It remains to prove \eqref{eqn: periodic iteration} by induction.  For $k = 0$ \eqref{eqn: periodic iteration} follows from the maximum principle.  Assuming \eqref{eqn: periodic iteration} for $k$, we prove it for $k+1$.  Note that
$$v_k(y) = C^{-k}_0r^{\alpha k} L^{-\alpha k} v(y + kr e_d)$$
satisfies $F( \cdot , y + kr e_d) = 0$ in $P_{e_d}$ with boundary data $\phi_k(y) = C^{-k}_0r^{\alpha k} L^{-\alpha k} v(y+ke_d)$.  Both the operator and the boundary data are periodic with respect to $(\ell_j)_{j = 1}^{d-1}$ translations by uniqueness, and $\osc \phi_k \leq 1$ by the inductive hypothesis.  Then by the interior H\"{o}lder estimate in $B_r(r e_d)$,
$$ 
| v_k|_{C^\alpha(B_{r/2}(r e_d))} \leq C_0r^{-\alpha} \ \hbox{ and so } \ \osc_{Q + re_d} v_k\leq C_0(L/r)^{\alpha},$$
where we have used $r > 2L$ so that $Q \subset B_{r/2}(0)$.  On the other hand $v_k$ is periodic on $\partial P_{e_d} + re_d$ with respect to the translations $(\ell_j)_{j=1}^{d-1}$ and periodicity cell $Q$. Therefore
$$
 \osc_{ P_{e_d} +re_d} v_k \leq \osc_{ \partial P_{e_d} +re_d} v_k\leq C_0(L/r)^{\alpha},
 $$
where we have again used maximum principle for the first inequality.  Rewriting this in terms of $v$,
$$
 \osc_{ P_{e_d} +(k+1)re_d} v = C^{k}_0(L/r)^{\alpha k} \osc_{ P_{e_d} +re_d} v_k \leq C_0^{k+1}(L/r)^{(k+1)\alpha}.
 $$
This completes the inductive proof.

\medskip

Lastly to show that the rate is optimal we take $F = -\Delta$ and $\phi = \cos \tfrac{2\pi y_1}{L}$.  Then $v(y)$ can be explicitly computed using separation of variables as
$$ v(y) = \cos( \tfrac{2\pi y_1}{L}) \exp(-\tfrac{2\pi}{L}y_2).$$
Plugging in $y_1 = 0$ and $y_2 = R$ completes the proof since evidently $\mu = 0$ in this case.
\end{proof}

We will need a slight variant of Lemma~\ref{lem: exp rate per} when the operator does not share the periodicity cell of the boundary data but its oscillations are at a smaller scale (see the proof of Proposition~\ref{lem: reduction 1}).   We no longer assume that $F$ shares the periodicity lattice of $\phi$ and instead we suppose that there is $0<\e\leq L$ such that,
\begin{equation}\label{eqn: F eps per}
\hbox{for every $y \in \partial P_{e_d}$ there is $y'$ with $|y-y'| \leq \e$ and $F(M,\cdot + y') = F(M,\cdot)$ in $P_{e_d}$.}
\end{equation}
\begin{lem}\label{lem: exp rate almost per}
Let $\phi, v$ and $L$ as given in Lemma~\ref{lem: exp rate per} and $F$ satisfies \eqref{eqn: F eps per}.  Then there exists $C,c>0$ depending only on $\Lambda,d$ such that,
$$ \osc_{y \in \partial P_{e_d}}v(Re_d+y) \leq C[(\osc \phi)\exp(-cL^{-1}R)+\omega_v(\e)].$$
Here $\omega_v$ is the modulus of continuity of $v$ at points of $\partial P_{e_d}$,
$$ \omega_v(r) := \sup \{ |v(y) - v(y')| : y \in \partial P_{e_d}, y' \in P_{e_d} \ \hbox{ and } \ |y-y'| \leq r\}.$$
\end{lem}
Note that by Lemma \ref{bdry cont} we have, for any $\beta \in (0,1)$, $\omega_v(\e) \leq C(\Lambda,d,\beta)\|\phi\|_{C^{0,\beta}}\e^\beta$.  
\begin{proof}
The proof is a minor modification of the proof of Lemma \ref{lem: exp rate per} and so we mainly focus on the difference in the proof. We will prove that there is $\alpha(\Lambda,d)\in(0,1), C_0(\Lambda,d)>0$ such that for $r \geq (2C_0)^{1/\alpha}L$
\begin{equation}\label{eqn: almost periodic iteration}
 \osc_{ P_{e_d} +k r e_d} v \leq C_0^k(L/r)^{\alpha k}(\osc\phi)+C_0\omega_v(\e) \quad\hbox{ for any } k\in\mathbb{N}.
 \end{equation}
 Following the proof of Lemma \ref{lem: exp rate per} we are done as long as we can show \eqref{eqn: almost periodic iteration}.
 
 \medskip

  The proof of \eqref{eqn: almost periodic iteration} is again by induction.  We wish to show that \eqref{eqn: almost periodic iteration} holds for all $k$.   Assuming that  \eqref{eqn: almost periodic iteration} holds up to $k$ we prove for $k+1$.  Note that $v_k(y) = v(y+kre_d)$ solves the equation
 $$
  F_k(D^2v_k,y) := F(D^2v_k,y+kre_d) = 0 \ \hbox{ in } \ P_{e_d} 
  $$
 with boundary data $\phi_k (y) = v_{k-1}(y+re_d)$.  Note that $F_k$ satisfies the same assumption as $F$. By the interior oscillation decay for the ellipticity class,  there exists a universal constant $C_1>1$ such that 
  $$
   \osc_{\ell + Q +re_d}v_k \leq C_1(L/r)^\alpha \osc_{\partial P_{e_d}} \phi_k \quad \hbox{ for any } \ell \in \Z.
  $$
For an arbitrary $y \in \partial P_{e_d} $ let $\ell \in \Z$ such that $y \in \ell+Q$,
  \begin{align}
   |v_k(y+re_d) - v_k(re_d)| &\leq 2\osc_{\ell + Q +re_d}v_k +|v_k(\ell+re_d) - v_k(re_d)|\notag \\
   &\leq 2C_0'(L/r)^\alpha \osc_{\partial P_{e_d}} \phi_k+ |v_k(\ell+re_d) - v_k(re_d)|. \label{eqn: est 1}
   \end{align}
 The second term on the right hand side above appears since $v_k$ is no longer periodic.  By the assumption on $F$ there is $\ell' \in P_{e_d}$ with $|\ell'-\ell| \leq \e$ and $F(M,\cdot + \ell') = F(M,\cdot)$ in $P_{e_d}$ for all symmetric matrices $M$.  Therefore we can estimate,
  $$
  |v_k(\ell+re_d) - v_k(re_d)| \leq |v_k(\ell+re_d) - v_k(\ell'+re_d)|+|v_k(\ell'+re_d)-v_k(re_d)|.
  $$
  The first term can be estimated by,
  \begin{equation}\label{eqn: est 2}
  |v_k(\ell+re_d) - v_k(\ell'+re_d)| \leq C_0'(|\ell-\ell'|/r)^\alpha \osc_{\partial P_{e_d}} \phi_k .
  \end{equation}
  For the second term it suffices to bound $w(y) = v(y + \ell' )- v(y)$ in  $P_{e_d}$.  From the invariance of the operator $F$ under translation by $\ell'$ and Lemma~\ref{lem: comp pucci}, $w$ is a viscosity solution of 
  \begin{equation*}
   -\mathcal{P}^+_{1,\Lambda}(D^2w) \leq 0 \leq - \mathcal{P}^-_{1,\Lambda}(D^2w) \ \hbox{ in } \ P_{e_d}.
   \end{equation*}
     Using the periodicity of $\phi$ we have $w(y) = v(y + \ell' )- \phi(y+\ell)$ on $\partial P_{e_d}$. Then by maximum principle and  the definition of $\omega_v$,
  \begin{equation}\label{eqn: est 3}
   \sup_{P_{e_d}} |w| \leq \sup_{\partial P_{e_d}} |w| = \sup_{\partial P_{e_d}} |v(y + \ell' )- \phi(y+\ell)| \leq \omega_v(|\ell-\ell'|).
   \end{equation}
Plugging \eqref{eqn: est 2} and \eqref{eqn: est 3} into \eqref{eqn: est 1} and using $|\ell-\ell'| \leq \e \leq L$ we obtain,
\begin{equation}
   |v_k(y+re_d) - v_k(re_d)| \leq  3C_0'(L/r)^\alpha \osc_{\partial P_{e_d}} \phi_k+ \omega_v(\e).
   \end{equation}
  We can conclude now since,
  \begin{align*}
 \osc_{\partial P_{e_d}+(k+1)re_d} v &=  \osc_{\partial P_{e_d}+re_d} v_k \\
 & \leq 3C_1(L/r)^\alpha \osc_{\partial P_{e_d}} \phi_k+ \bar\omega_{\phi}(\e) \\
 & \leq 3C_1C_0^k(L/r)^{(k+1)\alpha}\osc \phi+(3C_1C_0^k(L/r)^{k\alpha}+1)\omega_v(\e) \\
 & \leq C_0^{k+1}(L/r)^{(k+1)\alpha} \osc \phi + C_0\omega_v(\e)
  \end{align*}
  where the last step holds if we choose $C_0 = 3C_1$ and $r \geq 2^{1/\alpha}LC_0^{1/\alpha}$ since $1<C_1$.
 
\end{proof}

\section{Asymptotics of half-space solutions near rational directions} \label{sec: asymptotics}

In this section we study asymptotic behavior of half-space solutions as the normal direction $\nu$ approaches a rational direction, $\hat\xi$ for some $\xi \in \integer^d \setminus \{0\}$.  Let us  recall the solution $v_{\nu,\tau}$ of the cell problem \eqref{eqn:cell1} defined for a direction $\nu \in S^{d-1}$, a $\tau \in \real^d$ and a continuous $\integer^d$-periodic $\psi$ by,
\begin{equation}\label{eqn: cell1-2}\tag{\ref{eqn:cell1}}
\left\{
\begin{array}{lll}
F(D^2v_{\nu,\tau},y+\tau) = 0 & \hbox{ in } & P_\nu \vspace{1.5mm} \\
v_{\nu,\tau} = \psi(y+\tau) & \hbox{ on } & \partial P_\nu.
\end{array}
\right.
\end{equation}

Due to Theorem~\ref{thm: irrational dir hom}, for irrational directions  there exists a limit $\mu(\nu,\psi,F)$ such that,
 \begin{equation}\label{eqn: cell hom irr}
 \sup_{\tau \in \real^d} \sup_{y \in \partial P_\nu} |v_{\nu,\tau}(y+R\nu;(\psi,F)) - \mu(\nu,\psi,F)| \to 0 \ \hbox{ as } \ R \to \infty. 
 \end{equation}

We are interested to understand the asymptotic behavior of $\mu(\nu,\psi,F)$ as $\nu$ approaches a rational direction $\hat{\xi}$ for $\xi \in \integer^d \setminus \{0\}$. The limiting behavior, it will turn out, depends on the direction of tangential approach.   The main result of this section (stated quantitatively in Proposition \ref{lem: asymptotics}) is the following:

\begin{thm}
 Let $\beta \in (0,1)$ and $\psi \in C^{0,\beta}(\mathbb{T}^d)$. For any $\xi \in \integer^d \setminus \{0\}$, irreducible, there exists a function $L_{\xi}(\cdot)= L_{\xi}(\cdot;\psi,F)$ on unit vectors tangent to $S^{d-1}$ at $\hat\xi$ and a mode of continuity, $\omega_{\xi,\beta}$, such that the following holds:
$$ | \mu(\nu(t),\psi,F) - L_\xi(\nu'(0))| \leq \omega_{|\xi|,\beta}(|\nu(t) - \nu(0)|)$$
for any $\nu : [0,1) \to S^{d-1}$ a unit speed geodesic with $\nu(0) = \hat\xi$.
\end{thm}

The basic idea behind these asymptotics already appeared in \cite{CKL, ChoiKim13} for the Neumann problem as a part of their proof that, when $\overline{F}$ is rotation invariant, $\mu(\cdot,\psi,F)$ has a continuous extension from the irrational directions to the entire unit sphere.  The proof proceeds by a series of reductions which will be carried out by a multi-scale homogenization argument.  Our analysis is a quantitative and improved version of the proof given in \cite{ChoiKim13} in the following sense.  First we have tried to obtain optimal estimates at each stage of the argument.  We do this with the hope of clarifying the proof and of achieving improved quantitative results on the continuity of $\mu$ in the end.  Secondly we introduce the directional limit $L$ and observe that $L$ depends on a two-dimensional projected version of the problem (see Section \ref{step 3}). It is for this reason that we are able to use Nirenberg's two dimensional regularity result and the corresponding interior homogenization result, Theorem \ref{thm: 2d hom}.  By this careful exposition we are able to obtain a precise characterization of the asymptotic behavior of $\mu$ at rational directions and its dependence on the operator $F$ and boundary data $\psi$.  With this characterization we are able to understand both continuity and discontinuity, Sections 5 and 6 respectively, in a unified way.

\subsection{Step 1: Replacing the Boundary Condition at an Intermediate Scale}\label{step 1}

Let $ \xi \in \integer^d\setminus \{0\}$ be irreducible and let $v_{\xi, \tau}$ solve  \eqref{eqn: cell1-2} in $P_{\xi}$. By the results of Section \ref{sec: Basic Number Theory} the boundary data $\left. \psi\right|_{\partial P_\xi}$ is periodic with respect to a lattice on $\partial P_\xi$ with unit cell size $\leq C_d |\xi|$.  The result of the previous section implies that for each $\tau \in \real^d$ there is a limit at infinity in the $\xi$ direction. The  limit for $\tau,\tau' \in \real^d$ is the same when $\tau,\tau'$ are both in $\partial P_\xi + t \hat \xi$ modulo $\integer^d$ for some $t \in \real$.  By Lemma \ref{T period} this is exactly when $\tau \cdot \hat \xi \bmod \frac{1}{|\xi|}\integer = \tau' \cdot \hat \xi \bmod \frac{1}{|\xi|}\integer$. We define 
$$ {m}_\xi(t;(\psi,F)) : = \lim_{R \to \infty} v_{\hat{\xi}, t\hat{\xi}} (R\hat{\xi})$$
which is continuous, $\frac{1}{|\xi|}$-periodic on $\real$. By definition we have 
$$ \lim_{R \to \infty} v_{\xi,\tau}(R \hat \xi) = m_\xi(\tau \cdot \hat\xi ;(\psi,F)).$$
Generally speaking $m_\xi$ inherits the up to the boundary regularity of the cell problem solutions.  In the general fully nonlinear case this is limited to $C^{0,1}$, but for linear operators the result would hold for arbitrary $C^{k,\beta}$, see Section \ref{linear}.  
\begin{lem}\label{est:m}
If $\psi$ is continuous with modulus $\omega$ then $m_\xi(\cdot;(\psi,F))$ is continuous with the new modulus $\bar\omega$ from Lemma~\ref{bdry cont}.  In particular for $\beta \in (0,1)$,
$$ \|m_\xi(\cdot;(\psi,F))\|_{C^{0,\beta}} \leq C(\Lambda,d,\beta) \|\psi\|_{C^{0,\beta}}.$$
Moreover we also have,
$$ \|m_\xi(\cdot;(\psi,F))\|_{C^{0,1}}\leq C(\Lambda,d,\beta) \|\psi\|_{C^{1,\beta}}.$$
\end{lem}
  The proof is a straightforward application of the boundary continuity estimates Lemma~\ref{bdry cont} combined with the definition of $m_\xi$, and is postponed till the end of this section.  We drop the dependence of $m_\xi$ on $(\psi,F)$ as long as there is no ambiguity.  Due to Lemma \ref{lem: exp rate per} of the previous section,
$$ \sup_{P_\xi + R\hat{\xi}}|v_{\xi,t \hat{\xi}}(\cdot)-m_\xi(t)| \leq C(\osc \psi) \exp (- cR/|\xi|).$$

  Let $\nu \in S^{d-1}$ be an irrational direction and $\mu(\nu,\psi,F)$ the boundary layer tail of the cell problem solutions $v_{\nu,\tau}$ (see Theorem \ref{thm: irrational dir hom} for the definition of $\mu$).  Since the limit is independent of $\tau$ (from Theorem \ref{thm: irrational dir hom}) we simply refer to $v_\nu = v_{\nu,0}$ when $\nu$ is irrational.  We consider the asymptotics of $\mu(\nu,\psi,F)$ as $\nu$ approaches $\hat{\xi}$.   
  
  \medskip

      When $\nu  \neq - \hat \xi$ there is a unique vector $\eta \perp \xi$ (see Figure 1) so that,
    \begin{equation}\label{eqn: eta def}
     \nu = (\cos |\eta|) \hat\xi- (\sin |\eta|) \hat\eta \ \hbox{ with } \ |\eta| = |\nu - \hat{\xi}| + O(|\nu - \hat\xi|^2)
     \end{equation}
     It may be helpful, although it is not essential, to note that this is just the minus of the inverse of the exponential map $\exp_{\hat\xi} : T_{\hat\xi}S^{d-1} \to S^{d-1} \setminus \{- \hat\xi\}$ where $T_{\hat\xi}S^{d-1}$ is the tangent space to $S^{d-1}$ at $\hat\xi$.  The goal of this section is to show that, after moving to the interior and rescaling, the cell problem solution $v_\nu$ is very close, in terms of $|\nu - \hat{\xi}|$, to $w_{\xi,\eta}$ solving,
\begin{equation}\label{eqn: int eqn}
\left\{
\begin{array}{lll}
F(D^2w_{\xi,\eta},|\eta|^{-1}y) = 0 & \hbox{ in } & P_\xi \vspace{1.5mm} \\
w_{\xi,\eta} = {m}_\xi ( y \cdot \hat{\eta}) & \hbox{ on } & \partial P_\xi,
\end{array}\right.
\end{equation}
in their common domain of definition. We do not claim that \eqref{eqn: int eqn} has a boundary layer tail. Indeed the periodicity lattices of the boundary data and the operator may not be aligned. On the other hand, by Lemma \ref{lem: exp rate almost per}, it will almost have a limit up to an error small in $|\eta|$ and this will be sufficient for our purposes. More precisely we aim to prove:
\begin{figure}[t]
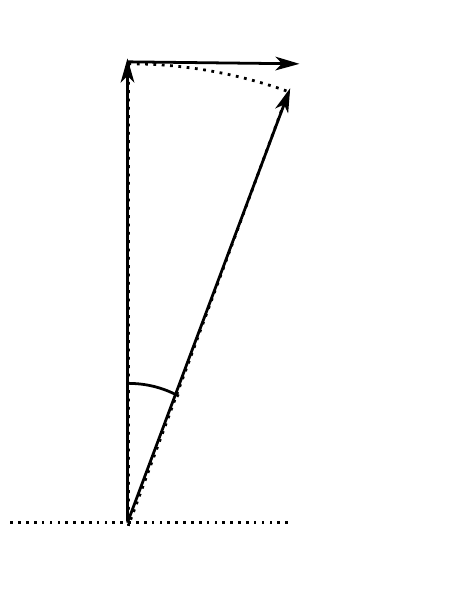%
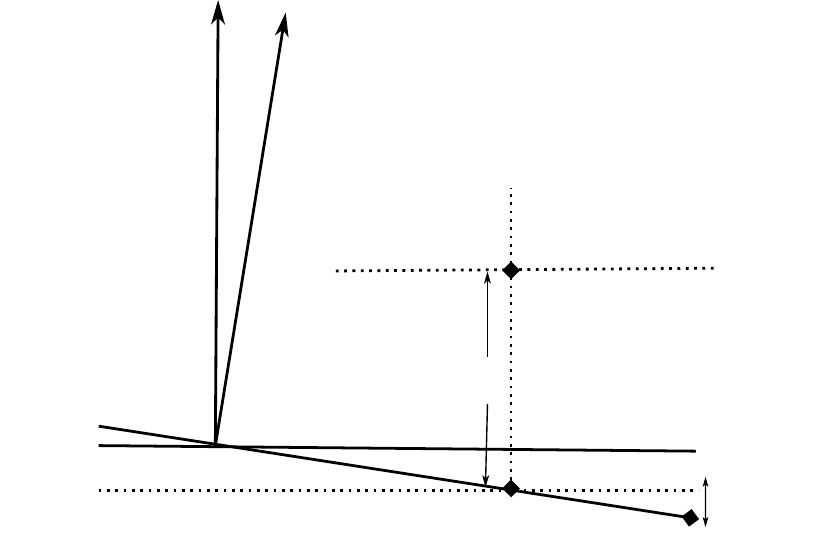%
\caption{The appearance of $m_\xi$ at an intermediate distance from $\partial P_\nu$.}
\label{fig: 2d reduction}
\end{figure}

\begin{prop}\label{lem: reduction 1}
Let $\xi \in \integer^d \setminus \{0\}$, irreducible, and $\nu \in S^{d-1}\setminus \real \integer^d$ with $|\xi-|\xi|\nu | \leq 1/2$ then, for any $\beta \in (0,1)$,
$$ \mu(\nu,F,\psi) - \liminf_{R \to \infty} w_{\xi,\eta}(R\hat{\xi};(\psi,F)) \leq C(\Lambda,d,\beta) |\psi|_{C^{0,\beta}}|\xi-|\xi|\nu|^\beta \log\tfrac{1}{|\xi-|\xi|\nu|}.$$
The parallel statement holds for the $\limsup$ as well.
\end{prop}

  We remark that the result does not depend on the H\"{o}lder continuity of $\psi$ (any continuity modulus for $\psi$ would yield an analogous result). Furthermore when $\psi \in C^{1,\beta}$ the estimate can be improved to
$$ \mu(\nu,F,\psi) - \liminf_{R \to \infty} w_{\xi,\eta}(R\hat{\xi};(\psi,F)) \leq C(\Lambda,d,\beta) \|\psi\|_{C^{1,\beta}}|\xi-|\xi|\nu| \log\tfrac{1}{|\xi-|\xi|\nu|}.$$

Let us give a heuristic proof of Proposition~\ref{lem: reduction 1}, which is illustrated in Figure 1.  Pick a point $y_0\in \partial P_{\nu}$. In a neighborhood of $y_0$ the boundary data $\psi$ for $v_{\nu}$ is very close to that of $\psi$ restricted to $\partial P_{\xi}+(y_0 \cdot \hat\xi)\hat\xi$. This causes $v_{\nu}$ to be close to $m_\xi(y_0\cdot\hat\xi)$ at $y_1:= y_0 + R\hat{\xi}$ for $R = o(|\hat\xi -\nu|^{-1})$. Next, observe that $y_0\cdot\hat{\xi} \sim y_0\cdot\eta$, since $\hat\xi -\eta$ is almost $\nu$ and $y_0$ is perpendicular to $\nu$.  But now, since $\eta$ is perpendicular to $\xi$, we have $y_0\cdot\eta = y_1\cdot\eta$. Consequently one can now say $v_{\nu}(y)$ is now close to $m_{\xi}(y\cdot\eta)$ $R$-away from $\partial P_{\nu}$, and this describes the near-boundary homogenization for $v_{\nu}$. Now taking $m_{\xi}(y\cdot\eta)$ as the new boundary data for the interior homogenization, we arrive at the interior problem \eqref{eqn: int eqn} and  Proposition~\ref{lem: reduction 1}.

\medskip

 The actual proof is slightly more involved for technical reasons.

   \begin{lem}\label{lemma_1}
Let $\xi \in \integer^d\setminus \{0\}$ be irreducible. For $|\xi-|\xi|\nu | \leq 1/2$  and  $\beta \in (0,1)$ let us define $R_0 := c^{-1}|\xi| \log \frac{1}{|\xi||\nu - \hat{\xi}|}$. Then we have
$$
  \sup_{y \in \partial P_\nu}|v_\nu(R_0\nu+y) - {m}_\xi( y \cdot \hat{\xi})| \leq C(\Lambda,d,\beta) |\psi|_{C^{0,\beta}}|\xi-|\xi|\nu|^\beta \log\tfrac{1}{|\xi-|\xi|\nu|}.
  $$
  \end{lem}
We remark that the $\log$ term in above esimate can be improved slightly as may be noticed from the proof.

\medskip

\textit{Proof of Lemma \ref{lemma_1}. } We first show that for all $R>1$,
$$\sup_{y \in \partial P_\nu} |v_\nu(R\nu+y) - {m}_\xi( y \cdot \hat{\xi})| \lesssim (\osc \psi)\exp(-cR/|\xi|)+|\psi|_{C^{0,\beta}}R^\beta|\nu-\hat{\xi}|^\beta. $$
This will imply the desired result by choosing $R_0 = c^{-1}|\xi| \log \frac{1}{|\xi||\nu - \hat{\xi}|}$ and using $(\osc \psi) \lesssim |\psi|_{C^{0,\beta}}$.  Fix $y_0 \in \partial P_\nu$ and we consider comparing $v_\nu$ with the solution $w$ of,
\begin{equation}
\left\{
\begin{array}{lll}
F(D^2w,y) = 0 & \hbox{ in } & P_{\xi}+y_0 \\
w = \psi(y) & \hbox{ in } & \partial P_{\xi}+y_0.
\end{array}
\right.
\end{equation}
Note that $ y_0 \in \partial P_\xi + (y_0 \cdot \hat{\xi}) \hat{\xi}$ and therefore, using that $ \nu \cdot \hat{\xi} \geq 1/2$,
\begin{equation}\label{eqn: w lim}
 |w(y_0+R\nu) - {m}( y_0 \cdot \hat{\xi})| \leq C(\osc \psi) \exp{(-c|\xi|^{-1}R)}.  
 \end{equation}
On the other hand for $y \in \partial [(P_\xi + y_0) \cap P_\nu]$ there exists $y' \in \partial P_\xi + y_0$ such that
$$ |y'-y| \leq |\nu - \hat{\xi}||y-y_0|,$$
and so by the H\"{o}lder continuity of $w$ up to the boundary,
$$ |w(y) - \psi(y)| \leq |w(y) - w(y')| + |\psi(y')-\psi(y)| \leq C(d,\Lambda)| \psi|_{C^{0,\beta}} |\nu - \hat{\xi}|^\beta|y-y_0|^\beta.$$
The same argument holds for $v_\nu$ and by combining the two estimates we have
$$ |v_\nu(y) - w(y)| \leq \min\{C| \psi|_{C^{0,\beta}} |\nu - \hat{\xi}|^\beta|y-y_0|^\beta , \osc \psi\} \ \hbox{ for } \ y \in \partial [(P_\xi + y_0) \cap P_\nu],$$
where the second term is from maximum principle, $ \min \psi \leq w,v_\nu \leq \max \psi$.  Now we claim that,
\begin{equation}\label{eqn: int diff}
 |v_\nu(y)-w(y)| \leq C(\Lambda,d,\beta)\min\{| \psi|_{C^{0,\beta}} |\nu - \hat{\xi}|^\beta|y-y_0|^\beta , \osc \psi\} \ \hbox{ for } \ y \in (P_\xi + y_0) \cap P_\nu,
 \end{equation}
but this is just a rescaling of Lemma \ref{bdry cont}.  In particular \eqref{eqn: int diff} combined with \eqref{eqn: w lim} implies,
\begin{align*}
 |v_\nu(R\nu+y_0)- {m}(y_0 \cdot \hat{\xi})|&\leq |v_\nu(R\nu+y_0)-w(R\nu+y_0)|+C|\psi|_{C^{0,\beta}}R^\beta |\nu - \hat{\xi}|^\beta \\&\lesssim (\osc \psi)\exp(-cR/|\xi|)+|\psi|_{C^{0,\beta}}R^\beta |\nu - \hat{\xi}|^\beta.
 \end{align*}
 This was the desired estimate.
\qed

\medskip

Next we return to the proof of Proposition \ref{lem: reduction 1} from Lemma \ref{lemma_1}.

\medskip

\textit{Proof of Proposition \ref{lem: reduction 1}.}  By maximum principle in the domain $P_\nu + R_0\nu$ Lemma \ref{lemma_1} implies that,
\begin{equation}\label{eqn: est u}
|v_\nu(y+R_0\nu) - u(y)| \lesssim   |\psi|_{C^{0,\beta}}|\xi|^\beta|\nu-\hat{\xi}|^\beta \log\tfrac{1}{|\xi||\nu - \hat{\xi}|} \ \hbox{ for } \ y \in P_\nu 
\end{equation}
where $u$ solves,
\begin{equation*}
\left\{
\begin{array}{lll}
F(D^2u,y) = 0 & \hbox{ in } & P_\nu \vspace{1.5mm} \\
u(y) = {m}_\xi ( y \cdot \hat{\xi}) = m_\xi(y\cdot \Pi_{\nu^\perp}\hat{\xi}) & \hbox{ on } & \partial P_\nu,
\end{array}\right.
\end{equation*}
where we recall that $\Pi_{\nu^\perp}\hat\xi := \hat \xi - (\hat \xi \cdot \nu)\nu$ is the orthogonal projection onto $\partial P_\nu$.  In particular an estimate of the same form as \eqref{eqn: est u} holds for the respective boundary layer tails.  Call $\eta_0 = \Pi_{\nu^\perp}\hat{\xi}$ and recall that we had defined $\eta$ as 
$$ \eta :=  -\exp_{\hat\xi}^{-1}(\nu)    \ \hbox{ defined so that } \ \nu = (\cos |\eta|) \hat\xi - (\sin |\eta|) \hat\eta .$$
From the definition of $\eta$ we calculate,
$$ \eta_0 = \Pi_{\nu^\perp}\hat{\xi} = \hat\xi - (\hat \xi \cdot \nu) \nu = (\sin |\eta|)^2 \hat \xi + ( \sin |\eta| )(\cos|\eta|)\hat \eta,$$
and so,
\begin{align*}
 |\eta_0 - \eta| &\leq |\sin |\eta| - |\eta||+|\eta_0 - (\sin|\eta|) \hat\eta| \\
 & \leq |\sin |\eta| - |\eta||+(\sin|\eta|)\sqrt{2(1-\cos|\eta|)} \\
 & \leq |\eta|^2
 \end{align*}
Now we rescale to $\tilde{u}(z) = u(|\eta|^{-1}z)$ which solves
\begin{equation*}
\left\{
\begin{array}{lll}
F(D^2\tilde{u},|\eta|^{-1}z) = 0 & \hbox{ in } & P_\nu \vspace{1.5mm} \\
\tilde{u}(z) =  m_\xi(z\cdot |\eta|^{-1}\eta_0) & \hbox{ on } & \partial P_\nu,
\end{array}\right.
\end{equation*}
and estimate the difference of $\tilde{u}$ and $w_{\xi,\eta}$ in their common domain $P_\nu \cap P_\xi$.  The aim is to obtain an estimate on the difference of their respective boundary layer tails.  From here the proof will follow a familiar argument. From the estimates above and Lemma~\ref{est:m},
$$ |m_\xi(z \cdot |\eta|^{-1}\eta_0) - m_\xi(z \cdot \hat \eta)| \leq  | m_\xi|_{C^{0,\beta}} |z|^\beta |\eta|^\beta.$$
 Using this we bound the difference $\tilde{u} - w_{\xi,\eta}$ for $z \in \partial (P_\nu \cap P_\xi)$. First  note that for $z \in \partial P_\nu \cap P_\xi$ there is $z' \in \partial P_\xi$ with $|z'-z| = |z||\eta| $. Therefore we have
\begin{align*}
 |\tilde{u}(z) - w_{\xi,\eta}(z)| &\leq |m_\xi(z \cdot |\eta|^{-1}\eta_0) - m_\xi(z \cdot \hat \eta)|+|w_{\xi,\eta}(z)-m_\xi(z \cdot \hat \eta)|\\
 & \leq |m_\xi|_{C^{0,\beta}} |z|^\beta |\eta|^\beta+|w_{\xi,\eta}(z)-m_\xi(z' \cdot \hat \eta)|+|m_\xi(z' \cdot \hat \eta)-m_\xi(z \cdot \hat \eta)| \\
 & \leq C|m_\xi|_{C^{0,\beta}} |z|^\beta |\eta|^\beta
 \end{align*}
where the middle term in the second line is estimated using the continuity up to the boundary of $w_{\xi,\eta}$ from Lemma \ref{bdry cont}.  Combining this with $\osc m_\xi \leq | m_\xi|_{C^{0,\beta}} |\xi|^{-\beta}$, and  $| m_\xi|_{C^{0,\beta}} \leq C(\Lambda,d)|\psi|_{C^{0,\beta}}$ we have 
\begin{equation*}
\left\{
\begin{array}{lll}
-\mathcal{P}^+_{1,\Lambda}(D^2(\tilde{u}-w_{\xi,\eta})) \leq 0 & \hbox{ in } & P_\nu \cap P_\xi \vspace{1.5mm} \\
(\tilde{u}-w_{\xi,\eta})(z) \leq C(\Lambda,d)| \psi|_{C^{0,\beta}}\min\{|z|^\beta |\nu - \hat{\xi}|^\beta,|\xi|^{-\beta}\}  & \hbox{ on } & \partial( P_\nu \cap P_\xi).
\end{array}\right.
\end{equation*}
Therefore by the rescaled version of Lemma \ref{bdry cont}, 
\begin{equation}\label{equation tilde u}
\tilde{u}(R\hat{\xi})- w_{\xi,\eta}(R\hat{\xi}) \leq C(d,\Lambda,\beta) |\psi|_{C^{0,\beta}}  R^{\beta}|\nu - \hat{\xi}|^\beta \hbox{ for any } R>0.
\end{equation}
At this stage we want to combine this estimate with the exponential convergence of $\tilde{u},w_{\xi,\eta}$ to their respective boundary layer tails, but there is a minor technical issue that $F(M,|\eta|^{-1}z)$ does not share the same periodicity lattice as $m_\xi( z \cdot \eta)$.  However, the conditions of Lemma \ref{lem: exp rate almost per} do hold and the rate of convergence established in Lemma \ref{lem: exp rate almost per} combined with \eqref{equation tilde u} implies, for any $R>0$,
$$ \limsup_{R' \to \infty} \tilde{u}(R'\hat{\xi})- \liminf_{R'\to\infty} w_{\xi,\eta}(R'\hat{\xi}) \lesssim (\osc m_\xi)\exp(-c|\xi|R)+|\psi|_{C^{0,\beta}} R^\beta |\nu - \hat{\xi}|^\beta+|\eta|^\beta |\psi|_{C^{0,\beta}}.$$
Now we are free to minimize over $R>0$, then plugging in $|\eta| \leq |\nu - \hat\xi|$ to obtain
$$ \limsup_{R\to\infty} \tilde{u}(R\hat{\xi})- \liminf_{R\to \infty} w_{\xi,\eta}(R\hat{\xi}) \leq C|\psi|_{C^{0,\beta}} |\nu - \hat\xi|^\beta\log\tfrac{1}{|\nu - \hat\xi|}.$$
Finally combining with \eqref{eqn: est u} and the remark below it that the same estimate holds for the boundary layer tails,
$$ \mu(\nu,\psi,F)- \liminf w_{\xi,\eta}(R\hat{\xi}) \leq C(\Lambda,d,\beta)|\psi|_{C^{0,\beta}} |\xi|^\beta|\nu - \hat\xi|^\beta\log\tfrac{1}{|\xi||\nu - \hat\xi|}.$$
A symmetric argument yields the same estimate for $ \limsup_{R \to \infty} w_{\xi,\eta}(R\hat{\xi})-\mu(\nu,\psi,F)$.

\qed

\medskip

\textit{Proof of Lemma~\ref{est:m}. }
For the purposes of this proof it will be useful to work with a slightly different definition of the cell problem solution.  We call $\tilde{v}_{\xi,\tau}(y)= v_{\xi,\tau}(y-\tau)$ which now solves,
\begin{equation}\label{eqn: cell prob12}
\left\{
\begin{array}{lll}
F(D^2\tilde{v}_{\xi,\tau},y) = 0 & \hbox{ in } & P_{\xi}+\tau \vspace{1.5mm} \\
\tilde{v}_{\xi,\tau} = \psi(y) & \hbox{ on } & \partial P_{\xi}+\tau.
\end{array}\right.
\end{equation}
Of course the boundary layer tail remains unchanged.  The point is that the $\tilde{v}_{\xi,\tau}$ now solve the same interior equation for all $\tau \in \real^d$, but in different domains.  When $\tau-\tau'$ is small the domains are close and we can combine the boundary continuity estimate of Lemma \ref{bdry cont} with comparison principle Lemma~\ref{lem: comparison half space} to estimate the difference of the cell problem solutions, and hence of their boundary layer tails as well.  It suffices to estimate the continuity of $m_\xi$ at $t=0$.  Let $\beta \in (0,1)$, $\e>0$ and any $y \in \partial P_\xi$,
\begin{align*}
 \tilde{v}_{\xi,-\e\hat\xi}(y)-\tilde{v}_{\xi,0}(y) &= \tilde{v}_{\xi,-\e\hat\xi}(y) - \psi(y) \\
 &= \tilde{v}_{\xi,-\e\hat\xi}(y)-v_{\xi,-\e\hat\xi}(y-\e\hat\xi)+\psi(y-\e\hat\xi)-\psi(y) \\
 & \leq \e^\beta(\sup_{\tau}\|\tilde{v}_{\xi,\tau}\|_{C^{0,\beta}(P_\xi)}+\|\psi\|_{C^{0,\beta}(\mathbb{T}^d)}) \\
 & \leq C(d,\Lambda,\beta)\e^{\beta}\|\psi\|_{C^{0,\beta}(\mathbb{T}^d)}.
 \end{align*}
Then, by maximum principle the same estimate holds in $P_\xi$ and therefore,
\begin{align*}
|m_\xi(-\e) - m_\xi(0)| &= \lim_{R \to \infty} |\tilde{v}_{\xi,-\e\hat\xi}(R\hat \xi)-\tilde{v}_{\xi,0}(R\hat\xi)| \\
&\leq \sup_{P_\xi}|\tilde{v}_{\xi,-\e\hat\xi}(R\hat \xi)-\tilde{v}_{\xi,0}(R\hat\xi)| \\
&\leq C(d,\Lambda,\beta)\e^{\beta}\|\psi\|_{C^{0,\beta}(\mathbb{T}^d)}.
\end{align*}
Parallel  arguments work for $\e <0$.  To get the Lipschitz estimate  use the fact that
$$ \sup_{\tau}\|\tilde{v}_{\xi,\tau}\|_{C^{0,1}(P_\xi)} \leq C(d,\Lambda,\beta)\|\psi\|_{C^{1,\beta}(\mathbb{T}^d)} \hbox{ for any } \beta>0.$$
\qed

\subsection{Step 2: Interior Homogenization at the Intermediate Scale}\label{step 2}
 From the reduction performed in the first step we are left to consider the following problem.  For an $\eta \perp \xi$ with $|\eta|>0$ small,
\begin{equation}\label{eqn: int level hom}
\left\{
\begin{array}{lll}
F(D^2w_{\xi,\eta},\tfrac{y}{|\eta|}) = 0 & \hbox{ in } & P_\xi \vspace{1.5mm} \\
w_{\xi,\eta} = {m}_\xi ( y \cdot \hat{\eta}) & \hbox{ on } & \partial P_\xi,
\end{array}\right.  \hbox{ which homogenizes to } 
\left\{
\begin{array}{lll}
\overline{F}(D^2\overline{w}_{\xi,\hat\eta}) = 0 & \hbox{ in } & P_\xi \vspace{1.5mm} \\
\overline{w}_{\xi,\hat\eta} = {m}_\xi ( y \cdot \hat{\eta}) & \hbox{ on } & \partial P_\xi.
\end{array}\right.
\end{equation}
 We wish to make this convergence quantitative so that we can get an estimate of the difference between $\mu(\nu,F,\psi)$ and boundary layer tail of the homogenized problem in \eqref{eqn: int level hom}. 

\medskip

At this stage it is useful to note that the homogenized solution $\overline{w}_{\xi,\hat\eta}$ is actually two dimensional.  The key observation here is that since the boundary data only varies in the $\hat\eta$ direction and the homogenized operator is translation invariant, the solution $\overline{w}_{\xi,\hat\eta}$ only varies in the $\hat\eta,\hat{\xi}$ directions.  This is a simple consequence of uniqueness.  
\begin{Claim}
$
\overline{w}_{\xi,\hat\eta}(x) \hbox{ only depends on } x\cdot\hat\xi \hbox{ and }  x\cdot\hat\eta.
$
\end{Claim}
We prove the claim only to emphasize the importance of passing from $w_{\xi,\eta}$ to $\overline{w}_{\xi,\hat\eta}$. 
\begin{proof}
For any $\zeta \perp \hat\eta,\hat\xi$ and $t \in \real$ note that $ \overline{w}'=\overline{w}_{\xi,\hat\eta}(y+t\zeta)$ solves
$$ \overline{F}(D^2\overline{w}')  = 0 \ \hbox{ in } \ P_\xi+\zeta = P_\xi \ \hbox{ with } \ \overline{w}'(y) = m((y+t\zeta) \cdot \hat\eta) = m(y \cdot \hat\eta) \ \hbox{ on } \ \partial P_\xi.$$
This is of course the same equation satisfied by $\overline{w}_{\xi,\hat\eta}$ so by the uniqueness of bounded solutions $$\overline{w}_{\xi,\hat\eta}(y+t\zeta) = \overline{w}_{\xi,\hat\eta}(y).$$
\end{proof}
In particular we have reduced to a situation where, by Nirenberg's Theorem, the homogenized solution is $C^{2,\alpha_0}$ on the interior.   By using the exponential rate of convergence to the boundary layer tail established in Section~\ref{sec: periodic bc} combined with Theorem~\ref{thm: 2d hom} we are able to show, up to a logarithmic factor, that the same rate of convergence holds for \eqref{eqn: int level hom}.  
\begin{lem}\label{lem: int hom unbd}Let $\eta \perp \xi$ with $|\eta||\xi| \leq 1/2$. Then there is $\alpha(\Lambda) \in (0,1)$ such that for any $\beta \in (0,1)$,
$$ |w_{\xi,\eta}(y) - \overline{w}_{\xi,\hat\eta}(y)| \leq C(\Lambda,d)|\psi|_{C^{0,\beta}(\mathbb{T}^d)}|\xi|^{\beta(\alpha-1)}|\eta|^{\alpha\beta}(\log\tfrac{1}{|\xi||\eta|}),$$
and, in particular, the same estimate holds between the boundary layer tail of $\overline{w}_{\xi,\hat\eta}$ and $\liminf_{R \to \infty} w_{\xi,\eta}(R\hat{\xi})$ or $\limsup_{R \to \infty} w_{\xi,\eta}(R\hat{\xi})$.
\end{lem}
Before we proceed with the proof we state a consequence of Proposition~\ref{lem: reduction 1} and Lemma~\ref{lem: int hom unbd}.

\begin{lem}\label{lem: step 2 summary}
Let $\xi \in \integer^d\setminus \{0\}$ irreducible and $\nu$ an irrational direction with $\eta = \eta(\nu)$ as in \eqref{eqn: eta def}. Then there is $\alpha(\Lambda) \in (0,1)$ such that for any $\beta \in (0,1)$,
$$
 |\mu(\nu,\psi,F) - \lim_{R \to \infty} \overline{w}_{\xi,\eta}(R\hat\xi)| \leq C(\Lambda,d,\beta)|\psi|_{C^{0,\beta}(\mathbb{T}^d)}|\xi|^{\alpha\beta}|\eta|^{\alpha\beta}.$$
\end{lem}

\medskip

\textit{Proof of Lemma \ref{lem: int hom unbd}.}  We would like to apply Theorem~\ref{thm: 2d hom} to prove the Lemma, however a modified argument is necessary since $w_{\xi,\eta}$ and $\overline{w}_{\xi,\hat\eta}$ are solutions in an entire half-space. In order to replace with a homogenization problem in a bounded width strip we use that $\overline{w}_{\xi,\hat\eta}(R\hat\xi)$ converges with exponential rate to its boundary layer tail $\overline{\mu}$ and $w_{\xi,\eta}(R\hat\xi)$, although it does not quite have a boundary layer tail, converges with exponential rate to a neighborhood of width small in $|\eta|$ centered at any of its subsequential limits $\mu$ (see Lemma \ref{lem: exp rate almost per}).  


\medskip

Now we begin with the technical details of the proof.  First recall that ${m}_\xi$ is $\frac{1}{|\xi|}$-periodic on $\real$.  Therefore ${m}_\xi( y \cdot \hat{\eta})$ is $\frac{1}{|\xi|}$-periodic on $\partial P_\xi$ in the direction $\hat{\eta}$ and constant in the directions orthogonal to $\hat{\eta}$.  Due to Lemma \ref{est:m} we can estimate
 $$\osc m_\xi \leq |\xi|^{-\beta}|m_\xi|_{C^{0,\beta}(\real)} \leq C|\xi|^{-\beta}|\psi|_{C^{0,\beta}},$$  
 and 
 $$|m_\xi|_{C^{0,\alpha\beta}(\real)} \leq |\xi|^{\beta(\alpha-1)}|m_\xi|_{C^{0,\beta}(\real)}.$$
 
 Next let $\mu $ and $\overline{\mu}$ respectively denote any subsequential limit of $w_{\xi,\eta}(R\xi)$ and the limit of $\overline{w}_{\xi,\hat\eta}(R\xi)$ as $R\to\infty$, and let $\alpha(\Lambda)$ be as given in Theorem \ref{thm: 2d hom}. Then from Lemma \ref{lem: exp rate almost per} we have
\begin{equation}\label{eqn: w err 1}
 |w_{\xi,\eta}(y) - \mu|+ |\overline{w}_{\xi,\eta}(y)-\overline{\mu}| \leq C(\osc m_\xi) \exp(-c|\xi|R)+C|m_\xi|_{C^{0,\alpha\beta}}|\eta|^{\alpha\beta} \ \hbox{ for } \ y \in  P_\xi + R\hat{\xi}.
 \end{equation}
We use \eqref{eqn: w err 1} to restrict to a domain where we can use Theorem \ref{thm: 2d hom}, then we simultaneously are able to estimate $\mu-\overline{\mu}$ and $w_{\xi,\eta}-\overline{w}_{\xi,\eta}$.   Fix an $R$ to be chosen and consider,
\begin{equation}\label{eqn: tilde w defn}
\tilde{w}_{\xi,\eta}(y) := w_{\xi,\eta}(y) +R^{-1}y \cdot \hat{\xi} \left [ (\overline{\mu}-\mu) + \sup_{\partial P_\xi+R\hat{\xi}}\big[|w_{\xi,\eta}(\cdot) - \mu|+ |\overline{w}_{\xi,\eta}(\cdot)-\overline{\mu}|\big] \right ]. 
\end{equation}
Note that with this modification $\tilde{w}_{\xi,\eta}$ still solves the same equation as $w_{\xi,\eta}$ in $P_\xi$ with the same boundary condition on $\partial P_\xi$ but also
\begin{equation}\label{eqn: bdry order hom}
  \tilde{w}_{\xi,\eta}(y)\geq \overline{w}_{\xi,\eta}(y) \ \hbox{ on } \ \partial P_\xi + R\hat{\xi}. 
 \end{equation}
Now Theorem \ref{thm: 2d hom} implies that
\begin{equation} \label{eqn: not quite}
\overline{w}_{\xi,\eta}(y) - \tilde{w}_{\xi,\eta}(y) \leq C(|\xi|^{-\beta}+R^\beta)(R^{-1}|\eta|)^{\alpha\beta}|m_\xi|_{C^{0,\beta}(\real)}.
\end{equation}
Note that we are not quite applying Theorem~\ref{thm: 2d hom} directly. To be precise we first solve the equation $F(D^2u,\frac{y}{|\eta|}) = 0$ with boundary data matching $\overline{w}_{\xi,\eta}$ in $ 0 < y \cdot \hat \xi < R$.  By comparison principle and the ordering \eqref{eqn: bdry order hom} we know $u \leq \tilde{w}_{\xi,\eta}$.  On the other hand from Theorem~\ref{thm: 2d hom} we have the desired estimate for $|u-\overline{w}_{\xi,\eta}|$, combining these two steps we get \eqref{eqn: not quite}.  

\medskip

Rewriting \eqref{eqn: not quite} in terms of ${w}_{\xi,\eta}$ using \eqref{eqn: w err 1},
\begin{equation*}\label{eqn: w est 0}
 \overline{w}_{\xi,\eta}(y) - w_{\xi,\eta}(y) \leq (\overline{\mu}-\mu)R^{-1}y \cdot \hat{\xi}+C|\psi|_{C^{0,\beta}}|\xi|^{-\beta}[(R^{-1}|\eta|)^{\alpha\beta}(1+|\xi|^\beta R^\beta) +  \exp(-c|\xi|R)+ |\xi|^{\alpha\beta}|\eta|^{\alpha\beta}].
\end{equation*}
Let us choose $R = 2(c|\xi|)^{-1}\log\frac{1}{|\xi||\eta|}$ to obtain
\begin{equation}\label{eqn: w est opt 0}
 \overline{w}_{\xi,\eta}(y) - w_{\xi,\eta}(y) \leq (\overline{\mu}-\mu)R^{-1}y \cdot \hat{\xi}+C|\psi|_{C^{0,\beta}}|\xi|^{\beta(\alpha-1)}|\eta|^{\alpha\beta}(\log \tfrac{1}{|\xi||\eta|}).
\end{equation}
This implies an estimate for $\overline{\mu} - \mu$ as well by evaluating for $y \in \partial P_\xi + \tfrac{1}{2}R\hat{\xi}$:
$$\overline{\mu}-\mu \leq \tfrac{1}{2}(\overline{\mu}-\mu)+C|\psi|_{C^{0,\beta}}|\xi|^{\beta(\alpha-1)}|\eta|^{\alpha\beta}(\log \tfrac{1}{|\xi||\eta|}).
$$
Here we have used \eqref{eqn: w err 1} to estimate $\mu - w_{\xi,\eta}$ and $\overline{\mu}-\overline{w}_{\xi,\eta}$ on $\partial P_\xi + \tfrac{1}{2}R\hat{\xi}$, the error is of the same order as in \eqref{eqn: w est opt 0} so we combined terms.  Rearranging the last inequality and making a similar argument for the lower bound, we conclude that
\begin{equation}\label{eqn: mu est 0}
 |\overline{\mu}-\mu| \leq  C|\psi|_{C^{0,\beta}}|\xi|^{\beta(\alpha-1)}|\eta|^{\alpha\beta}(\log \tfrac{1}{|\xi||\eta|}).
 \end{equation}
But now we can plug \eqref{eqn: mu est 0} back into \eqref{eqn: w est opt 0} and obtain for any $0< y \cdot \hat \xi < R$, 
\begin{equation}\label{eqn: w est 2}
 |\overline{w}_{\xi,\eta}(y) - w_{\xi,\eta}(y)| \leq C|\psi|_{C^{0,\beta}}|\xi|^{\beta(\alpha-1)}|\eta|^{\alpha\beta}(\log \tfrac{1}{|\xi||\eta|}),
\end{equation}
the same estimate is obtained for $y\cdot \hat \xi \geq R$ by using \eqref{eqn: mu est 0} in combination with \eqref{eqn: w err 1}.  Thus we obtain \eqref{eqn: w est 2} for all $y \in P_\xi$.

\qed

\subsection{Step 3: Reduction to a two-dimensional Problem}\label{step 3}
The third step of our reduction procedure is actually more of notation change.  Let $\bar{F}$ be a homogeneous, uniformly elliptic operator. We are concerned with the solution of, 
\begin{equation}
\left\{
\begin{array}{lll}
\overline{F}(D^2\overline{w}_{\xi,\eta}) = 0 & \hbox{ in } & P_{\xi} \vspace{1.5mm} \\
\overline{w}_{\xi,\eta} = {m}_\xi ( y \cdot \eta) & \hbox{ on } & \partial P_{\xi}
\end{array}\right. 
\end{equation}
for a fixed unit vector $\eta \in S^{d-1}$ with $\eta \cdot \xi = 0$.

\medskip

\noindent In the previous section we have already observed that $\overline{w}_{\xi,\eta}$ varies only in the $\hat\xi,\eta$ directions.  To emphasize the two-dimensionality of $\overline{w}_{\xi,\eta}$  let us define $W_{\xi,\eta} : \real^2_+ \to \real$ by,
\begin{equation}\label{eqn: capital w}
 W_{\xi,\eta}(z) = \overline{w}_{\xi,\eta}(z_1 \eta+z_2 \hat{\xi}).
 \end{equation}
Now $W_{\xi,\eta}$ will solve an equation in the upper half space with an operator $G_{\eta,\xi}$ which is essentially the projection of $\overline{F}$ onto the $\xi$-$\eta$ plane.   Let $M \in \M_{2 \times 2}$ a symmetric $2 \times 2$ matrix, the definition of $G_{\xi,\eta}(M)$ is somewhat cumbersome in terms of notation but the idea is quite simple,
\begin{equation}\label{eqn: G op}
G_{\xi,\eta}(M) := \overline{F}( \sum_{1 \leq i ,j \leq 2} M_{ij} f_i \otimes f_j ) \ \hbox{ with } \ f_1 = \eta, \ f_2 = \hat\xi.
\end{equation}
It is quite important to note the dependence of $G$ on the orientation of $\eta$; $G_{\eta,\xi}$ may not be the same operator as $G_{-\eta,\xi}$. 
\begin{lem}
 Let $W_{\xi,\eta}$ and $G_{\xi,\eta}$ be as given in \eqref{eqn: capital w} and \eqref{eqn: G op}. Then $W_{\xi,\eta}(z)$ is the unique solution of 
\begin{equation}\label{eqn: G}
\left\{
\begin{array}{lll}
G_{\xi,\eta}(D^2_zW_{\xi,\eta}) = 0 & \hbox{ in } & \real^2_+ \vspace{1.5mm} \\
W_{\xi,\eta} = {m}_\xi (z_1) & \hbox{ on } & \partial \real^2_+.
\end{array}\right. 
\end{equation}
\end{lem}
\noindent The key point of this reduction is that we realize $\overline{w}_{\xi,\eta}$ as the solutions of \textit{different} pdes in the \textit{same} domain with the \textit{same} boundary conditions.

\subsection{Step 4: The directional limits of $\mu$ at rational directions}  

We are now ready to precisely characterize the limiting behavior of $\mu(\cdot,\psi, F)$ near a rational vector $\xi \in \integer^d \setminus \{0\}$ of $\mu(\cdot,\psi,F)$ in terms of the boundary layer tails of the class of simpler two dimensional problems \eqref{eqn: G}.

\begin{DEF}\label{lagrangian}
Let $F$ be a uniformly elliptic operator as given in Section 2. For $\psi\in C^{0,\beta}(\mathbb{T}^d)$, a rational vector $\xi \in \integer^d$ and a unit vector $\eta\perp \xi$ define,
$$ L_\xi(\eta; (\psi,F)) : = \lim_{R \to \infty} W_{\xi,\eta}(0,R;(\psi,F)).$$
\end{DEF}

Similar arguments to those used in the previous section will show that $L_\xi$ is continuous in $\eta$.  For example see the proof of Proposition~\ref{lem: reduction 1}.
\begin{lem}\label{lem: L eta reg}
For $\xi \in \integer^d \setminus\{0\}$, $\beta \in (0,1)$ and any $\eta, \eta '$ unit vectors orthogonal to $\xi$,
$$ |L_\xi(\eta;(\psi,F)) - L_\xi(\eta';(\psi,F))|\leq C(\Lambda,d,\beta)\|\psi\|_{C^{0,\beta}}|\eta- \eta'|^\beta(1+\log\tfrac{1}{|\eta-\eta'|}).$$
\end{lem}

A combination of the results of the previous sections yields the following classification of the asymptotic behavior of $\mu(\cdot,\psi,F)$ near $\xi$:
\begin{prop}\label{lem: asymptotics}
Let $\xi \in \integer^d\setminus \{0\}$ be irreducible and let $\nu : [0,1) \to S^{d-1}$ a geodesic path with unit speed and $\nu(0) = \hat{\xi} $.  Then there is $\alpha(\Lambda) \in (0,1)$ such that for any $\beta \in (0,1)$,
$$ | \mu(\nu(t),\psi,F) - L_{\xi}(\nu'(0);(\psi,F))| \leq C(\Lambda,d)|\psi|_{C^{0,\beta}(\mathbb{T}^d)}|\xi|^{\alpha\beta}t^{\alpha\beta}.$$
\end{prop}

\section{Continuity of $\mu$}\label{continuity}

One immediate consequence of Proposition~\ref{lem: asymptotics} is a continuity result analogous to Theorem $4.1$ of Choi and Kim \cite{ChoiKim13} for operators $F(M,y)$ such that $\overline{F}$ is rotation invariant.  Let us repeat that the proof we have given of this result is not new, rather we have made each of the steps \cite{ChoiKim13} quantitative and elucidated the secondary two-dimensional cell problem underlying the limiting behavior near rational directions. This additional work will be essential to the results that follow but is not so important just to get a continuous extension of $\mu(\cdot,\psi,F)$ to the rational directions without an explicit modulus.
\begin{thm}\label{thm: cont no mod}
Let $\xi \in \integer^d \setminus \{0\}$ be irreducible.  If $\overline{F}$ is invariant with respect to the rotations/reflections that preserve $\xi$ or $\overline{F}$ is linear, then $L_{\xi}(\cdot;(\psi,F)) \equiv L_\xi(\psi,F)$ independent of the approach direction. As a consequence, $\mu(\cdot,\psi,F)$ extends continuously to $\hat{\xi}$ with value $L_\xi(\psi,F)$ and,
$$ | \mu(\nu,\psi,F) - L_{\xi}(\psi,F)|\leq C(\Lambda,d,\beta)|\psi|_{C^{0,\beta}(\mathbb{T}^d)}|\xi-|\xi|\nu |^{\alpha\beta},$$
for some $\alpha(\Lambda) \in (0,1)$ and any $\beta \in (0,1)$.  In particular, if $\overline{F}$ is rotation invariant or linear, $\mu(\cdot,\psi,F)$ extends from $S^{d-1} \setminus \real \integer^d$ to a continuous function on $S^{d-1}$.
\end{thm}
\begin{proof}
It suffices to show $L_\xi$ is constant in the cases claimed, the rest of the Theorem will then follow from Proposition \ref{lem: asymptotics}. For any $\eta_1,\eta_2 \perp \xi$ let $O$ be a rotation sending $\eta_1$ to $\eta_2$ and holding $\xi$ fixed.  Now $\overline{w}_{\xi,\eta_1}(O^t \cdot)$ has the same boundary data as $\overline{w}_{\xi,\eta_2}(\cdot)$ and by the rotation invariance of $\overline{F}$ they solve the same pde in $P_\xi$. Thus by uniqueness they are equal.  In particular they have the same boundary layer tail so $L_\xi(\eta_1;(\psi,F)) = L_\xi(\eta_2;(\psi,F))$.

\medskip

In the second case we refer to Lemma~3.6 of \cite{Feldman13} which shows, using Riesz Representation Theorem, that when $F$ is linear and homogeneous $\mu(\nu,\psi,F) = \langle \psi \rangle$ (the average over the torus).  We apply this to $\overline{w}_{\xi,\eta}$ which satisfies the assumptions of the Lemma since it is a solution of $\overline{F}$ which is homogeneous and, by assumption, linear.  We derive for every $\eta \perp \xi$,
$$ L_\xi(\eta;(\psi,F)) = \lim_{R \to \infty} \overline{w}_{\xi,\eta}(R\hat\xi) = |\xi|\int_0^{1/|\xi|}m_\xi(t;(\psi,F)) dt.$$
The right hand side is independent of $\eta$ which was the desired result.
\end{proof}

As a corollary of Theorem \ref{thm: cont no mod} we will show an explicit modulus of H\"{o}lder continuity for the homogenized boundary condition when $\overline{F}$ is rotation invariant or linear. The argument is entirely number theoretic and relies on Dirichlet's Theorem, Theorem~\ref{Dirichlet}.  A sharper estimate in the linear case can be found in Section \ref{linear}. The improvement there is in the rate of convergence at a single rational direction. The argument using Dirichlet's Theorem stays the same.

\begin{cor}\label{cor: cont mod}
Let $F$ satisfying the assumptions of Theorem \ref{thm: cont no mod}.  There is $\alpha(\Lambda) \in (0,1)$ such that for all $\beta \in (0,1)$, $\psi \in C^{0,\beta}(\mathbb{T}^d)$, and $\nu_1$ and $\nu_2$ irrational vectors in $S^{d-1}$ we have 

$$  |\mu(\nu_1,\psi,F)-\mu(\nu_2,\psi,F) | \leq C(d,\Lambda,\beta)  |\psi|_{C^{0,\beta}}|\nu_1-\nu_2|^{\beta \alpha/d}. 
$$
\end{cor}

\begin{proof}
Assume $|\psi|_{C^{0,\beta}} \leq 1$, the general case follows from scaling.  Let $\e:= |\nu_1-\nu_2|$, and let $N = \e^{-(d-1)/d}$.  Then due to Lemma~\ref{Dirichlet} there exists $\xi\in \integer^d$ and $n\in\integer$ with $1\leq n\leq N$ such that 
$$
\left|n\frac{\nu_1}{|\nu_1|_\infty}-\xi\right| \leq (d-1)^{1/2} N^{-1/(d-1)}.
$$
Note that $n  \gtrsim |\nu_1|_\infty |\xi| \geq d^{-1/2}|\xi|$. Due to this and the choice of $N$ we have 
\begin{align*}
|\nu_2 - n^{-1}|\nu_1|_\infty \xi| &= |\nu_1-\nu_2| + \left|\nu_1-n^{-1}|\nu_1|_\infty \xi\right| \\
& \leq \e + (d-1)^{1/2} n^{-1}|\nu_1|_\infty N^{-1/(d-1)} \\
& \leq \e + C_d |\xi|^{-1}N^{-1/(d-1)}.
\end{align*}
Now we apply Theorem~\ref{thm: cont no mod} with $\nu=\nu_j$ at the rational direction $\xi$ to conclude that
$$|\mu(\nu_1) - \mu(\nu_2)| \lesssim N^{-\frac{\alpha\beta}{(d-1)}} + |\xi|^{\alpha\beta}\e^{\alpha\beta} .$$ 
Using that $|\xi| \lesssim N$ we obtain
$$|\mu(\nu_1) - \mu(\nu_2)| \lesssim N^{-\frac{\alpha\beta}{(d-1)}}\log N + (N\e )^{\alpha\beta} \lesssim \e^{\alpha\beta/d}.$$
\end{proof}

\section{Discontinuity of $\mu$}\label{sec: disc}

Given the set up of the previous sections it may seem at least plausible to the reader that when $F$ is nonlinear and not rotation invariant, for a given $\xi$ the directional limit function  $L_\xi$ will typically be non-constant, resulting in the discontinuity of the homogenized boundary data.  On the other hand it is not obvious, at least to the authors, how to prove that any specific pair $(\psi,F)$ results in a non-constant $L_\xi$.  Apart from explicitly computing the solutions the only way to differentiate the boundary layer tails of the $W_{\xi,\eta}$ would be to use maximum principle. However, except in some specially arranged cases, one cannot choose $\eta_1,\eta_2 \perp \xi$ so that $G_{\xi,\eta_1} \geq G_{\xi,\eta_2}$ and so there is no reason for $W_{\xi,\eta_j}$ to be ordered in the whole of $\real^2_+$ for any such pair $\eta_1,\eta_2$.  We instead find monotonicity by perturbing $(\psi,F)$.  We are then able to show that the class of $(\psi,F)$ for which $L_\xi(\cdot ; (\psi,F))$ is non-constant is open and dense in the appropriate topologies.

\medskip

Let us give a heuristic description of how this monotonicity arises.  The goal is to show that for any $(\psi,F)$ and $\xi\in\integer^d\setminus\{0\}$ we can find a nearby $(\psi',F')$ such that $L_\xi(\cdot;(\psi',F'))$ is non-constant.  In this paper we are only able to show that a small perturbation of $\overline{F}$ would lead to $L_\xi$ being non-constant, which directly corresponds to perturbation of the homogeneous operators since $F = \overline{F}$.  In the general case of inhomogeneous $F$ it is not clear to us whether it is possible to perturb $F$ to correspond to the desired perturbation of $\bar{F}$; we leave this as an open question.  Let us now describe the perturbation of homogeneous operators $F$. First note that we only need to perturb $(\psi,F)$ when $L_\xi(\cdot ; (\psi,F))$ is constant, otherwise we could take $(\psi',F') = (\psi,F)$. When $d \geq 3$ we can find two directions $\eta_1,\eta_2$ perpendicular both to each other and to $\xi$.  We then perturb $\overline{F}$ in a monotone and hence intrinsically nonlinear way, heuristically affecting the choice of diffusions in the $\eta_1$ direction while leaving the $\eta_2$ direction unchanged.  More concretely the perturbation will satisfy that $G_{\xi,\eta_1}' \gneq G_{\xi,\eta_1}$ while $G_{\xi,\eta_2}' = G_{\xi,\eta_2}$.  Then, up to a small perturbation of $\psi$, strong maximum principle will imply that  $W_{\xi,\eta_1}' < W_{\xi,\eta_1}$ and, since periodicity provides compactness in the lateral directions, also $L_{\xi}(\eta_1,(\psi',F')) < L_{\xi}(\eta_2,(\psi,F))$, while $L_\xi(\eta_2)$ remains unchanged.  Now, having assumed that $L_\xi(\cdot ; (\psi,F))$ was originally constant, $L_\xi(\cdot ; (\psi',F'))$ must be non-constant.

\medskip

The only natural notion of genericity in this setting, to our knowledge, is topological. We make precise the topological setting. Our boundary data will be taken from the space,
\begin{equation}\label{eqn: Ckb spaces}
 C(\mathbb{T}^d) = \{ \psi : \mathbb{T}^d \to \real \ \hbox{continuous}\} \ \hbox{ with the supremum norm}.
 \end{equation}
Let us next define the space of uniformly elliptic operators,
\begin{equation}\label{eqn: UE space}
 \textup{UE}_d = \{ F : \M_{d \times d} \to \real |  \ F \in \cup_{\Lambda >1 }\mathcal{S}_{1,\Lambda} \hbox{ uniformly elliptic and positively } 1\hbox{-homogeneous} \}.  
 \end{equation}
Here we recall that $\M_{d \times d}$ is the space of symmetric $d \times d$ matrices with real entries. For $F \in \textup{UE}_d$ we define the ellipticity ratio $\Lambda(F)$ to be the minimal $\Lambda>1$ such that $F \in \mathcal{S}_{1,\Lambda}$.  It is easy to check from this that $F \in \textup{UE}_d$ are Lipschitz continuous with Lipschitz constant $d \Lambda(F)$ with respect to the operator norm metric on $\M_{d\times d}$.  Conversely consider an $F$ which is Lipschitz continuous with respect to the operator norm metric on $\mathcal{M}_{d \times d}\simeq \real^{\frac{d(d+1)}{2}}$. For this $F$ the gradient $DF$, from standard inner product $\Tr(AB)$ on $d\times d$ matrices $A,B$, is defined Lebesgue almost everywhere. The Lipschitz constant of $F$ is $\|DF\|_{L^\infty(\mathcal{M}_{d\times d})}$ where we implicitly take the underlying matrix norm to be the dual of the operator norm. Based on this definition it is straightforward to check that $\Lambda(F) \leq \|DF\|_\infty$. We take as the metric on $\textup{UE}_d$,
$$
 d_{\textup{UE}_d}(F_1,F_2) : =   \sup_{\|M\| = 1}|F_1(M) - F_2(M)|+\|DF_1 - DF_2\|_{L^\infty(\M_{d\times d})}.
 $$
 Noting that Cauchy sequences have $\|DF_n\|_\infty$ bounded and hence $\Lambda(F_n)$ bounded we see that $(\textup{UE}_d,d_{\textup{UE}_d})$ is complete.  We draw our operator and boundary data $(\psi,F)$ from the space,
$$ X = C(\mathbb{T}^d) \times \textup{UE}_d \ \hbox{ with distance } \ d_X((\psi_1,F_1),(\psi_2,F_2)) = \sup|\psi_1 - \psi_2| + d_{\textup{UE}_d}(F_1,F_2), $$
which is a complete metric space.

\begin{thm}
Let $d \geq 3$ and $ \xi \in \integer^d\setminus \{0\}$. Then the set
$$ E_\xi = \{ (\psi,F) \in X \  |  \ \mu(\cdot,\psi,F) \ \hbox{ is discontinuous at } \hat{\xi} \} \ \hbox{ is open and dense in $X$.}$$
  In particular there is a residual set $E \subset X$, a countable intersection of open dense sets $E = \cap_{\xi \in \integer^d \setminus \{0\}} E_\xi$, such that for all $(\psi,F) \in E$,
$$ \mu ( \cdot,\psi,F) \ \hbox{ is discontinuous at \textbf{every} rational direction.}$$
\end{thm}

\vspace{10pt}

 The proof of the theorem consists of the following two steps. First we prove that $E_\xi$ is open.  The proof of Lemma \ref{lem: L cont} is more or less standard, and is due to comparison principle and the stability of viscosity solutions with respect to uniform convergence.

\begin{lem}\label{lem: L cont}
For each $\xi \in \integer^d$, $(\psi,F) \in X$, $L_\xi : \{\eta\in S^d, \eta\cdot\xi=0\} \times X \to \real $ is continuous with respect to $d_X$ at $(\psi,F)$, 
$$ \sup_{ \eta \in S^d, \eta\cdot\xi=0} |L_\xi( \eta ; (\psi',F')) - L_\xi( \eta ; (\psi,F))| \to 0 \ \hbox{ as } \ d_X((\psi',F'),(\psi,F)) \to 0.$$
In particular by Proposition \ref{lem: asymptotics} $E_\xi$ is open.
\end{lem}

Next  we will show that $E_\xi$ is dense, whose proof strongly depends on the conditions $d \geq 3$ and that $F$ is homogeneous.

\begin{prop}\label{lem: mu discont}
Let $d \geq 3$ and $\xi\in\integer^d$. Then  for given $(\psi,F) \in X$ and $\e>0$, there exists $(\psi_\e,F_\e)$ such that 
$$ d_{X}((\psi_\e,F_\e),(\psi,F)) \leq \e \ \hbox{ and } \ L_{\xi}( \cdot \  ; (\psi_\e,F_\e)) \ \hbox{ is non-constant.}$$
In particular $\mu(\cdot,\psi_\e,F_\e)$ is discontinuous at $\xi$ by Proposition \ref{lem: asymptotics}.
\end{prop}

\medskip

Now we proceed with the proofs. 

\medskip
  
\textit{Proof of Lemma \ref{lem: L cont}.} 
  Let $(\psi_n,F_n)$ be a sequence in $(X,d_X)$ converging to $(\psi,F)$.  Let us recall the definition of $L_\xi (\eta; (\psi_n, F_n))$ given in Definition~\ref{lagrangian}:

\begin{equation}
\left\{
\begin{array}{lll}
\overline{F}_n(D^2w_n) = 0 & \hbox{ in } & P_{\xi} \vspace{1.5mm}\\
w_n(y) = m_\xi(y \cdot \eta;(\psi_n,F_n)) & \hbox{ on } & \partial P_{\xi}
\end{array}\right. \ \hbox{ and } \ L_\xi(\eta;(\psi_n,F_n)) = \lim_{R \to \infty} w_n(R\xi).
\end{equation}

\medskip

     Since $F_n's$ are homogeneous,  $\overline{F}_n = F_n$  but we continue to write $\overline{F}_n$ to emphasize the correct definition of $L_\xi$.  We begin by first investigating the continuity properties of $m_{\xi}$. The claim is 
  $$
\sup_t |m_\xi(t;(\psi_n,F_n)) - m_\xi(t;(\psi,F))| \to 0 \hbox{ as } n\to\infty.
  $$
  
  Observe that by maximum principle,
 $$
  |m_\xi(\cdot;(\psi_n,F_n))- m_\xi(\cdot;(\psi,F_n))| \leq \|\psi_n - \psi\|_\infty.
  $$
Thus it remains to show that $\sup_t |m_\xi(t;(\psi,F_n)) - m_\xi(t;(\psi,F))| \to 0$.  The pointwise convergence with fixed $t$ is due to stability of viscosity solutions with respect to uniform convergence of $F_n$, but a little extra work is required to show that the convergence is uniform in $t$.  Note that by Lemma~\ref{est:m}, we have some modulus $\bar\omega$ depending on the continuity modulus $\omega$ of $\psi$ and $\Lambda(F_n)$ so that,
$$
 |m_\xi(t;(\psi,F_n)) - m_\xi(t';(\psi,F_n))| \leq \bar\omega(|t-t'|) \ \hbox{ and } \  |m_\xi(t;(\psi,F_n))| \leq \|\psi\|_\infty.$$
  Since $F_n \to F$ is a convergent sequence in $d_{\textup{UE}_d}$, $\|F_n\|_\infty$ and $\Lambda(F_n)$ are bounded.  Since $m_\xi(\cdot;(\psi,F_n))$ are uniformly bounded and equicontinuous $\frac{1}{|\xi|}$-periodic functions on $\real$, every subsequence has a uniformly convergent subsequence.  It follows that, since $m_\xi(\cdot; (\psi,F_n))$ converge pointwise to $m_\xi(\cdot; (\psi,F))$, they will also converge uniformly.  

\medskip

Now let us define $\tilde{w}_n$ to solve
 \begin{equation*}
\left\{
\begin{array}{lll}
\overline{F}_n(D^2\tilde{w}_n) = 0 & \hbox{ in } & P_{\xi};\vspace{1.5mm}\\
\tilde{w}_n(y) = m_\xi(y \cdot \eta;(\psi,F)) & \hbox{ on } & \partial P_{\xi}.
\end{array}\right. 
\end{equation*}
Since we have already proven that $m_\xi(y \cdot \eta;(\psi_n,F_n)) \to m_\xi(y \cdot \eta;(\psi,F))$ uniformly on $\partial P_\xi$, by comparison principle Lemma~\ref{lem: comparison half space},
\begin{align*}
 |\lim_{R \to \infty} \tilde{w}_n(R\xi) -L_{\xi}(\eta;(\psi_n,F_n)) |&\leq \sup_{y \in P_\xi} |w_n(y) - \tilde{w}_n(y)| \\
 & \leq \sup_{t\in \real}|m_\xi(t;(\psi_n,F_n)) - m_\xi(t;(\psi,F))| \to 0 \ \hbox{ as } \ n \to \infty.
 \end{align*}
By a similar argument as above, since $\overline{F}_n \to \overline{F}$ uniformly on compact sets of $\mathcal{M}_{d\times d}$ when $d_{\textup{UE}_{d}}(F_n, F) \to 0$, we have that $\tilde{w}_n \to w$ locally uniformly in $P_\xi$ and
$$ |\lim_{R \to \infty} \tilde{w}_n(R\xi) - L_{\xi}(\eta;(\psi,F))| \to 0 .$$
Combined with the previous estimate this yields that
\begin{equation*}\label{eqn: L cont wrt psi F}
 | L_{\xi}(\eta;(\psi_n,F_n))-L_{\xi}(\eta;(\psi,F))| \to 0 \ \hbox{ as } \ n \to \infty.
 \end{equation*}
This shows pointwise convergence of $L_\xi(\cdot;(\psi_n,F_n))$. Uniform convergence over all unit vectors $\eta \perp\xi$ will again follow from uniform boundedness and equicontinuity of $L_{\xi}$ (see Lemma~\ref{lem: L eta reg}). 

\qed

\medskip

Finally we give the proof of Proposition \ref{lem: mu discont}. 

\medskip

\textit{Proof of Proposition \ref{lem: mu discont}.}
Let $\xi\in \integer^d$ and $(\psi,F)\in X$. If $L_{\xi}(\cdot;(\psi,F))$ is non-constant then we are done, so we suppose it is constant and construct $(\psi_\e,F_\e)$.  

\medskip

Let us first show that we can assume without loss that $m_{\xi}(\cdot;(\psi,F))$ is non-constant. We will choose $\psi'$ with $\|\psi' - \psi\|_{C(\mathbb{T}^d)} \leq \e$ such that $m_{\xi}(\cdot; (\psi',F))$ is non-constant.  If $m_{\xi}(\cdot; (\psi,F))$ is already non-constant then we don't need to do anything and can take $\psi' = \psi$. Otherwise let us take
\begin{equation}\label{eqn: psi perturb}
 \psi'(y) := \psi(y) + \e \cos\left( 2 \pi y \cdot \xi \right) \ \hbox{ which satisfies } \ \|\psi' - \psi\|_{C(\mathbb{T}^d)} \leq \e.
 \end{equation}
  Observe that for each fixed hyperplane $\partial P_\xi + t\hat{\xi}$ we have $\psi'(y)=\psi(y) + \e\cos(2\pi|\xi|t)$. Therefore
$$
 m_{\xi}(t; (\psi',F)) = m_{\xi}(t;(\psi,F))+\e\cos (2\pi |\xi|t).
 $$
This is evidently non-constant when $m_{\xi}(\cdot;(\psi,F))$ is constant.  If $L_\xi(\cdot;(\psi',F))$ is non-constant we are done, thus we may suppose without loss that it is constant.  

\medskip

Let $\eta_1,\eta_2$ be unit vectors such that $\eta_j \perp \xi$ and $\eta_1 \perp \eta_2$, which is possible since $d \geq 3$. We will aim to perturb $F$ to construct $(F_\e,\psi_\e)$ so that
\begin{equation}\label{claim:main}
L_\xi(\eta_1;(\psi_\e,F_\e))< L_\xi(\eta_2;(\psi_\e,F_\e)).
\end{equation}
To this end let us define,
$$ F_\e(M) := \max\{F(M), F(M+\e (\eta_1^{T}M\eta_1) \eta_1 \otimes \eta_1)\}.$$
 It is not difficult to check the definition of ellipticity to see that $F_\e \in \textup{UE}_d$. Also,
 $$ \sup_{\|M\| = 1} |F(M+\e (\eta_1^{T}M\eta_1) \eta_1 \otimes \eta_1)-F(M)| \leq \Lambda(F)\e,$$
 and furthermore, since $D[F(M+\e (\eta_1^{T}M\eta_1) \eta_1 \otimes \eta_1)] = DF+\e (\eta_1^{T}DF\eta_1) \eta_1 \otimes \eta_1$,
 $$ |D F_\e - DF| \lesssim_d \e \Lambda(F) \ \hbox{ where it is defined.} $$
 Combining these two estimates it follows that $F_\e$ is close to $F$ in $d_{\textup{UE}_d}$ metric,
 $$ d_{\textup{UE}_d}(F,F_\e) \lesssim_d \Lambda(F) \e.$$

\medskip

Let us now show \eqref{claim:main}. From the definition of the $2$-d operators $G_{\xi,e}$ for $N \in \M_{2 \times 2}$,
\begin{align*}
G_{\xi,\eta_2}^\e(N) &= \max\left\{F\left(\sum_{ij} N_{ij} \alpha_i \otimes \alpha_j\right),F\left(\sum_{ij} N_{ij} (\alpha_i \otimes \alpha_j +\e (\eta_1 \cdot \alpha_i)(\eta_1 \cdot \alpha_j) \eta_1 \otimes \eta_1)\right)\right\} \\
&\hbox{ with } \alpha_1 = \eta_2 \hbox{ and } \alpha_2 = \hat\xi.
\end{align*}
note that $\eta_1 \cdot \alpha_j = 0$ for $j=1,2$ since $\xi,\eta_2 \perp \eta_1$ and so 
\begin{equation}\label{observation2}
G_{\xi,\eta_2}^\e(N) = G_{\xi,\eta_2}(N) \hbox{ for all } \e>0.
\end{equation}
On the other hand, calling $e_1 = (1,0)$,
$$G_{\xi,\eta_1}^\e(N) = \max\{G_{\xi,\eta_1}(N),G_{\xi,\eta_1}(N+\e N_{11} e_1 \otimes e_1)\}.$$
 Now for any $N$ with $N_{11} \leq 0$ by uniform ellipticity,
$$G_{\xi,\eta_1}(N+\e N_{11} e_1 \otimes e_1) \geq G_{\xi,\eta_1}(N) - \e  N_{11}, $$
and thus
\begin{equation}\label{observation1}
 G_{\xi,\eta_1}^\e(N) > G_{\xi,\eta_1}(N) \hbox{ if } N_{11}<0.
 \end{equation}
 We will make use of these observations below.
 
 \medskip
 
 We now aim to perturb $\psi$ to $\psi_e$ so that $m_\xi(t,(\psi_\e,F_\e)) = m_\xi(t;(\psi,F))$.  We can always write,
\[ m_\xi(t,(\psi,F)) = m_\xi(t,(\psi,F_\e)) + (m_\xi(t,(\psi,F)) - m_\xi(t,(\psi,F_\e))).\]
From the proof of Lemma~\ref{lem: L cont} $(m_\xi(t,(\psi,F)) - m_\xi(t,(\psi,F_\e)))$ converges to zero uniformly as $\e \to 0$.  Define a modified boundary data $\psi_\e$, satisfying $\psi_\e - \psi \to 0$ uniformly as $\e \to 0$,
\[ \psi_\e(y) = \psi(y) + (m_\xi(y \cdot \hat \xi,(\psi,F)) - m_\xi(y \cdot \hat \xi,(\psi,F_\e))) \ \hbox{ and we claim } \ m_\xi(\cdot,(\psi_\e,F_\e)) = m_\xi(\cdot,(\psi,F)).\]
The claim is immediate from the fact that the added term is constant on hyperplanes parallel to $\partial P_\xi$, it is the same argument as for the previous perturbation \eqref{eqn: psi perturb}.  Now let us refer to 
\[W_{\xi,\eta}^\e(\cdot) = W_{\xi,\eta}(\cdot;(\psi_\e,F_\e)) \ \hbox{ and } \ W_{\xi,\eta}(\cdot) = W_{\xi,\eta}(\cdot;(\psi,F)),\]
 these two solutions have the same boundary data on $\partial \real^2_+$ but the interior operators are $G^\e_{\xi,\eta}$ and $G_{\xi,\eta}$ respectively.

\medskip

 From above discussions we know that, 
$$W_{\xi,\eta_1}^\e \leq W_{\xi,\eta_1} \ \hbox{ and } \ W_{\xi,\eta_2}^\e \equiv W_{\xi,\eta_2}$$
where the first inequality is due to the fact $G_{\xi,\eta_1}^\e \geq G_{\xi,\eta_1}$.  On the other hand, since $G_{\xi,\eta_1}(D^2W_{\xi,\eta_1}^\e) \leq 0$, the dichotomy holds due to the strong maximum principle:
\begin{equation}\label{eqn: dichotomy 0}
(i) \ W_{\xi,\eta_1}^\e < W_{\xi,\eta_1} \ \hbox{ or } \ (ii)  \  W_{\xi,\eta_1}^\e \equiv W_{\xi,\eta_1}.
\end{equation}
In case $(i)$, since $W_{\xi,\eta_1}$ and $W_{\xi,\eta_1}^\e$ are $\frac{1}{|\xi|}$-periodic in the $z_1$ direction,
$$ W_{\xi,\eta_1}^\e(z_1,1 )  \leq W_{\xi,\eta_1}(z_1,1)- \delta \ \hbox{ for some } \ \delta >0.$$
By maximum principle it follows
$$  W_{\xi,\eta_1}^\e(z )  \leq W_{\xi,\eta_1}(z)- \delta  \ \hbox{ for } \ z_2 >1$$
and therefore
\begin{equation}\label{eqn: dichotomy int}
L_\xi(\eta_1;(\psi_\e,F_\e)) = \lim_{z_2 \to +\infty}W_{\xi,\eta_1}^\e(z ) \leq   \lim_{z_2 \to +\infty} W_{\xi,\eta_1}(z)- \delta<L_\xi(\eta_1;(\psi,F)). 
\end{equation}
We have just shown that \eqref{eqn: dichotomy 0} $(i)$ implies \eqref{eqn: dichotomy int} and therefore we actually have the dichotomy,
\begin{equation}\label{eqn: dichotomy 1}
(i) \ L_\xi(\eta_1;(\psi_\e,F_\e))< L_\xi(\eta_1;(\psi,F)) \ \hbox{ or } \ (ii)  \  W_{\xi,\eta_1}(\cdot ;(\psi_\e,F_\e)) \equiv W_{\xi,\eta_1}(\cdot; (\psi,F)).
\end{equation}
In case $(i)$ we have that 
\[L_\xi(\eta_1;(\psi_\e,F_\e)) < L_\xi(\eta_1;(\psi,F)) = L_\xi(\eta_2;(\psi,F)) = L_\xi(\eta_2;(\psi_\e,F_\e))\]
 which achieves the result since $(\psi_\e,F_\e) \to (\psi,F)$ in $d_X$ as $\e \to 0$.

\medskip

We just need to justify that $(i)$ holds.   Suppose this is not the case and $(ii)$ holds in \eqref{eqn: dichotomy 1}.  By Nirenberg's Theorem (Theorem~\ref{thm: nirenberg}) $W_{\xi,\eta_1},W^\e_{\xi,\eta_1}$ are $C^{2,\alpha_0}$ for a small $\alpha_0(\Lambda)>0$ and are classical solutions of their respective equations.  But then $D^2W^\e_{\xi,\eta_1} \equiv D^2W_{\xi,\eta_1}$ and
\begin{align*}
 G_{\xi,\eta_1}(D^2W_{\xi,\eta_1}(z)) &= 0 = G_{\xi,\eta_1}^\e(D^2W_{\xi,\eta_1}^\e(z)) = G_{\xi,\eta_1}^\e(D^2W_{\xi,\eta_1}(z)) \ \hbox{ for every $z \in \real^2_+$ } \\
 &\hbox{ and therefore, by \eqref{observation1},}  \ D^2_{11}W_{\xi,\eta_1}(z) \geq 0 \ \hbox{ for every $z \in \real^2_+$ }. 
 \end{align*}
    On the other hand, by the $1/|\xi|$ periodicity of $W_{\xi,\eta}$ in the $z_1$ variable,
    $$ \int_a^{a+1/|\xi|} D^2_{11}W_{\xi,\eta_1}(z_1,z_2) dz_1 = 0  \ \hbox{ for all } \ a \in \real, z_2 >0.$$
    and therefore $D^2_{11}W_{\xi,\eta_1} = 0$ in the whole of $\real^2_+$.  In particular for any $t>0$, $W_{\xi,\eta_1}(\cdot,t)$ is constant, so $W_{\xi,\eta_1}$ is also constant in $\{z_2>t\}$ by uniqueness of the bounded solution of $G_{\xi,\eta_1}(\cdot)=0$.  Since $t>0$ was arbitrary $W_{\xi,\eta_1}$ is constant in $\real^2_+$, but this contradicts the boundary data $m_\xi(\cdot; (\psi,F))$ being non-constant.  This completes the proof.

\section{Improved Estimates in the Linear Case}\label{linear}
In this section we show the best possible continuity estimates for $\mu(\nu, \psi, F)$ by our current methods in the linear case. The main tool is the higher regularity estimates available for linear operators in $\R^d$ or in half-spaces with smooth boundary data. For our purpose $W^{3,d}$ estimates would be sufficient, but we do not pursue this minimal assumption since it would be too much to expect in the general nonlinear case anyway.

\medskip

Consider $u^\e$ solving, for  $\nu \in S^{d-1}$ and $R>1$,
\begin{equation}\label{linear_pb}
\left\{
\begin{array}{lll}
F(D^2u^\e,\tfrac{x}{\e}):=-\Tr(A(\tfrac{x}{\e})D^2u^\e) = 0 & \hbox{ in } & 0 <x \cdot \nu <R \vspace{1.5mm}\\
u^\e = g(x) & \hbox{ on } & x \cdot \nu  \in \{0,R\}.
\end{array}\right.
\end{equation}
  We assume that $A$ is $\integer^d$-periodic and smooth and satisfies $(Id)_{d\times d}\leq A \leq \Lambda(Id)_{d\times d}$.  Due to the linearity,  the interior corrector  can be written as $v (y,M) = \Sigma_{ij} v_{ij}(y)M_{ij}$, where $v_{ij}(y)$ solves
$$ -\Tr\left[A(y)(e_i \otimes e_j+D^2_yv_{ij})\right] = -\overline{A}_{ij} \ \hbox{ in } \ \real^d, $$
with the estimate
\begin{equation}\label{corrector_est}
 \|v_{ij}\|_\infty ,\|Dv_{ij}\|_\infty \leq C(\Lambda,d).
 \end{equation}

\medskip

The following result is not optimal in terms of the required regularity of $g$, but it is sufficient to improve the estimate of Section \ref{step 2}, to match the order of the estimate in Section \ref{step 1}.  
\begin{thm}\label{thm: linear rate}
Let $g$ and $u^\e$ as given above, and let $\bar{u}$ be as given in \eqref{eqn: hom unbdd} solving the homogenized equation. Then $u^\e$ converges to $\bar{u}$ with the convergence rate
$$ \sup_{0<x \cdot \nu<R} |u^\e(x) - \overline{u}(x)| \leq C(\Lambda,d,\beta)\|g(R  \cdot) \|_{C^{3,\beta}}R^{-1}\e \quad \hbox{ for any } R>0.$$
\end{thm}

 In order to use the above theorem to obtain the desired interior homogenization rate,  higher regularity of the boundary condition is needed.  In our setting in Section \ref{step 2}, that boundary condition is $m_\xi( x \cdot \hat\eta)$ for some $\eta \perp \xi$.  In the following lemma we show that $m_\xi(\cdot;(\psi,F))$ has sufficient regularity to apply Theorem \ref{thm: linear rate} when $\psi$ is regular.  Indeed the proof of this lemma is the most interesting and delicate part of this section. 

\begin{lem}\label{lem: m reg}
Let $F$ be as in \eqref{linear_pb}. Then for all $k \in \mathbb{N}$, $\beta \in (0,1]$,
$$ \|m_\xi(\cdot;(\psi,F))\|_{C^{k,\beta}(\real)} \leq C(k,\beta,\Lambda,d,\|A\|_{C^{k+1}})(\log(2+|\xi|))^{\left \lceil{k+\beta}\right \rceil } \|\psi\|_{C^{k+4}}.$$
\end{lem}

Note that by Lemma~\ref{est:m} when $k=0$, $\beta \in (0,1)$ an improved estimate holds without any logarithmic growth in $|\xi|$.  It is important here that the estimate is not getting too much worse in terms of $|\xi|$ as $k$ increases.

\medskip 

Now we can return to the proof of Lemma \ref{lem: int hom unbd} and use Theorem \ref{thm: linear rate} to achieve the following result:
\begin{lem}\label{lem appendix 3}
Let $\xi \in \integer^d \setminus \{0\}$ and $\eta \perp \xi$.  Define $w_{\xi,\eta},\overline{w}_{\xi,\eta}$ as in \eqref{eqn: int level hom}. Then, when $F$ is linear, 
$$ \sup_{P_\xi}|w_{\xi,\eta}-\overline{w}_{\xi,\hat\eta}| \leq C(\beta,\Lambda,d,\|A\|_{C^5})\|m_\xi(|\xi|^{-1}\cdot)\|_{C^{3,\beta}(\real)}|\xi||\eta|[1+(\log\tfrac{1}{|\xi||\eta|})^{3}]. $$
\end{lem}
Note that by Lemma~\ref{lem: m reg}, 
$$\|m_\xi(|\xi|^{-1}\cdot)-m_\xi(0)\|_{C^{3,\beta}(\real)}\lesssim |\xi|^{-3-\beta}|D^3m_\xi|_{C^{0,\beta}} \lesssim_{A} \|\psi\|_{C^{7}},$$
 where we have used the $\frac{1}{|\xi|}$-periodicity to estimate the lower order terms by $|D^3m_\xi|_{C^{3,\beta}}$ and we also used $(\log(2+|\xi|))^4/|\xi|^{3+\beta} \lesssim 1$ for $|\xi| \geq 1$.  Combining the estimates of Theorem~\ref{thm: linear rate}, Lemma~\ref{lem: m reg} and Lemma~\ref{lem appendix 3}, in the same way as we did before in the general case in Theorem \ref{thm: cont no mod} and Corollary \ref{cor: cont mod}, we now obtain the continuity estimate,
\begin{thm}\label{thm: improved cont linear}
If $F$ is linear and $\psi \in C^{7}(\mathbb{T}^d)$, then for every irreducible $\xi \in \integer^d \setminus \{0\}$ and $\nu \in S^{d-1} \setminus \real \integer^d$,
$$ |\mu(\nu,\psi,F) - L_\xi(\psi,F)| \leq C(\Lambda,d,\|A\|_{C^5})\|\psi\|_{C^{7}}|\xi-|\xi|\nu | [1+(\log\tfrac{1}{|\xi-|\xi|\nu|})^3].$$
Furthermore, for every $\nu,\nu' \in S^{d-1}\setminus \real \integer^d$,
$$ |\mu(\nu,\psi,F) - \mu(\nu',\psi,F)| \leq C(\Lambda,d,\|A\|_{C^5})\|\psi\|_{C^{7}} |\nu - \nu'|^{1/d}[1+ (\log\tfrac{1}{|\nu - \nu'|})^{3}].$$
\end{thm}

\medskip

We omit the proof of Theorem~\ref{thm: improved cont linear} as it is a straightforward consequence of the improved estimates of Lemma~\ref{lem: m reg} and Lemma~\ref{lem appendix 3} and its usage in the proof of Theorem \ref{thm: cont no mod} and Corollary \ref{cor: cont mod}. Before we proceed with the proofs of Theorem ~\ref{thm: linear rate} and Lemmas ~\ref{lem: m reg} - \ref{lem appendix 3}, we make an interlude to explain some background results that we will make use of.

\subsection{Background Results}

For the proofs in the next subsection we will need to use the regularity results of Avellaneda-Lin \cite{AvellanedaLin} for solutions of non-divergence form linear homogenization problems.  We state the result here in the form that we will use it.  Suppose that $u^\e,v^\e$ solve respectively,
\begin{equation}
-\Tr(A(\tfrac{x}{\e})D^2u^\e) = f(x) \ \hbox{ in } \ B_1 \ \hbox{ and } \ 
\left\{
\begin{array}{lll}
-\Tr(A(\tfrac{x}{\e})D^2v^\e) = f(x) & \hbox{ in } & B_1 \cap P_\nu \vspace{1.5mm}\\
v^\e = g(x) & \hbox{ on } & \partial P_\nu \cap  B_1.
\end{array}\right.
\end{equation}
The following results hold uniformly in $\nu \in S^{d-1}$. We first state the classical results in unit scale.
\begin{thm}\label{thm: c2b estimates}
For every $\beta \in (0,1)$ there exists a constant $C(d,\Lambda,\beta,\|DA\|_{L^\infty(\mathbb{T}^d)})$ such that,
\begin{equation*}
\begin{array}{ll}
(1) & \|D^2u^1\|_{C^{0,\beta}(B_{1/2})} \leq C(\osc_{B_1} u^1+\|f\|_{C^{0,\beta}(B_1)}) \vspace{1.5mm}\\
(2) &  \|D^2v^1\|_{C^{0,\beta}(B_{1/2} \cap P_\nu)} \leq C(\osc_{B_1 \cap P_\nu} v^1+\|f\|_{C^{0,\beta}(B_1 \cap P_\nu)}+\|g\|_{C^{2,\beta}(\partial P_\nu \cap B_1)})  
 \end{array}
 \end{equation*}
\end{thm}
Scaling arguments yields that the best possible uniform in $\e$ regularity estimate is $C^{1,1}$. \cite{AvellanedaLin, AL-Lp} shows that this estimate indeed holds, although for our purposes the more useful estimate will be the $W^{2,p}$ estimate.  Below is a combinations of the main Theorems in \cite{AvellanedaLin, AL-Lp}:
\begin{thm}[Avellaneda-Lin]\label{thm: AL}
For every $1<p<\infty$, $\beta >0$ there exist constants $C_1(d,\Lambda,\beta,\|DA\|_{L^\infty(\mathbb{T}^d)})$ and $C_2(d,\Lambda,p,\|DA\|_{L^\infty(\mathbb{T}^d)})$ such that for all $\e>0$,
\begin{equation*}
\begin{array}{lrl}
(1) & \|D^2u^\e\|_{L^\infty(B_{1/2})} &\leq C_1(\osc_{B_1} u^\e+\|f\|_{C^{0,\beta}(B_1)}) \vspace{1.5mm}\\
(2) &   \|D^2u^\e\|_{L^p(B_{1/2})} &\leq C_2(\osc_{B_1} u^\e+\|f\|_{L^p(B_1)}) \vspace{1.5mm} 
 \end{array}
 \end{equation*}
\end{thm}

\subsection{Proofs of the results of Section \ref{linear}}\label{linear: m} Finally we prove Theorem \ref{thm: linear rate}, and Lemmas~\ref{lem: m reg} and \ref{lem appendix 3}.

\medskip

\textit{Proof of Theorem \ref{thm: linear rate}. } It suffices to consider the case $R=1$, the case for general $R>0$ follows from rescaling.  Let $\overline{u}$ be as given in the theorem. Let $U := \{0 < x \cdot \nu < 1\}$.  By the classical $C^{2,1}$ estimate up to the boundary (for the Laplacian) at unit scale, 
$$\|D^3\overline{u}\|_{L^\infty(U)} \leq C(\Lambda,d,\beta)(\osc_U \overline{u}+\|D^3g\|_{C^{0,\beta}(\partial U)}) \leq C\|g\|_{C^{3,\beta}(\partial U)} .$$
Let $\overline{u}^\e(x) := \rho_\e \star \overline{u} (x)$ with a standard mollifier $\rho_\e$, well defined on $U_\e = \{\e < x \cdot \nu < 1-\e\}$, and define
$$ \phi^\e(x) := \overline{u}^\e(x) + \sum_{i,j}\e^2v_{ij}(\tfrac{x}{\e})D^2_{ij}\overline{u}^\e(x).$$
Then we have
\begin{align*}
|\phi^\e(x) - \overline{u}(x)| &\leq \e\|D\overline{u}\|_{L^\infty(U)}+C(d)\e^2\|D^2\overline{u}^\e\|_{L^\infty(U)}\sup_{ij} \|v_{ij}\|_{L^\infty(\real^d)} \\
&\leq C(\Lambda,d)\|g\|_{C^{2,1}(\partial U)}\e.
\end{align*}
On $x \cdot \nu = \e$ we have, by the regularity up to the boundary of Theorem \ref{thm: AL},
$$
 |\phi^\e(x) - u^\e(x)| \leq |\phi^\e(x) -\overline{u}(x)|+|\overline{u}(x) - g(x')| +|g(x')- u^\e(x)|\leq C(\Lambda,d)\e\|g\|_{C^{2,1}}.
 $$
A similar estimate holds on $x \cdot \nu = 1-\e$.  Thus we can estimate,
\begin{equation}\label{eqn: u to phi}
 \sup_U |\overline{u} - u^\e| \leq \sup_{U_\e} |\overline{u} - u^\e|+C(\Lambda,d)\e\|g\|_{C^{2,1}} \leq \sup_{U_\e} |\phi^\e - u^\e|+C(\Lambda,d)\e\|g\|_{C^{2,1}}.
 \end{equation}
It remains to estimate $\sup_{U_\e} |\phi^\e - u^\e|$.  In $U_\e$, using the uniform ellipticity and then the definition of the corrector,
\begin{align*}
 -\Tr(A(\tfrac{x}{\e})D^2\phi^\e(x)) &\geq -\Tr[A(\tfrac{x}{\e})(D^2\overline{u}^\e(x) + \sum_{i,j}D^2v_{ij}(\tfrac{x}{\e})D^2_{ij}\overline{u}^\e(x))] \cdots \\
 & \quad \cdots - \sum_{i,j}\Lambda\left[2\e |Dv_{ij}(\tfrac{x}{\e})|| DD^2_{ij}\overline{u}^\e(x)| +\e^2|v_{ij}(\tfrac{x}{\e})||D^2D^2_{ij}\overline{u}^\e(x)|\right] \\
 & \geq - \Tr(\overline{A}D^2\overline{u}(x)) -C|D^2\overline{u}(x)-D^2\overline{u}^\e(x)| \cdots \\
 & \quad \cdots  -C\e (\sup_{i,j}\|Dv_{ij}\|_\infty)|D^3\overline{u}^\e(x)|-C\e^2(\sup_{i,j}\|v_{ij}\|_\infty)]|D^4\overline{u}^\e(x)|.
 \end{align*}
   Recall that, from \eqref{corrector_est}, $\|D^kv_{ij}\|_\infty \leq C(\Lambda,d)$ for $k=0,1$.  Since $-\Tr(\overline{A}D^2\overline{u}(x))=0$ we get the following supersolution/subsolution conditions,
  $$ f(x) \geq -\Tr(A(\tfrac{x}{\e})D^2\phi^\e(x)) \geq -f(x) $$
   where we have called $f(x)$ to be the error from the preceding calculation,
  $$ f(x) : = C(\Lambda,d)[|D^2\overline{u}(x)-D^2\overline{u}^\e(x)|+\e|D^3\overline{u}^\e(x)|+\e^2|D^4\overline{u}^\e(x)|].$$
  The first and third terms can be estimated in terms of $\bar{u}$ by standard mollification estimates:
  $$ |D^4\overline{u}^\e(x)| \leq \e^{-1}\|D^3\overline{u}\|_{L^\infty} \ \hbox{ and } \ |D^2\overline{u}(x)-D^2\overline{u}^\e(x)| \leq \e\|D^3\overline{u}\|_{L^\infty} \ \hbox{ for all } \ \e < x \cdot \nu < R-\e.$$
  Thus we have,
  \begin{align*}
   \|f\|_{L^\infty} &\leq C(\Lambda,d) \|D^3\overline{u}\|_{L^\infty}\e.
   \end{align*}
 Finally, using $C(\Lambda,d)\|f\|_\infty(x \cdot \nu)(1-x \cdot \nu)$ as a barrier
  $$ \sup_{U_\e}|\phi^\e - u^\e| \leq \sup_{\partial U_\e} |\phi^\e - u^\e| +C(\Lambda,d)\|f\|_{L^\infty(U_\e)} ,$$
  which, when combined with the estimate near the boundary gives,
  $$\sup_{U_\e}|\phi^\e - u^\e| \leq C(\Lambda,d,\beta)\|g\|_{C^{3,\beta}(\partial U)}\e.$$
  Combining with \eqref{eqn: u to phi} completes the proof.
\qed

\medskip

The proof of Lemma \ref{lem appendix 3} is almost exactly the same as the proof of Lemma \ref{lem: int hom unbd} except we now use the improved interior homogenization rate given in Theorem \ref{thm: linear rate}. 

\medskip

\textit{Proof of Lemma \ref{lem appendix 3}.}  Recall that ${m}_\xi( y \cdot \hat{\eta})$ is $\frac{1}{|\xi|}$-periodic on $\partial P_\xi$ in the direction $\hat{\eta}$ and constant in the directions orthogonal to $\hat{\eta}$.  As usual we can restrict to the case $|\xi||\eta| \leq 1/2$.  Let $w_{\xi,\eta}$, $\overline{w}$, $\mu$ and $\bar{\mu}$ as given in the proof of Lemma \ref{lem appendix 3}. Recall from Lemma \ref{lem: exp rate per} and Lemma \ref{lem: exp rate almost per} that,
\begin{equation}\label{eqn: w mu est}
 |w_{\xi,\eta}(y) - \mu|+ |\overline{w}_{\xi,\eta}(y)-\overline{\mu}| \leq C[(\osc m_\xi) \exp(-c|\xi|R)+\|Dw_{\xi,\eta}\|_{L^\infty(P_\xi)}|\eta|] \ \hbox{ for } \ y \in \partial P_\xi + R\hat{\xi}.
 \end{equation}
 By Theorem \ref{thm: AL}, for any $r>0$,
 $$\|Dw_{\xi,\eta}\|_{L^\infty(B_{|\xi|^{-1}} \cap P_\xi)} \leq C(|\xi|\osc m_\xi+|\xi|^{-1}\|D^2m_\xi\|_\infty) \leq C \|m_\xi(|\xi|^{-1}\cdot)\|_{C^{3,\beta}}|\xi|,$$
 Fix an $R \geq |\xi|^{-1}$ to be chosen and consider,
$$\tilde{w}_{\xi,\eta}(y) = w_{\xi,\eta}(y) +(\overline{\mu}-\mu)R^{-1}y \cdot \hat{\xi} + \sup_{\partial P_\xi+R\hat{\xi}}\big[|w_{\xi,\eta}(\cdot) - \mu|+ |\overline{w}_{\xi,\eta}(\cdot)-\overline{\mu}|\big]. $$
Note that with this modification $\tilde{w}_{\xi,\eta}$ still solves the same equation as $w_{\xi,\eta}$ in $P_\xi$ with the same boundary condition on $\partial P_\xi$ but now also,
$$  \tilde{w}_{\xi,\eta}(y)\geq \overline{w}_{\xi,\eta}(y) \ \hbox{ on } \ \partial P_\xi + R\hat{\xi}. $$
Now Theorem \ref{thm: linear rate} implies that,
$$ \overline{w}_{\xi,\eta}(y) - \tilde{w}_{\xi,\eta}(y) \leq C\|m_\xi(R \cdot)\|_{C^{3,\beta}}R^{-1}|\eta| \leq C \|m_\xi(|\xi|^{-1} \cdot)\|_{C^{3,\beta}}(R|\xi|)^{3+\beta}R^{-1}|\eta| .$$
 Rewriting this in terms of ${w}_{\xi,\eta}$,
\begin{equation*}
 \overline{w}_{\xi,\eta}(y) - w_{\xi,\eta}(y) \leq (\overline{\mu}-\mu)R^{-1}y \cdot \hat{\xi}+C\|m_\xi(|\xi|^{-1} \cdot)\|_{C^{3,\beta}}[R^{2+\beta}|\xi|^{3+\beta}|\eta|+ C(\osc m_\xi) \exp(-c|\xi|R)+|\xi||\eta|].
\end{equation*}
Let us choose $R = 2(c|\xi|)^{-1}\log\frac{1}{|\eta||\xi|}$ to obtain 
\begin{equation}\label{eqn: w est}
 \overline{w}_{\xi,\eta}(y) - w_{\xi,\eta}(y) \leq (\overline{\mu}-\mu)R^{-1}y \cdot \hat{\xi}+C\|m_\xi(|\xi|^{-1} \cdot)\|_{C^{3,\beta}}|\xi||\eta|(\log\tfrac{1}{|\xi||\eta|})^{2+\beta}.
\end{equation}
Due to  \eqref{eqn: w mu est}, evaluating  \eqref{eqn: w est} for $y \in \partial P_\xi + \tfrac{1}{2}R\hat{\xi}$ yields 
$$
\overline{\mu}-\mu \leq \tfrac{1}{2}(\overline{\mu}-\mu)+C\|m_\xi(|\xi|^{-1} \cdot)\|_{C^{3,\beta}}|\xi||\eta|(\log\tfrac{1}{|\xi||\eta|})^{2+\beta}.
$$
 Rearranging the last inequality and making a similar argument for the lower bound,
\begin{equation}\label{eqn: mu est}
 |\overline{\mu}-\mu| \leq  C\|m_\xi(|\xi|^{-1} \cdot)\|_{C^{3,\beta}}|\xi||\eta|(\log\tfrac{1}{|\xi||\eta|})^{2+\beta}.
 \end{equation}
But now we can plug \eqref{eqn: mu est} back into \eqref{eqn: w est} to arrive at the desired result,
\begin{equation}\label{eqn: w est 3}
 |\overline{w}_{\xi,\eta}(y) - w_{\xi,\eta}(y)| \leq C\|m_\xi(|\xi|^{-1} \cdot)\|_{C^{3,\beta}}|\xi||\eta|(\log\tfrac{1}{|\xi||\eta|})^{2+\beta}.
\end{equation}
\qed

\medskip

For the proof of Lemma~\ref{lem: m reg} we introduce a special norm for functions $f : P_{e_d} \to \real$ which is suited to solving the equation $- \sum_{ij}A_{ij}(y)D^2_{ij}u = f(y)$ in the half-space.  See Appendix~\ref{appendix} for the full development. For each $R>0$ we define $\mathcal{F}_R$, the class of axis aligned cubes contained in $P_{e_d}$ with side length $R$ and distance to $\partial P_{e_d}$ at least $R/2$, 
$$\mathcal{F}_R =  \{ Q \ \hbox{ cube}: Q = [-R/2,R/2)^d+x \ \hbox{ with } x_d \geq R\}.$$
Then we define,
\begin{equation}
M_p(f,R) = \sup_{Q \in \mathcal{F}_R} \left(\frac{1}{|Q|}\int_Q |f(y)|^p  \ dy\right)^{1/p} \ \hbox{ and } \ I_p(f) = M_p(f,0)+\sum_{N \in 2^{\mathbb{N}}} N^2M_p(f,N),
\end{equation}
These norms measure an $L^p$-averaged decay of $f$ in the $y_d$ direction.  One can easily extend this norm to functions $f$ in other half spaces by simply rotating appropriately the cubes used in the definition.  We will use the norms $M_p(f,R)$ and $I_p(f)$ for functions $f : P_{\nu} \to \real$ for a general $\nu \in S^{d-1}$ without further comment.  Roughly speaking, when $I_p(f)$ is finite then there exists a bounded solution of $- \sum_{ij}A_{ij}(y)D^2_{ij}u = f(y)$ in $P_\nu$ with zero (or a bounded $\psi$) Dirichlet boundary data and one can bound $\sup_{P_\nu} |u| \lesssim I_p(f)$.  Again we refer to Appendix~\ref{appendix} for the full explanation.

\medskip

\textit{Proof of Lemma \ref{lem: m reg}.}  
Recall that $m_\xi(t)$ is defined as the boundary layer tail of $v_{\xi,t\hat\xi }$ solving the cell problem,
\begin{equation}\label{eqn: cell prob111}
\left\{
\begin{array}{lll}
- \sum_{ij}A_{ij}(y+\tau)D^2_{ij}v_{\xi,\tau} = 0 & \hbox{ in } & P_{\xi} \vspace{1.5mm} \\
v_{\xi,\tau} = \psi(y+\tau) & \hbox{ on } & \partial P_{\xi},
\end{array}\right.
\end{equation}
Let us denote the differential operator $\partial:=\hat\xi \cdot D_\tau $.  Take any $p>d/2$, for example $p=d$ will work fine.  We will prove by induction that, for all $k \in \mathbb{N} \cup \{0\}$: if $\|\psi\|_{C^{k+3}} \leq 1$ then the following hold for all $R>0$ uniformly in $\tau \in \real^d$,
\begin{align}
 \tag{$i$}\label{max} & \sup_{y \in P_\xi}|\partial^kv_{\xi,\tau}(y)| \lesssim_{k,A}(\log(2+|\xi|))^{k} \\ 
 \tag{$ii$}\label{osc}  &  \osc_{ y \cdot \hat\xi \geq R} \partial^kv_{\xi,\tau}(y) \lesssim_{k,A}(\log(2+|\xi|))^{k}\exp(-c_{0}R/|\xi|) \\
       \tag{$iii$}\label{C11}   & \ M_p(D^2\partial^kv_{\xi,\tau},R) \lesssim_{k, A}(\log(2+|\xi|))^{k}(2+R)^{-2}\exp(-c_{0}R/|\xi|)
 \end{align}
 We note that if \eqref{C11} holds for some $k$ then by Lemma~\ref{lem: If interest},
 \begin{equation}\tag{$iv$}\label{Ip}
 I_p(D^2\partial^kv_{\xi,\tau}) \lesssim_{k, A}(\log(2+|\xi|))^{k+1}
 \end{equation}
  Once we have proven that \eqref{max} holds for all $k$ we will be done since $\frac{d^k}{dt^k}m_\xi(t) = \lim_{R \to \infty} \partial^k v_{\xi,t\hat\xi}(R\hat\xi)$.  The outline of the argument is as follows. For each $k$ we prove \eqref{max}-\eqref{osc} using that \eqref{max}-\eqref{C11} hold for all $m<k$.  Then we show \eqref{C11} using Theorem~\ref{thm: AL} to estimate $M_p(f,R)$ and thereby $I_p(f)$.

\medskip

For $k=0$, \eqref{max} is maximum principle and \eqref{osc} is a consequence of  Lemma~\ref{lem: exp rate per}.  To show \eqref{C11} note that Theorem~\ref{thm: AL} implies 
\begin{align*}
 \|D^2v_{\xi,\tau}\|_{L^\infty(\{y \cdot \hat\xi \geq R\})} &\lesssim_A R^{-2} \osc_{\{y \cdot \hat\xi \geq R/2\}} v_{\xi,\tau} \\
 &\lesssim_A (R+2)^{-2}(\osc \psi ) \exp(-\tfrac{1}{2}c_0R/|\xi|) \ \hbox{ for } R>1,
 \end{align*}
where $c_0$ is the exponential rate of convergence to the boundary layer tail from Lemma~\ref{lem: exp rate per rhs}. For $R\leq1$ the estimate follows from the up to the boundary $C^{2,\beta}$ estimates at unit scale, see Theorem~\ref{thm: c2b estimates}.

\medskip

Suppose that \eqref{max}-\eqref{C11} hold for all $m\leq k-1$. We aim to prove \eqref{max}-\eqref{C11} for $\partial^{k}v_{\xi,\tau}$ under the assumption $\|\psi\|_{C^{k+3}} \leq 1$. Our induction is based on the fact that  $\partial^k v_{\xi,\tau}$ solves
\begin{equation}\label{PDE_k}
-\sum_{ij}A_{ij}(y+\tau)D^2_{ij}\partial^k v_{\xi,\tau}(y) = f_{\xi,\tau}(y)\quad \hbox{ in } P_\xi,
\end{equation}
where 
$$
f_{\xi,\tau}(y):= \sum_{ij}\sum_{\ell+m = k, \ell \neq 0}  [ (\hat\xi \cdot D_y)^\ell A_{ij}(y+\tau)]D^2_{ij}\partial^m v_{\xi,\tau}(y)
$$

with boundary data
$$ \partial^{k}v_{\xi,\tau}(y) = [(\hat\xi \cdot D_y)^k\psi](y+\tau) \ \hbox{ on } \ \partial P_\xi.$$
(To be completely precise we should take difference quotients instead of $\partial^k v_{\xi,\tau}$; the reader can easily see how to make our formal argument rigorous.) Note that $f_{\xi,\tau}$, the boundary data $\partial^{k}v_{\xi,\tau}$, and the operator $F$ are all periodic with respect to a lattice on $\partial P_\xi$ with unit cell diameter of order $|\xi|$.  This will allow us to use the results of the Appendix which extend Section~\ref{sec: periodic bc} Lemma~\ref{lem: exp rate per} to include equations with a right hand side.  We use the inductive hypothesis \eqref{C11} to see that $f_{\xi,\tau}(y)$ satisfies,
\begin{align*}
M_p(f_{\xi,\tau},R) 
&\lesssim_{k, A} (\log(2+|\xi|))^{k-1}\exp({-c_0R/|\xi|}).
\end{align*}
In particular $f_{\xi,\tau}$ fits under the assumptions of the results of the Appendix.  We can apply directly Lemma~\ref{lem: If interest} in combination with Lemma~\ref{lem: f rhs} to get, 
\begin{equation}\label{eqn: i k}
 \sup_{ y \in P_{\xi}} \partial^kv_{\xi,\tau}(y)  \lesssim_{k,A}\|\partial^k\psi\|_\infty+  I_p(f_{\xi,\tau}) \lesssim  (\log(2+|\xi|))^k.
 \end{equation}
  We have used that, by assumption, $\|\psi\|_{C^{k+3}} \leq 1 $.  Then we use Lemma \ref{lem: exp rate per rhs} to get, 
  \begin{align}
 \notag\osc_{y\cdot \nu \geq R} \partial^k v_{\xi,\tau} &\lesssim_{k,A}\left[\|\partial^k\psi\|_\infty+  (\log(2+|\xi|))^{k}\right]\exp(- c_{0}R/|\xi|) \\ \label{eqn: ii k}
 &\lesssim (\log(2+|\xi|))^{k}\exp(- c_{0}R/|\xi|). 
 \end{align}
 This establishes \eqref{osc}. 
 
 \medskip
 
It remains to prove \eqref{C11}.  Theorem \ref{thm: AL} yields that, for $R>1$,
\begin{align*}
M_p(D^2 \partial^kv_{\xi,\tau},R) &\lesssim R^{-2} \osc_{\{y \cdot \hat\xi \geq R/2\}} v_{\xi,\tau}+M_p(f_{\xi,\tau},\tfrac{R}{2}) \notag \\
 &\lesssim_{k,A} (2+R)^{-2} (\log(2+|\xi|))^k\exp(-c_{0}R/|\xi|). 
 \end{align*}
 For $R<1$ the estimate follows from the up to the boundary $C^{2,\beta}$ estimates at unit scale, see Theorem~\ref{thm: c2b estimates}, combined with \eqref{eqn: i k} and $\|\psi\|_{C^{k+3}} \leq 1$. Together we obtain for all $R>0$,
  \begin{equation}\label{eqn: iii k}
  M_p(D^2\partial^kv_{\xi,\tau},R) \lesssim_{k,A} (2+R)^{-2} (\log(2+|\xi|))^k\exp(-c_{0}R/|\xi|). 
  \end{equation}
  
\qed

\appendix

\section{}\label{appendix}

Here we prove Lemma~\ref{lem: exp rate per} but now including a right hand side in the equation.  We restrict to linear equations, although a version of these results for nonlinear equations is true as well with a stronger assumption on the right hand side.  Let us consider $v$ solving the following problem:

\begin{equation}\label{eqn: periodic BC rhs}
\left\{
\begin{array}{lll}
-\sum_{ij}A_{ij}(y)D^2_{ij}v = f(y) & \hbox{ in } & P_{e_d} \vspace{1.5mm} \\
v = \phi(y) & \hbox{ on } & \partial P_{e_d}.
\end{array}\right.
\end{equation}
We assume that $A$, $f$ and $\phi$ are periodic with respect to linearly independent translations $\ell_1,...\ell_{d-1} \in\partial P_{e_d}$.  Recall that we denote $\Z$ by the periodicity lattice generated by $\{\ell_j\}$, $Q$ its unit cell, and $L$ the diameter of $Q$.  Let us assume $L>1$.  We will furthermore assume that $A$ is periodic with respect to some rotation of the $\integer^d$ lattice and $A\in C^{0,\beta}(\real^d)$ for some $\beta>0$. This periodicity and regularity will only be used apply the results of \cite{AvellanedaLin} to obtain estimates for the Green's function $G(x,y)$ associated with the inhomogeneous operator and the domain $P_{e_d}$. 

\medskip

When $f$ is continuous and has sufficient decay there is a unique bounded solution of \eqref{eqn: periodic BC rhs}. In order to investigate $v$ we define an appropriate norm for the decay of $f$ at infinity in the $y_d$ direction, this norm measures the $L^p$-averaged decay of $f$ far from $\partial P_{e_d}$.  For each $R>0$ we define $\mathcal{F}_R$, the class of axis aligned cubes contained in $P_{e_d}$ with side length $R$ and distance to $\partial P_{e_d}$ at least $R/2$,
\begin{equation}
\mathcal{F}_R =  \{ Q \ \hbox{ cube}: Q = [-R/2,R/2)^d+x \ \hbox{ with } x_d \geq R\}.
\end{equation}
Then we define,
\begin{equation}
M_p(f,R) = \sup_{Q \in \mathcal{F}_R} \left(\frac{1}{|Q|}\int_Q |f(y)|^p  \ dy\right)^{1/p} \ \hbox{ and } \ I_p(f) = M_p(f,0)+\sum_{N \in 2^{\mathbb{N}}} N^2M_p(f,N),
\end{equation}
where we define $M_p(f,0) = \|f\|_{L^{\infty}(P_{e_d})}$.  It will also be useful to define $I_p(f,R)$ where the sum in the definition of $I_p(f)$ is restricted to $2^{\mathbb{N}} \ni N \geq R$,
\begin{equation}
 I_p(f,R) = \sum_{2^{\mathbb{N}}\ni N \geq R } N^2M_p(f,N).
\end{equation}
Now we remind about some facts about the Green's function for operators with periodically oscillating coefficients.  Recall that for each $x \in P_{e_d}$ the Green's function $G(x,y)$ for the domain $P_{e_d}$ solves the adjoint equation,
\begin{equation}\label{eqn: green fcn}
\left\{
\begin{array}{lll}
-\sum_{ij}D_{y_i}D_{y_j}(A_{ij}(y)G(x,y)) = \delta(y-x) & \hbox{ for } & y\in P_{e_d} \vspace{1.5mm} \\
G(x,y) = 0 & \hbox{ for } & y\in \partial P_{e_d}.
\end{array}\right.
\end{equation}
From  the work of Avellaneda-Lin \cite{AL-div, AvellanedaLin, AL-Lp} $G(x,y)$ satisfies many of the same estimates as the Green's function associated with the Laplace operator.  We will make use of the following,
\begin{thm}\label{thm: AL gf}(Avellaneda-Lin \cite{AL-div, AvellanedaLin, AL-Lp}) Let $\beta \in (0,1)$ and suppose that $1 \leq A(y) \leq \Lambda$ is periodic on $\real^d$ with respect to $\integer^d$ (or a rotation of the integer lattice) translations and $\|A\|_{C^{0,\beta}(\real^d)} <\infty$. There is $c = c(d,\Lambda, \|A\|_{C^{0,\beta}},\beta)<\infty$ so that the Green's function for the half-space $G(x,y)$ solving \eqref{eqn: green fcn} satisfies,
\begin{equation}\label{eqn: gf std est}
 G(x,y) \leq c \frac{x_d y_d}{|x-y|^d} \ \hbox{ and } \ G(x,y) \leq 
 \left\{\begin{array}{ll}
 c|\log|x-y|| & d=2 \\
  c|x-y|^{2-d} & d \geq 3.
 \end{array}\right.
  \end{equation}
 \end{thm}
 Note that the result proven in \cite{AL-div} is actually for divergence form operators but a standard trick, which is explained in \cite{AvellanedaLin, AL-Lp}, allows one to translate the results for non-divergence form equations as well.  The point is that non-divergence form equations can be written as a divergence form equation with uniformly elliptic matrix $A'$ satisfying $\textup{div} A' = 0$.  
 
 \medskip
 
 Now we are able to state and prove, with the aid of Theorem~\ref{thm: AL gf}, our result on the upper and lower bounds of solution of \eqref{eqn: periodic BC rhs} in half-spaces.


\begin{lem}\label{lem: f rhs}
Let $p>\frac{d}{2}$ and suppose ${I}_p(f)$ is finite. There exists $C(p,d,\|A\|_{C^{0,\beta}},\beta)$ so that there exists a unique bounded solution $v$ of \eqref{eqn: periodic BC rhs} with,
$$\sup_{P_{e_d}}|v| \leq  \max_{\partial P_{e_d}}|\phi| +C {I}_p(f) \ \hbox{ when } \ d \geq 3$$
$$\sup_{P_{e_d}}|v| \leq  \max_{\partial P_{e_d}}|\phi|+C [{I}_p(f)+\sup_{N \in 2^\mathbb{N}} N^2(\log N) M_p(f,N)] \ \hbox{ when } \ d =2$$
\end{lem}
\begin{proof}
It suffices to show the result with $\phi = 0$.  We just need to show that 
 \begin{equation}\label{eqn: gf Ip}
  \int_{P_{e_d}} G(x,y) |f(y)| \ dy \lesssim I_p(f)
  \end{equation}
 so then
 $$  v(x) = \int_{P_{e_d}} G(x,y) f(y) \ dy$$
 is well defined, solves the desired equation in $P_{e_d}$ and has the upper bound claimed by the Lemma.  
 
 \medskip
 
 Fix an $x = (x',x_d) \in P_{e_d}$, we partition $P_{e_d}$ by dyadic cubes which are adapted to the location of $x$.  This is a technical step, the idea is to make the computations later easier by keeping all of the cubes in the partition but one a controlled distance from the singularity at $x$ of $G(x,y)$.  If $x_d \geq 1$ let $N_x \in 2^{\mathbb{N}} $ so that $N_x \leq x_d <2N_x$ and call $\alpha \in [1,2)$ such that $\alpha N_x = x_d$, if $x_d \leq 1$ then set $\alpha = 2$. Then we define the cubes $Q_{N,j} = x'+\alpha N(j,1) + [-\alpha N/2,\alpha N/2)^d$, $N \in 2^{\mathbb{N}}$ and $j \in \integer^{d-1}$, with side length comparable (by a factor of $2$) to their distance to the boundary.  Note that by our set up $x\in Q_{N_x,0}$ and the distance of $x$ to any other cube $Q_{N,j}$ is at least $c_d N(1+|j|)$.  
 
 \medskip
 
 Then we estimate $v$ by,
 \begin{align*}
 \int_{P_{e_d}} G(x,y) |f(y)| \ dy &= \sum_{Q} \int_{Q} G(x,y) |f(y)| \ dy \\
 & \leq \sum_{Q} |Q|^{1/p}\left(\int_{Q} G(x,y)^{p'}  dy\right)^{1/p'} \left(\frac{1}{|Q|}\int_{Q}|f(y)|^p dy\right)^{1/p} \\
 & \lesssim \sum_{N \in 2^{\mathbb{N}}} \sum_{j \in \integer^{d-1}}N^{d/p}M_p(f,N)\left(\int_{Q_{N,j}} G(x,y)^{p'}  dy\right)^{1/p'}
 \end{align*}
 We claim that for any $p>\frac{d}{2}$ and every $N \in 2^{\mathbb{N}}$, $j \in \integer^{d-1}$:
\begin{equation}\label{eqn: gf p' est}
\left(\frac{1}{|Q_{N,j}|}\int_{Q_{N,j}}G(x,y)^{p'} dy\right)^{1/p'} \leq C(p,d,\|A\|_{C^{0,\beta}},\beta) N^{2-d}(1+|j|)^{-d}.
\end{equation}
 Taking for granted \eqref{eqn: gf p' est} we continue the computation
 \begin{align*}
\sum_{N \in 2^{\mathbb{N}}} \sum_{j \in \integer^{d-1}}N^{d/p}M_p(f,N)\left(\int_{Q_{N,j}} G(x,y)^{p'} dy\right)^{1/p'}&\lesssim \sum_{N \in 2^{\mathbb{N}}} \sum_{j \in \integer^{d-1}} N^{d(\frac{1}{p}+\frac{1}{p'})}M_p(f,N)N^{2-d}(1+|j|)^{-d} \\
& = \sum_{N \in 2^{\mathbb{N}}} \sum_{j \in \integer^{d-1}} N^{2}M_p(f,N)(1+|j|)^{-d} \\
& \lesssim_d \sum_{N \in 2^{\mathbb{N}}} N^{2}M_p(f,N) = I_p(f).
\end{align*}
This proves the desired estimate \eqref{eqn: gf Ip}.

\medskip

To finish the proof we need to justify \eqref{eqn: gf p' est}. Let $y \in Q_{N,j}$ with either $j \neq 0$ or $N \neq N_x$.  If $N \geq N_x$ then $|x-y| \gtrsim (1+|j|)N$ and so using the first part of \eqref{eqn: gf std est},
\begin{equation*}
G(x,y)  \lesssim N^{2-d}(1+|j|)^{-d}.
\end{equation*}
If $N < N_x$ then $x_d - y_d \geq N_x/2 = x_d/2$ and $|y'-x'| \geq N|j|/2$ so using the first part of \eqref{eqn: gf std est} again,
\begin{equation*}
G(x,y) \lesssim \frac{Nx_d}{(x_d+N|j|)^{d}} \lesssim 
\left\{\begin{array}{ll}
N^{2-d}|j|^{-d} & N|j| \geq x_d \\
Nx_d^{1-d} & N|j| \leq x_d
\end{array}\right. \lesssim N^{2-d}(1+|j|)^{-d}.
\end{equation*}
Finally, if $j=0$ and $N=N_x$ then we use the second part of the Green's function estimate \eqref{eqn: gf std est} and $p' < \frac{d}{d-2}$ to get, when $d \geq 3$,
\begin{equation*}
\left(\int_{Q_{N,0}}  G(x,y)^{p'} dy\right)^{1/p'} \lesssim \left(\int_{Q_{N_x,0}}  |x-y|^{(2-d)p'} dy\right)^{1/p'} \lesssim \left(\int_0^{\sqrt{d} N} r^{(2-d)p'}r^{d-1}dr\right)^{1/p'} \lesssim N^{2-d}N^{d/p'}
\end{equation*}
which, after dividing through by $N^{d/p'} \sim |Q_{N,0}|^{1/p'}$, is the desired estimate \eqref{eqn: gf p' est}.  When $d=2$ we just use the $d=2$ estimate from \eqref{eqn: gf std est} instead.

\end{proof}

For notational simplicity we will only focus on the particular case to be used in the paper: let us assume 
\begin{equation}\label{eqn: f interest}
 M_p(f,R) \leq K(2+R)^{-2}e^{-bR} \  \hbox{ with }  b>0.
 \end{equation}
With above assumption we estimate $I_p(f)$ in the following Lemma:
\begin{lem}\label{lem: If interest}
Let $f$ such that $M_p(f,R)$ satisfies \eqref{eqn: f interest} for all $R>0$, then
\begin{equation}\label{eqn: If interest}
I_p(f) \leq CK\log(2+\tfrac{1}{b}) \ \hbox{ and } \ I_p(f,R) \leq CK\log(2+\tfrac{1}{b})e^{-bR}.
 \end{equation}
 Also, for the $d=2$ case, we have,
 \begin{equation}\label{eqn: Mpf d=2 interest}
N^2(\log N)M_p(f,N) \leq CK\log(2+\tfrac{1}{b}) \ \hbox{ and } \ \sup_{N \geq R} N^2(\log N)M_p(f,N) \leq CK\log(2+\tfrac{1}{b})e^{-bR}.
 \end{equation}
\end{lem}
\begin{proof}
Without loss of generality we can set $K=1$.  Note that by taking $R \to 0 $ in \eqref{eqn: f interest} (with $K=1$) we see $M_p(f,0) = \|f\|_{L^\infty(P_{e_d})} \leq 1$. 
\begin{align*}
I_p(f) &\leq 1+\sum_{N \in 2^{\mathbb{N}}} N^2M_p(f,N) \leq 1+\sum_{N \in 2^{\mathbb{N}}} e^{-bN} \\
&\lesssim \log(2+ \tfrac{1}{b})+\sum_{N > \frac{1}{b}}e^{-bN} \\
&\leq \log(2+ \tfrac{1}{b})+\sum_{N > \frac{1}{b}}\frac{1}{bN} \\
& \leq \log(2+ \tfrac{1}{b}).
\end{align*}
where we have used $e^x \geq x$ in the second inequality.  A similar argument proves the estimate of $I_p(f,R)$ and \eqref{eqn: Mpf d=2 interest} is straightforward from calculus.
\end{proof}
Now we are able to prove the corresponding version of Lemma ~\ref{lem: exp rate per} when $f$ is of the form in \eqref{eqn: f interest}.

\begin{lem}\label{lem: exp rate per rhs}
Suppose that $f$ satisfies \eqref{eqn: f interest} for some $K,b>0$.  Then there exists $\mu(\phi,F,f)$ and $c_0(\Lambda,d)>0$ such that,
$$ \sup_{y \cdot e_d \geq R} |v(y) - \mu| \leq C(p,d,\|A\|_{C^{0,\beta}},\beta)[(\osc \phi)e^{-c_0R/L}+K\log(2+\tfrac{1}{b})e^{-bR/L}] .$$
\end{lem}
\begin{proof}
We just argue when $d \geq 3$, but it will be clear given the matching estimates obtained from Lemma~\ref{lem: f rhs} and Lemma~\ref{lem: If interest} that the same result holds in $d=2$.

\medskip

After rescaling  we may assume that $K,\osc \phi \leq 1$. Let $\alpha(d,\Lambda) \in (0,1)$ and $C_0$ as given in Lemma~\ref{lem: exp rate per}.  Next define $A(j) := 2{I}_p(f,jr)+r^2M_p(f,(j+1/2)r)$ for $j\in\mathbb{N}$ and $r>2C_0^{1/\alpha}L$.  We claim that
\begin{equation}\label{eqn: periodic iteration2}
 \osc_{ P_{e_d} +k r \nu} v \leq C_0^k(2L/r)^{\alpha k }+\sum_{j=0}^{k-1}C_0^{k-j}(2L/r)^{\alpha (k-j) }A(j)+ 2{I}_p(f,kr).
 \end{equation}
 Let us assume for now that \eqref{eqn: periodic iteration2} is correct and continue with the rest of the proof.  Define
 $$
 r := (2eC_0)^{1/\alpha}L \ \hbox{ and } \ c_0 := (2eC_0)^{-1/\alpha}
 $$ 
 so that $C_0^{k}(2L/r)^{\alpha k } = e^{-k}$ and 
 \begin{equation}\label{eqn: sub r}
  \osc_{ P_{e_d} +k r \nu} v \leq e^{-k}+e^{-k}\sum_{j=0}^{k-1}e^jA(j)+ 2{I}_p(f,kr).
  \end{equation}
 Due to our assumption \eqref{eqn: f interest} we have the bound
 \begin{align*}
  r^2M_p(f,(j+1/2)r) &\leq r^2(2+(j+1/2)r)^{-2}e^{-b (j+1/2)r/L} \\
  &\lesssim e^{-b jr/L} \lesssim  e^{-c_0^{-1}b j},
  \end{align*}
 and using Lemma~\ref{lem: If interest} we also have,
 $$ {I}_p(f,jr) \lesssim_{b} \log(2+\tfrac{1}{b})e^{-b jr/L} \sim \log(2+\tfrac{1}{b})e^{-c_0^{-1}b j}.$$
 Plugging these estimates into \eqref{eqn: sub r} yields
 \begin{align*}
  \osc_{ P_{e_d} +k r \nu} v &\lesssim e^{-k}+\log(2+\tfrac{1}{b})e^{-k}\sum_{j=0}^{k-1} e^{(1-c_0^{-1} b) j}+ \log(2+\tfrac{1}{b})e^{-c_0^{-1} b k} \\
  &\lesssim e^{-k}+ \log(2+\tfrac{1}{b}) e^{-c_0^{-1} b k}.
  \end{align*}
Now for any $R>0$, let $k = [R/r]$ to get the desired result
$$ \osc_{ P_{e_d} +R\nu} v  \leq \osc_{ P_{e_d} +k r\nu} v \lesssim e^{-c_0R/L}+\log(2+\tfrac{1}{b})e^{- bR/L}.
$$

\medskip

It remains to prove \eqref{eqn: periodic iteration2} by induction.  For $k = 0$ \eqref{eqn: periodic iteration2} follows from Lemma~\ref{lem: f rhs}.  Assuming \eqref{eqn: periodic iteration2} for $k$, we prove it for $k+1$.  Note that
$$v_k(y) = v(y + kr e_d)$$
satisfies $-\sum A_{ij}(y+kre_d)D^2_{ij}v_k = f(y+kre_d)$ in $P_{e_d}$ with boundary data $\phi_k(y) =  v(y+ke_d)$.  Both the operator (by assumption) and the boundary data (by $\Z$-periodicity of operator and uniqueness) are periodic with respect to $(\ell_j)_{j = 1}^{d-1}$ translations, and 
$$\osc_{P_{e_d}} v_k \leq C_0^k(2L/r)^{\alpha k }+\sum_{j=0}^{k-1}C_0^{k-j}(2L/r)^{\alpha (k-j) }A(j)+ 2I_p(f,kr)$$
by the inductive hypothesis.  Then by the interior H\"{o}lder estimate in $B_{r/2}(r e_d)$,
$$ 
| v_k|_{C^\alpha(B_{r/4}(r e_d))} \leq 2^\alpha C_0r^{-\alpha}(\osc v_k + r^2M_p(f,(k+1/2)r)),
$$
and so, since $r>4L$, the oscillation on $Q$ can be estimated by
   $$\osc_{Q + re_d} v_k\leq  C_0(2L/r)^{\alpha}(\osc v_k + r^2M_p(f,(k+1/2)r)),$$
  On the other hand $v_k$ is periodic on $\partial P_{e_d} + re_d$ with respect to the translations $(\ell_j)_{j=1}^{d-1}$ and periodicity cell $Q$ so $ \osc_{\partial P_{e_d} + re_d} v_k = \osc_{Q + re_d} v_k$.  Using the inductive hypothesis the right hand side above is bounded by,
  \begin{align*}
   \osc_{\partial P_{e_d} + re_d} v_k &\leq C_0(2L/r)^{\alpha}(C_0^k(2L/r)^{\alpha k }+\sum_{j=0}^{k-1}C_0^{k-j}(2L/r)^{\alpha (k-j) }A(j) +2I_p(f,kr)+ r^2M_{f}((k+1/2)r)) \\
   & = C_0^k(2L/r)^{\alpha (k+1) }+\sum_{j=0}^{k-1}C_0^{k+1-j}(2L/r)^{\alpha (k+1-j) }A(j)+ C_0(2L/r)^{\alpha}A(k) \\
   & = C_0^k(2L/r)^{\alpha (k+1) }+\sum_{j=0}^{k}C_0^{k+1-j}(2L/r)^{\alpha (k+1-j) }A(j).
   \end{align*}
 Finally using Lemma~\ref{lem: f rhs} in place of maximum principle to bound the oscillation in the entire half space,
\begin{align*}
 \osc_{ P_{e_d} +(k+1)re_d} v &= \osc_{ P_{e_d} +re_d} v_k \\
 &\leq \osc_{ \partial P_{e_d} +re_d} v_k + 2{I}_p(f,(k+1)r) \\
 & \leq C_0^k(2L/r)^{\alpha (k+1) }+\sum_{j=0}^{k}C_0^{k+1-j}(2L/r)^{\alpha (k+1-j) }A(j)+ 2I_p(f,(k+1)r),
 \end{align*}
which is the inductive hypothesis for $k+1$.
\end{proof}

\bibliographystyle{plain}
\bibliography{discontinuity_articles}

\begin{thebibliography}{10}

\bibitem{Aleksanyan}
H.~{Aleksanyan}.
\newblock {Regularity of boundary data in periodic homogenization of elliptic
  systems in layered media}.
\newblock {\em ArXiv e-prints}, September 2014.

\bibitem{FA1}
H.~{Aleksanyan}, H.~{Shahgholian}, and P.~{Sj{\"o}lin}.
\newblock {Applications of Fourier analysis in homogenization of Dirichlet
  problem I. Pointwise estimates}.
\newblock {\em Journal of Differential Equations}, 254:2626--2637, 2013.

\bibitem{FA3}
H.~{Aleksanyan}, H.~{Shahgholian}, and P.~{Sj{\"o}lin}.
\newblock {Applications of Fourier Analysis in Homogenization of the Dirichlet
  Problem II. L$^p$ estimates}.
\newblock {\em Archive for Rational Mechanics and Analysis}, August 2014.

\bibitem{FA2}
H.~{Aleksanyan}, P.~{Sj{\"o}lin}, and H.~{Shahgholian}.
\newblock {Applications of Fourier analysis in homogenization of Dirichlet
  problem III. Polygonal estimates}.
\newblock {\em Journal of Fourier Analysis and Applications}, 20(3):524--546.,
  2014.

\bibitem{Arisawa03}
Mariko Arisawa.
\newblock Long time averaged reflection force and homogenization of oscillating
  {N}eumann boundary conditions.
\newblock {\em Ann. Inst. H. Poincar\'e Anal. Non Lin\'eaire}, 20(2):293--332,
  2003.

\bibitem{ArmstrongSmart}
Scott~N. Armstrong and Charles~K. Smart.
\newblock Quantitative stochastic homogenization of elliptic equations in
  nondivergence form.
\newblock {\em Arch. Ration. Mech. Anal.}, 214(3):867--911, 2014.

\bibitem{AL-Lp}
M.~Avellaneda and Fang~Hua Lin.
\newblock Lp bounds on singular integrals in homogenization.
\newblock {\em Communications on Pure and Applied Mathematics},
  44(8-9):897--910, 1991.

\bibitem{AL-div}
Marco Avellaneda and Fang-Hua Lin.
\newblock Compactness methods in the theory of homogenization.
\newblock {\em Communications on Pure and Applied Mathematics}, 40(6):803--847,
  1987.

\bibitem{AvellanedaLin}
Marco Avellaneda and Fang-Hua Lin.
\newblock Compactness methods in the theory of homogenization. {II}.
  {E}quations in nondivergence form.
\newblock {\em Comm. Pure Appl. Math.}, 42(2):139--172, 1989.

\bibitem{BDLS08}
G.~Barles, F.~Da~Lio, P.-L. Lions, and P.~E. Souganidis.
\newblock Ergodic problems and periodic homogenization for fully nonlinear
  equations in half-space type domains with {N}eumann boundary conditions.
\newblock {\em Indiana Univ. Math. J.}, 57(5):2355--2375, 2008.

\bibitem{BarlesMironescu12}
G.~{Barles} and E.~{Mironescu}.
\newblock {On homogenization problems for fully nonlinear equations with
  oscillating Dirichlet boundary conditions}.
\newblock {\em Asymptotic Analysis}, 82(3):187--200, 2013.

\bibitem{BPL78}
Alain Bensoussan, Jacques-Louis Lions, and George Papanicolaou.
\newblock {\em Asymptotic analysis for periodic structures}, volume~5 of {\em
  Studies in Mathematics and its Applications}.
\newblock North-Holland Publishing Co., Amsterdam, 1978.

\bibitem{CC95}
Luis~A. Caffarelli and Xavier Cabr{\'e}.
\newblock {\em Fully nonlinear elliptic equations}, volume~43 of {\em American
  Mathematical Society Colloquium Publications}.
\newblock American Mathematical Society, Providence, RI, 1995.

\bibitem{CaffarelliSouganidis}
Luis~A. Caffarelli and Panagiotis~E. Souganidis.
\newblock Rates of convergence for the homogenization of fully nonlinear
  uniformly elliptic pde in random media.
\newblock {\em Invent. Math.}, 180(2):301--360, 2010.

\bibitem{CKL}
Sunhi Choi, Inwon Kim, and Ki-Ahm Lee.
\newblock Homogenization of {N}eumann boundary data with fully nonlinear
  operator.
\newblock {\em Anal. PDE}, 6(4):951--972, 2013.

\bibitem{ChoiKim13}
Sunhi Choi and Inwon~C. Kim.
\newblock Homogenization for nonlinear pdes in general domains with oscillatory
  neumann boundary data.
\newblock {\em Journal de Math\'ematiques Pures et Appliqu\'ees}, 102(2):419 --
  448, 2014.

\bibitem{CIL92}
Michael~G. Crandall, Hitoshi Ishii, and Pierre-Louis Lions.
\newblock User's guide to viscosity solutions of second order partial
  differential equations.
\newblock {\em Bull. Amer. Math. Soc. (N.S.)}, 27(1):1--67, 1992.

\bibitem{Evans}
Lawrence~C. Evans.
\newblock The perturbed test function method for viscosity solutions of
  nonlinear {PDE}.
\newblock {\em Proc. Roy. Soc. Edinburgh Sect. A}, 111(3-4):359--375, 1989.

\bibitem{Feldman13}
William~M. Feldman.
\newblock Homogenization of the oscillating dirichlet boundary condition in
  general domains.
\newblock {\em Journal de Math\'ematiques Pures et Appliqu\'ees}, 101(5):599 --
  622, 2014.

\bibitem{FeldmanKimSouganidis14}
William~M Feldman, Inwon~C Kim, and Panagiotis~E Souganidis.
\newblock Quantitative homogenization of elliptic partial differential
  equations with random oscillatory boundary data.
\newblock {\em Journal de Math{\'e}matiques Pures et Appliqu{\'e}es}, 2014.

\bibitem{GVM11}
David G{\'e}rard-Varet and Nader Masmoudi.
\newblock Homogenization in polygonal domains.
\newblock {\em J. Eur. Math. Soc. (JEMS)}, 13(5):1477--1503, 2011.

\bibitem{GVM12}
David G{\'e}rard-Varet and Nader Masmoudi.
\newblock Homogenization and boundary layers.
\newblock {\em Acta Mathematica}, 209:133--178, 2012.

\bibitem{GS14}
Nestor Guillen and Russell~W Schwab.
\newblock Neumann homogenization via integro-differential operators.
\newblock {\em arXiv preprint arXiv:1403.1980}, 2014.

\bibitem{Jensen88}
Robert Jensen.
\newblock The maximum principle for viscosity solutions of fully nonlinear
  second order partial differential equations.
\newblock {\em Arch. Rational Mech. Anal.}, 101(1):1--27, 1988.

\bibitem{NadirashviliVladut1}
Nikolai Nadirashvili and Serge Vl{\u{a}}du{\c{t}}.
\newblock Nonclassical solutions of fully nonlinear elliptic equations.
\newblock {\em Geom. Funct. Anal.}, 17(4):1283--1296, 2007.

\bibitem{NadirashviliVladut2}
Nikolai Nadirashvili and Serge Vl{\u{a}}du{\c{t}}.
\newblock Singular viscosity solutions to fully nonlinear elliptic equations.
\newblock {\em J. Math. Pures Appl. (9)}, 89(2):107--113, 2008.

\bibitem{Nirenberg}
Louis Nirenberg.
\newblock On nonlinear elliptic partial differential equations and {H}\"older
  continuity.
\newblock {\em Comm. Pure Appl. Math.}, 6:103--156; addendum, 395, 1953.

\bibitem{Prange}
C.~Prange.
\newblock Asymptotic analysis of boundary layer correctors in periodic
  homogenization.
\newblock {\em SIAM Journal on Mathematical Analysis}, 45(1):345--387, 2013.

\bibitem{T84}
Hiroshi {Tanaka}.
\newblock {Homogenization of diffusion processes with boundary conditions.}
\newblock {Stochastic analysis and applications, Adv. Probab. Relat. Top. 7,
  411-437 (1984).}, 1984.

\bibitem{Tartar}
Luc Tartar.
\newblock {\em The general theory of homogenization}, volume~7 of {\em Lecture
  Notes of the Unione Matematica Italiana}.
\newblock Springer-Verlag, Berlin; UMI, Bologna, 2009.
\newblock A personalized introduction.

\end{thebibliography}

\vspace{.2in}

\end{document}